\documentclass[12pt]{compositio}
\usepackage{amsxtra}
\usepackage{amscd}
\usepackage{eucal}
\usepackage{verbatim}
\usepackage[all]{xy}
\usepackage{color}
\definecolor{darkspringgreen}{rgb}{0.09, 0.45, 0.27}
\usepackage{stmaryrd}
\usepackage[pdftex,bookmarks=false,colorlinks=true,debug=true,citecolor=darkspringgreen, naturalnames=true,pdfnewwindow=true]{hyperref}
\usepackage{microtype} 
\usepackage{amsmath, amsfonts,amssymb,amsthm}

\usepackage[normalem]{ulem}
\newcommand{\stkout}[1]{\ifmmode\text{\sout{\ensuremath{#1}}}\else\sout{#1}\fi}
\usepackage{xcolor,cancel}

\usepackage[all]{xy}
\usepackage{etoolbox}%
\usepackage{enumitem}
\patchcmd{\footnotemark}{\stepcounter{footnote}}{\refstepcounter{footnote}}{}{}
\newtheorem{thm}[subsubsection]{Theorem}
\newtheorem{cor}[subsubsection]{Corollary}

\newtheorem{lem}[subsubsection]{Lemma}
\newtheorem{prop}[subsubsection]{Proposition}
\newtheorem{conj}{Conjecture}
\theoremstyle{definition}
\newtheorem{definition}[subsubsection]{Definition}
\newtheorem{question}{Question}
\theoremstyle{remark}
\newtheorem{rem}[subsubsection]{Remark}

\newcommand{\nc}{\newcommand}
\nc{\renc}{\renewcommand} \nc{\ssec}{\subsection}
\nc{\sssec}{\subsubsection}

\nc{\on}{\operatorname} \nc{\wh}{\widehat}
\nc\ol{\overline} \nc\ul{\underline} \nc\wt{\widetilde}

\nc{\BA}{{\mathbb{A}}} \nc{\BC}{{\mathbb{C}}} \nc{\BF}{{\mathbb{F}}}
\nc{\BD}{{\mathbb{D}}} \nc{\BG}{{\mathbb{G}}} \nc{\BQ}{{\mathbb{Q}}}
\nc{\BM}{{\mathbb{M}}} \nc{\BN}{{\mathbb{N}}} \nc{\BO}{{\mathbb{\bfO}}}
\nc{\BP}{{\mathbb{P}}} \nc{\BR}{{\mathbb{R}}}
\nc{\BZ}{{\mathbb{Z}}} \nc{\BS}{{\mathbb{S}}} \nc{\BW}{{\mathbb{W}}}

\nc{\CA}{{\mathcal{A}}} \nc{\CB}{{\mathcal{B}}} \nc{\CalD}{{\mathcal{D}}}
\nc{\CE}{{\mathcal{E}}} \nc{\CF}{{\mathcal{F}}}
\nc{\CG}{{\mathcal{G}}} \nc{\CH}{{\mathcal{H}}}
\nc{\CI}{{\mathcal{I}}} \nc{\CK}{{\mathcal{K}}} \nc{\CL}{{\mathcal{L}}}
\nc{\CM}{{\mathcal{M}}} \nc{\CN}{{\mathcal{N}}}
\nc{\CO}{{\mathcal{\bfO}}} \nc{\CP}{{\mathcal{P}}}
\nc{\CQ}{{\mathcal{Q}}} \nc{\CR}{{\mathcal{R}}}
\nc{\CS}{{\mathcal{S}}} \nc{\CT}{{\mathcal{T}}}
\nc{\CU}{{\mathcal{U}}} \nc{\CV}{{\mathcal{V}}}  \nc{\CY}{{\mathcal Y}}
\nc{\CW}{{\mathcal{W}}} \nc{\CZ}{{\mathcal{Z}}}

\nc{\cM}{{\check{\mathcal M}}{}} \nc{\csM}{{\check{\mathcal A}}{}}
\nc{\oM}{{\overset{\circ}{\mathcal M}}{}}
\nc{\obM}{{\overset{\circ}{\mathbf M}}{}}
\nc{\oCA}{{\overset{\circ}{\mathcal A}}{}}
\nc{\obA}{{\overset{\circ}{\mathbf A}}{}}
\nc{\ooM}{{\overset{\circ}{M}}{}}
\nc{\osM}{{\overset{\circ}{\mathsf M}}{}}
\nc{\vM}{{\overset{\bullet}{\mathcal M}}{}}
\nc{\nM}{{\underset{\bullet}{\mathcal M}}{}}
\nc{\oD}{{\overset{\circ}{\mathcal D}}{}}
\nc{\obD}{{\overset{\circ}{\mathbf D}}{}}
\nc{\oA}{{\overset{\circ}{\mathbb A}}{}}
\nc{\op}{{\overset{\bullet}{\mathbf p}}{}}
\nc{\cp}{{\overset{\circ}{\mathbf p}}{}}
\nc{\oU}{{\overset{\bullet}{\mathcal U}}{}}
\nc{\ofZ}{{\overset{\circ}{\mathfrak Z}}{}}

\nc{\ff}{{\mathfrak{f}}} \nc{\fv}{{\mathfrak{v}}}
\nc{\fa}{{\mathfrak{a}}} \nc{\fb}{{\mathfrak{b}}}
\nc{\fd}{{\mathfrak{d}}} \nc{\fe}{{\mathfrak{e}}}
\nc{\fg}{{\mathfrak{g}}} \nc{\fgl}{{\mathfrak{gl}}}
\nc{\fh}{{\mathfrak{h}}} \nc{\fri}{{\mathfrak{i}}}
\nc{\fj}{{\mathfrak{j}}} \nc{\fk}{{\mathfrak{k}}} \nc{\fl}{{\mathfrak{l}}}
\nc{\fm}{{\mathfrak{m}}} \nc{\fn}{{\mathfrak{n}}}
\nc{\ft}{{\mathfrak{t}}} \nc{\fu}{{\mathfrak{u}}}
\nc{\fw}{{\mathfrak{w}}} \nc{\fz}{{\mathfrak{z}}}
\nc{\fp}{{\mathfrak{p}}} \nc{\fq}{{\mathfrak{q}}} \nc{\frr}{{\mathfrak{r}}}
\nc{\fs}{{\mathfrak{s}}} \nc{\fsl}{{\mathfrak{sl}}}
\nc{\hsl}{{\widehat{\mathfrak{sl}}}}
\nc{\hgl}{{\widehat{\mathfrak{gl}}}}
\nc{\hg}{{\widehat{\mathfrak{g}}}}
\nc{\chg}{{\widehat{\mathfrak{g}}}{}^\vee}
\nc{\hn}{{\widehat{\mathfrak{n}}}}
\nc{\chn}{{\widehat{\mathfrak{n}}}{}^\vee}

\nc{\fA}{{\mathfrak{A}}} \nc{\fB}{{\mathfrak{B}}} \nc{\fC}{{\mathfrak{C}}}
\nc{\fD}{{\mathfrak{D}}} \nc{\fE}{{\mathfrak{E}}}
\nc{\fF}{{\mathfrak{F}}} \nc{\fG}{{\mathfrak{G}}} \nc{\fH}{{\mathfrak{H}}}
\nc{\fI}{{\mathfrak{I}}} \nc{\fJ}{{\mathfrak{J}}}
\nc{\fK}{{\mathfrak{K}}} \nc{\fL}{{\mathfrak{L}}}
\nc{\fM}{{\mathfrak{M}}} \nc{\fN}{{\mathfrak{N}}}
\nc{\frP}{{\mathfrak{P}}} \nc{\fQ}{{\mathfrak{Q}}}
\nc{\fS}{{\mathfrak{S}}} \nc{\fT}{{\mathfrak{T}}} \nc{\fU}{{\mathfrak{U}}}
\nc{\fV}{{\mathfrak{V}}} \nc{\fW}{{\mathfrak{W}}}
\nc{\fX}{{\mathfrak{X}}} \nc{\fY}{{\mathfrak{Y}}}
\nc{\fZ}{{\mathfrak{Z}}}

\nc{\ba}{{\mathbf{a}}}
\nc{\bb}{{\mathbf{b}}} \nc{\bc}{{\mathbf{c}}}
\nc{\be}{{\mathbf{e}}} \nc{\bj}{{\mathbf{j}}} \nc{\bm}{{\mathbf{m}}}
\nc{\bn}{{\mathbf{n}}} \nc{\bp}{{\mathbf{p}}}
\nc{\bq}{{\mathbf{q}}} \nc{\br}{{\mathbf{r}}} \nc{\bt}{{\mathbf{t}}}
\nc{\bfu}{{\mathbf{u}}} \nc{\bv}{{\mathbf{v}}}
\nc{\bx}{{\mathbf{x}}} \nc{\by}{{\mathbf{y}}} \nc{\bz}{{\mathbf{z}}}
\nc{\bw}{{\mathbf{w}}} \nc{\bA}{{\mathbf{A}}}
\nc{\bB}{{\mathbf{B}}} \nc{\bC}{{\mathbf{C}}}
\nc{\bD}{{\mathbf{D}}} \nc{\bF}{{\mathbf{F}}} \nc{\bG}{{\mathbf{G}}}
\nc{\bH}{{\mathbf{H}}} \nc{\bI}{{\mathbf{I}}} \nc{\bJ}{{\mathbf{J}}}
\nc{\bK}{{\mathbf{K}}} \nc{\bM}{{\mathbf{M}}} \nc{\bN}{{\mathbf{N}}}
\nc{\bO}{{\mathbf{\bfO}}} \nc{\bS}{{\mathbf{S}}} \nc{\bT}{{\mathbf{T}}}
\nc{\bU}{{\mathbf{U}}} \nc{\bV}{{\mathbf{V}}} \nc{\bW}{{\mathbf{W}}}
\nc{\bX}{{\mathbf{X}}}
\nc{\bY}{{\mathbf{Y}}} \nc{\bP}{{\mathbf{P}}}
\nc{\bZ}{{\mathbf{Z}}} \nc{\bh}{{\mathbf{h}}}

\nc{\sA}{{\mathsf{A}}} \nc{\sB}{{\mathsf{B}}}
\nc{\sC}{{\mathsf{C}}} \nc{\sD}{{\mathsf{D}}}
\nc{\sE}{{\mathsf{E}}} \nc{\sF}{{\mathsf{F}}} \nc{\sG}{{\mathsf{G}}}
\nc{\sI}{{\mathsf{I}}} \nc{\sK}{{\mathsf{K}}} \nc{\sL}{{\mathsf{L}}}
\nc{\sfm}{{\mathsf{m}}} \nc{\sM}{{\mathsf{M}}} \nc{\sO}{{\mathsf{\bfO}}}
\nc{\sQ}{{\mathsf{Q}}} \nc{\sP}{{\mathsf{P}}}
\nc{\sT}{{\mathsf{T}}} \nc{\sZ}{{\mathsf{Z}}}
\nc{\sV}{{\mathsf{V}}} \nc{\sW}{{\mathsf{W}}}
\nc{\sfp}{{\mathsf{p}}} \nc{\sq}{{\mathsf{q}}} \nc{\sr}{{\mathsf{r}}}
 \nc{\sfb}{{\mathsf{b}}}
\nc{\sfc}{{\mathsf{c}}} \nc{\sd}{{\mathsf{d}}}
\nc{\sz}{{\mathsf{z}}}

\nc{\tA}{{\widetilde{\mathbf{A}}}}
\nc{\tB}{{\widetilde{\mathcal{B}}}}
\nc{\tg}{{\widetilde{\mathfrak{g}}}} \nc{\tG}{{\widetilde{G}}}
\nc{\TM}{{\widetilde{\mathbb{M}}}{}}
\nc{\tO}{{\widetilde{\mathsf{\bfO}}}{}}
\nc{\tU}{{\widetilde{\mathfrak{U}}}{}} \nc{\TZ}{{\tilde{Z}}}
\nc{\tx}{{\tilde{x}}} \nc{\tbv}{{\tilde{\bv}}}
\nc{\tfP}{{\widetilde{\mathfrak{P}}}{}} \nc{\tz}{{\tilde{\zeta}}}
\nc{\tmu}{{\tilde{\mu}}}

\nc{\urho}{\underline{\rho}} \nc{\uB}{\underline{B}}
\nc{\uC}{{\underline{\mathbb{C}}}} \nc{\ui}{\underline{i}}
\nc{\uj}{\underline{j}} \nc{\ofP}{{\overline{\mathfrak{P}}}}
\nc{\oB}{{\overline{\mathcal{B}}}}
\nc{\og}{{\overline{\mathfrak{g}}}} \nc{\oI}{{\overline{I}}}

\nc{\eps}{\varepsilon} \nc{\hrho}{{\hat{\rho}}}
\nc{\blambda}{{\boldsymbol{\lambda}}} \nc{\bmu}{{\boldsymbol{\mu}}} \nc{\bnu}{{\boldsymbol{\nu}}}

\nc{\one}{{\mathbf{1}}} \nc{\two}{{\mathbf{t}}}

\nc{\Sym}{\mathop{\operatorname{\rm Sym}}}
\nc{\SI}{{\rm{SI}}}
\nc{\Tot}{{\mathop{\operatorname{\rm Tot}}}}
\nc{\Spec}{\mathop{\operatorname{\rm Spec}}}
\nc{\Ker}{{\mathop{\operatorname{\rm Ker}}}}
\nc{\Isom}{{\mathop{\operatorname{\rm Isom}}}}
\nc{\Hilb}{{\mathop{\operatorname{\rm Hilb}}}}
\nc{\deeq}{{\mathop{\operatorname{\rm deeq}}}}
\nc{\End}{{\mathop{\operatorname{\rm End}}}}
\nc{\Ext}{{\mathop{\operatorname{\rm Ext}}}}
\nc{\Hom}{{\mathop{\operatorname{\rm Hom}}}}
\nc{\CHom}{{\mathop{\operatorname{{\mathcal{H}}\it om}}}}
\nc{\GL}{{\mathop{\operatorname{\rm GL}}}}
\nc{\Stab}{{\mathop{\operatorname{\rm Stab}}}}
\nc{\Mir}{{\mathop{\operatorname{\rm Mir}}}}
\nc{\gr}{{\mathop{\operatorname{\rm gr}}}}
\nc{\Id}{{\mathop{\operatorname{\rm Id}}}}
\nc{\perf}{{\mathop{\operatorname{\rm perf}}}}
\nc{\ver}{{\mathop{\operatorname{\rm ver}}}}
\nc{\verdier}{{\mathop{\operatorname{\rm Verdier}}}}
\nc{\glob}{{\mathop{\operatorname{\rm glob}}}}
\nc{\ratio}{{\mathop{\operatorname{\rm ratio}}}}
\nc{\Ran}{{\mathop{\operatorname{\rm Ran}}}}
\nc{\Conf}{{\mathop{\operatorname{\rm Conf}}}}
\nc{\mix}{{\mathop{\operatorname{\rm mxd}}}}
\nc{\sm}{{\mathop{\operatorname{\rm small}}}}
\nc{\defi}{{\mathop{\operatorname{\rm def}}}}
\nc{\length}{{\mathop{\operatorname{\rm length}}}}
\nc{\supp}{{\mathop{\operatorname{\rm supp}}}}
\nc{\Fun}{{\mathop{\operatorname{\rm Funct}}}}
\nc{\Hei}{{\mathop{\operatorname{\rm Heis}}}}
\nc{\HC}{{\mathcal H}{\mathcal C}}

\nc{\Cliff}{{\mathsf{Cliff}}}
\nc{\loc}{{\operatorname{loc}}}
\nc{\Fl}{{\operatorname{Fl}}} 
\nc{\Ffl}{{\mathcal{F}\ell}}
\nc{\Fib}{{\mathsf{Fib}}}
\nc{\Coh}{{\operatorname{Coh}}} \nc{\FCoh}{{\mathsf{FCoh}}}
\nc{\Perf}{{\mathsf{Perf}}}

\nc{\reg}{{\text{\rm reg}}}
\nc{\gvee}{{\mathfrak g}^{\!\scriptscriptstyle\vee}}
\nc{\tvee}{{\mathfrak t}^{\!\scriptscriptstyle\vee}}
\nc{\nvee}{{\mathfrak n}^{\!\scriptscriptstyle\vee}}
\nc{\bvee}{{\mathfrak b}^{\!\scriptscriptstyle\vee}}
       \nc{\rhovee}{\rh\bfO^{\!\scriptscriptstyle\vee}}

\nc{\cplus}{{\mathbf{C}_+}} \nc{\cminus}{{\mathbf{C}_-}}
\nc{\cthree}{{\mathbf{C}_*}} \nc{\Qbar}{{\bar{Q}}}

\nc{\Gtimes}{\vphantom{j^{X^2}}\smash{\overset{G}{\vphantom{\rule{0pt}{0.30em}}\smash{\times}}}}
\nc{\sGtimes}{\vphantom{j^{X^2}}\smash{\overset{\mathsf G}{\vphantom{\rule{0pt}{0.30em}}\smash{\times}}}}

\nc{\bOmega}{{\overline{\Omega}}}

\nc{\seq}[1]{\stackrel{#1}{\sim}}

\nc{\aff}{{\operatorname{aff}}}
\nc{\fin}{{\operatorname{fin}}}
\nc{\mir}{{\operatorname{mir}}}
\nc{\triv}{{\operatorname{triv}}}
\nc{\ext}{{\operatorname{ext}}}
\nc{\righ}{{\operatorname{right}}}
\nc{\lef}{{\operatorname{left}}}
\nc{\forg}{{\operatorname{forg}}}
\nc{\fid}{{\operatorname{fd}}}
\nc{\modu}{{\operatorname{-mod}}}
\nc{\Gr}{{\operatorname{Gr}}}
\nc{\FT}{{\operatorname{FT}}}
\nc{\Mat}{{\operatorname{Mat}}}
\nc{\MSt}{{\operatorname{MSt}}}
\nc{\sph}{{\operatorname{sph}}}
\nc{\GR}{{\mathbf{Gr}}}
\nc{\Perv}{{\operatorname{Perv}}}
\nc{\Rep}{{\operatorname{Rep}}}
\nc{\Fact}{{\operatorname{FactMod}}}
\nc{\Ind}{{\operatorname{Ind}}}
\nc{\IC}{{\operatorname{IC}}}
\nc{\Bun}{{\operatorname{Bun}}}
\nc{\Br}{{\operatorname{Br}}}
\nc{\mult}{{\operatorname{mult}}}
\nc{\Proj}{{\operatorname{Proj}}}
\nc{\pt}{{\operatorname{pt}}}
\nc{\bfmu}{{\boldsymbol{\mu}}}
\nc{\bfomega}{{\boldsymbol{\omega}}}
\nc{\calM}{\mathcal M}
\nc{\calA}{\mathcal A}
\nc{\calO}{\mathcal O}
\nc{\cC}{\mathcal C}
\nc{\CC}{\mathbb C}
\nc{\calN}{\mathcal N}
\nc{\grg}{\mathfrak g}
\nc{\tslash}{/\!\!/\!\!/}
\nc\grt{\mathfrak t}
\nc\bfM{\mathbf M}
\nc\bfN{\mathbf N}

\nc\ZZ{\mathbb{Z}}
\nc\calC{\mathcal C}
\nc\calF{\mathcal F}
\nc\calX{\mathcal X}
\nc\calY{\mathcal Y}
\nc\Char{\operatorname{char}}
\nc\QCoh{\operatorname{QCoh}}
\nc\ShvD{\operatorname{D-mod}}
\nc\Shv{\operatorname{Shv}}
\nc\IndCoh{\operatorname{IndCoh}}
\nc\Maps{\operatorname{Maps}}
\nc\Dmod{D-\operatorname{mod}}
\newcommand\Hecke{\operatorname{Hecke}}
\nc{\calD}{\mathcal D}
\nc\bfO{\mathbf O}
\nc\bfF{\mathbf F}

\nc\GG{\mathbb G}
\nc\calK{\mathcal K}
\nc{\calG}{\mathcal G}
\nc\RHom{\operatorname{RHom}}
\nc\Res{\operatorname{Res}}
\nc\Av{\operatorname{Av}}
\nc\oblv{\operatorname{oblv}}
\nc\pr{\operatorname{pr}}
\nc\unit{\operatorname{unit}}
\nc\add{\operatorname{add}}
\nc\ind{\operatorname{ind}}
\nc\coind{\operatorname{coind}}
\nc\sprd{\operatorname{sprd}}
\nc\projection{\operatorname{projection}}
\nc\averaging{\operatorname{averaging}}
\nc\pullback{\operatorname{pullback}}
\nc\grs{\mathfrak s}
\nc{\tilX}{\widetilde X}
\nc\calB{\mathcal B}
\nc\calS{\mathcal S}
\nc\calT{\mathcal T}
\nc\calZ{\mathcal Z}
\nc\LS{\operatorname{LocSys}}
\nc\bfL{\on{\mathbf L}}

\newcommand*\circled[1]
{*{#1}}
\makeatletter
\newcommand{\raisemath}[1]{\mathpalette{\raisem@th{#1}}}
\newcommand{\raisem@th}[3]{\raisebox{#1}{$#2#3$}}
\nc{\binlim}[2][]{\def\@tempa{#1}\@ifnextchar^{\@binlim{#2}}{\@binlim{#2}^{}}}
\def\@binlim#1^#2{\mathbin{\@ifempty{#2}{\mathop{#1}}{\mathop{#1}\@xp\displaylimits\@tempa^{#2}}}}

\makeatother

\nc\cX{{\mathcal X}}

\nc\Gm{{\mathbb G_m}}
\nc{\isoto}
    {\stackrel{\displaystyle\raisemath{-.2ex}{\sim}}\to}
\nc\cD\CalD
\nc\setm\setminus
\nc\cE\CE
\renc\Hecke{\mathit{\CH\kern-.2ex ecke}}
\nc\bsl\backslash
\nc\Fq{\mathbb F_q}
\nc\x\times
\nc\bGO{{\bG_\bO}}
\nc\bla\blambda
\nc\blt\bullet
\nc\opp{{\on{op}}}
\nc\srel\stackrel
\nc\tbx{\binlim{\widetilde\boxtimes{}}}
\nc\into\hookrightarrow
\nc\otto\leftrightarrow
\renc\setminus\smallsetminus
\nc\phitau{\varphi\tau}
\nc\ceil[1]{\lceil#1\rceil}  \nc\floor[1]{\lfloor#1\rfloor}
\nc\Lie{\on{Lie}}
\def\arxiv#1{\href{http://arxiv.org/abs/#1}{\tt arXiv:#1}} \let\arXiv\arxiv
\nc\kap{\kappa}
\nc\gra{\mathfrak a}
\nc\gl{\mathfrak{gl}}
\nc\sTr{\operatorname{sTr}}
\nc\hatG{\widehat{G}}
\nc\calL{\mathcal L}
\nc\Whit{\operatorname{Whit}}
\nc\KM{\operatorname{KM}}
\nc\KL{\operatorname{KL}}

\makeatletter
\renewcommand{\subsection}{\@startsection{subsection}{2}{0pt}{-3ex
plus -1ex minus -0.2ex}{-2mm plus -0pt minus
-2pt}{\normalfont\bfseries}} \makeatother

\allowdisplaybreaks
\nc\mto{\mapsto }
\nc\en{\enspace }

\numberwithin{equation}{subsection}
\newtheorem*{rep@theorem}{\rep@title}
\newcommand{\newreptheorem}[2]{%
\newenvironment{rep#1}[1]{%
 \def\rep@title{#2 \ref{##1}}%
 \begin{rep@theorem}}%
 {\end{rep@theorem}}}
\makeatother
\usepackage{tikz-cd}

\usepackage{wasysym, stackengine, makebox, graphicx}

 \newcommand{\ncmd}{\newcommand*}
\newcommand{\rncmd}{\renewcommand*}
\ncmd{\DMO}{\DeclareMathOperator}
\ncmd{\ncmdd}[2]{\ncmd{#1}{{#2}}}

\makeatletter
\ncmd{\DefOps}[1]{\def\OPERATOR@NAME##1{\DeclareMathOperator{##1}{\expandafter\@gobble\string##1}}
    \def\OPERATOR@LIST##1{\ifcat\noexpand\relax\noexpand##1\OPERATOR@NAME##1\expandafter\OPERATOR@LIST\fi}
    \OPERATOR@LIST#1.}

\ncmd{\DefRm}[1]{\def\OPERATOR@NAME##1{\ncmd{##1}{\mathrm{\expandafter\@gobble\string##1}}}
    \def\OPERATOR@LIST##1{\ifcat\noexpand\relax\noexpand##1\OPERATOR@NAME##1\expandafter\OPERATOR@LIST\fi}
    \OPERATOR@LIST#1.}

\ncmd{\phtr}[2]{\lefteqn{#1{\phantom{#2}}}#2}

\def\Alphabet{ABCDEFGHIJKLMNOPQRSTUVWXYZ}
\def\newalph#1#2{\begingroup
    \def\procL@tt@r##1{%
        \@xp\gdef\csname#1\endcsname{#2}}%
    \proc@lph@bet\endgroup}
\def\proc@lph@bet{\@xp\prlist@\Alphabet\relax}
\def\prlist@#1{\ifx#1\relax\else\procL@tt@r{#1}\@xp\prlist@\fi}

\ncmd{\SmSub}[2][]{_\bgroup #2\smsub@{#1}}
\def\smsub@#1{\@ifnextchar_{#1\@smsb}{\egroup}}
\def\@smsb_#1{#1\smsub@{}}

\ncmd{\SmSup}[2][]{^\bgroup #2\smsup@{#1}}
\def\smsup@#1{\@ifnextchar^{#1\@smsp}{\@ifnextchar'{\prime\@xp\smsup@\@xp{\@xp}\@gobble}{\egroup}}}
\def\@smsp^#1{#1\smsup@{}}

\def\@binlim#1_#2{\mathbin{\@ifempty{#2}{\mathop{#1}}{\mathop{#1}\@xp\displaylimits\@tempa_{#2}}}}

\catcode`\=\active
\ncmd\Text[1]{\ifmmode\text{#1}\else#1\fi}
\ncmd\todo[1][todo]{\RedNote{#1}%
 \ifx\undefined\@todoflag
   \global\let\@todoflag\relax
  \AtEndDocument{\par\vspace{3em}\noindent \RedNote{TO DO:}}%
 \fi
 \edef\pagenmbr{\thepage}%
 \expandafter\redendnote
 \expandafter{\pagenmbr}{#1}}
\ncmd\redendnote[2]{\AtEndDocument{\\$\bullet$\ \RedNote{page #1: #2}}}
\def{\todo}
\ncmd\hidetodos{\rncmd\todo\relax}

\theoremstyle{remark}
\newalph{c#1}{\ensuremath{{\mathcal #1}}}
\newalph{s#1}{{\mathsf #1}}
\newalph{#1#1}{{\mathbb #1}}
\newalph{e#1}{{\EuScript #1}}
\newalph{g#1}{{\mathfrak #1}}
\newalph{r#1}{{\ensuremath{\mathscr #1}}}
   
\ncmd{\Fp}{{\FF_p}}

\DMO{\cHom}{\text{\textrm{\itshape{\cH}\kern-.2ex{}om}}}
\DMO{\cEnd}{\text{\textrm{\itshape{\cE}\kern-.2ex{}nd}}}    \DMO{\cExt}{\text{\textrm{\itshape{\cE}\kern-.2ex{}xt}}}

\ncmd{\sff}{\mathsf f}

\let\Im\undefined   \let\det\undefined
\ncmd{\young}[1]{\vcenter{\begin{Young}#1\crcr\end{Young}}}
\DefOps   { \REnd  \Aut \colim
             \det \tr \rk \curv  \Coker \Im  \ad \Ad }
\DefRm{ \id \Vect \Vac \Pic \Loc \Hitch \Higgs \Eu \Fr}

\ncmd{\RGam}{\text{\upshape R}\Gamma}
\DMO{\chr}{char}
\ncmd{\angs}[1]{\langle#1\rangle}
\ncmd{\hmod}{\text{\upshape-mod}}
\ncmd{\cxym}[1]{\ensuremath{\vcenter{\xymatrix{#1}}}}
\def\arxiv#1{\href{http://arxiv.org/abs/#1}{\tt arXiv:#1}} \let\arXiv\arxiv
\makeatother

\title[Twisted Whittaker sheaves on affine flags]{Twisted Whittaker category on affine flags and the category of representations of the mixed quantum group}
\date{}
\author{Ruotao Yang}
\email{yruotao@gmail.com}
\address{Skolkovo Institute of Science and Technology, Moscow, Russia}
\begin{document}

\begin{abstract}
Let $G$ be a reductive group, and let $\check{G}$ be its Langlands dual group. S. Arkhipov and R. Bezrukavnikov prove that the Whittaker category on the affine flags $\Fl_G$ is equivalent to the category of $\check{G}$-equivariant quasi-coherent sheaves on the Springer resolution of the nilpotent cone. This paper proves this theorem in the quantum case. We show that the twisted Whittaker category on $\Fl_G$ and the category of representations of the mixed quantum group are equivalent. In particular, we prove that the quantum category $\mathsf{O}$ is equivalent to the twisted Whittaker category on $\Fl_G$ in the generic case. The strong version of our main theorem claims a motivic equivalence between the Whittaker category on $\Fl_G$ and a factorization module category, which holds in the de Rham setting, the Betti setting, and the $\ell$-adic setting.
\end{abstract}

\maketitle
\setcounter{tocdepth}{1}
\tableofcontents


\section{Introduction}\label{introduction}

\subsection{Reminder on work of \cite{[AB]}}\label{work of ab}
\subsubsection{}
Let $G$ be a reductive group over an algebraically closed field $\mathsf{k}$. Fix a pair ($B$, $B^-$) of opposite Borel subgroups, and denote by $N$ and $N^-$ their unipotent radicals, respectively. Denote by $\mathcal{K}=\mathsf{k}(\!(t)\!)$ the field of Laurent series and by $\mathcal{O}=\mathsf{k}[\![t]\!]$ the ring of formal power series. We denote by $G(\mathcal{K})$ the loop group of $G$, by $I$ the Iwahori subgroup, and by $\Fl_G:= G(\mathcal{K})/I$ the affine flags. In this section, we assume $\mathsf{k}=\mathbb{C}$.


It is known that the category of D-modules on $\Fl_G$, with certain equivariance properties, can be realized in terms of the category of representations of the Langlands dual group $\check{G}$. 

An important equivariance condition is the Whittaker condition. Denote by $N(\mathcal{K})$ the loop group of $N$. We refer to the category of Whittaker D-modules on $\Fl_G$ as the (DG) category of D-modules on $\Fl_G$ which are $N(\mathcal{K})$-equivariant against a non-degenerate character $\chi$. We denote it by $\Whit(\ShvD(\Fl_G))$.
A well-known result of S. Arkhipov and R. Bezrukavnikov (ref. \cite{[AB]}) says that there is an equivalence of categories
\begin{equation}\label{1.21} \Whit(\ShvD(\Fl_G)) \simeq \QCoh(\tilde{\mathcal{N}}/\check{G}).
\end{equation}
In the above formula, $\tilde{\mathcal{N}}:= T^*(\check{G}/\check{B})$ is the Springer resolution of the nilpotent cone, and $\QCoh(\tilde{\mathcal{N}}/\check{G})$ is the (DG) category of $\check{G}$-equivariant quasi-coherent sheaves on $\tilde{\mathcal{N}}$.

It is natural to consider the following question.
\begin{question}
What is the deformed version of \eqref{1.21}?
\end{question}
\subsubsection{}
It is expected that \eqref{1.21} deforms over the space of levels, i.e., the space of Weyl group-invariant symmetric bilinear forms on $\Lambda$.  Here $\Lambda$ is the coweight lattice of $G$.

The left-hand side of \eqref{1.21} admits a naturally defined level-parameterized deformation. Namely, a level $\kappa$ gives rise to a twisting, and we can consider $\kappa$-twisted D-modules on $\Fl_G$. Then $\Whit(\ShvD_{\kappa}(\Fl_G))$, the category of $\kappa$-twisted Whittaker D-modules on $\Fl_G$, is defined as the category of $(N(\mathcal{K}),\chi)$-equivariant $\kappa$-twisted D-modules on $\Fl_G$. 

The deformation of the right-hand side of \eqref{1.21} is not obvious. In order to present it, we rewrite $\tilde{\mathcal{N}}/\check{G}$ as $\check{\mathfrak{n}}/\check{B}$. Here $\check{\mathfrak{n}}$ is the Lie algebra of the unipotent radical $\check{N}$ of $\check{B}$, and the action of $\check{B}$ on $\check{\mathfrak{n}}$ is the adjoint action.

A quasi-coherent sheaf on  $\check{\mathfrak{n}}/\check{B}$ is a $\mathcal{O}(\check{\mathfrak{n}})$-module with a compatible action of $\check{B}$. It can be regarded as a $\Lambda$-graded vector space with compatible actions of $\Sym(\check{\mathfrak{n}}^-)$ and $U(\check{\mathfrak{n}})$, with the locally nilpotent condition. The universal enveloping algebra $U(\check{\mathfrak{n}})$ naturally deforms. Namely, the Lusztig quantum group $U^L_q(\check{\mathfrak{n}})$ provides a deformation of $U(\check{\mathfrak{n}})$ over the space of levels. Here the relation between the level $\kappa$ and the quantum parameter $$q: \Lambda\to \mathbb{C}/\mathbb{Z}=\mathbb{C}^\times$$ is given by $q(\lambda)=  \exp(\pi i\cdot\kappa(\lambda,\lambda))$ for $\lambda\in \Lambda$. Note that $\Sym(\check{\mathfrak{n}}^-)$ is the graded dual of $U(\check{\mathfrak{n}})$, so the graded dual of $U^L_q(\check{\mathfrak{n}})$ provides a deformation of $\Sym(\check{\mathfrak{n}}^-)$ over the space of levels.

From this point of view, the right-hand side of \eqref{1.21} has a deformation. It is given by the category of representations of a certain quantum group, whose positive part is the Lusztig quantum group $U^L_q(\check{\mathfrak{n}})$ and the negative part is the De Concini-Kac quantum group $U^{DK}_q(\check{\mathfrak{n}}^-)$ (i.e., the graded dual of the Lusztig quantum group). It is exactly the category of representations of the \textit{mixed} quantum group introduced by D. Gaitsgory in \cite[Section 5.3]{[Ga1]}, which is denoted by $\Rep^{\mix}_q(\check{G})$.

Now it is natural to ask the following question.
\begin{question}
Is there an equivalence
\begin{equation}\label{ques 2}
    \Whit(\ShvD_{\kappa}(\Fl_G))\simeq \Rep^{\mix}_q(\check{G})?
\end{equation}
\end{question}

\subsection{Main theorem of this paper}
The weak version of the main theorem (Theorem \ref{thm 2.1}) says that, when $q$ avoids small torsion (see Section \ref{small torsion}), there exists a t-exact equivalence of categories between $\Whit(\ShvD_\kappa(\Fl_G))$ and $\Rep_q^{\mix}(\check{G})$. Furthermore, both $\Rep_q^{\mix}(\check{G})$ and $\Whit(\ShvD_\kappa(\Fl_G))$ acquire structures of highest weight categories,\footnotemark\ and the equivalence functor preserves highest weight category structure.\footnotetext{For a highest weight category structure of a DG-category $\mathcal{C}$, we mean a collection of standard objects \{$c_i$\} and a collection of costandard objects \{$c^\vee_i$\} with the orthogonality property. Furthermore, we require that \{$c_i$\} generate $\mathcal{C}$ by colimits, shifts, and extensions.}

However, Theorem \ref{strong thm}, the strong version of the main theorem, proves a more general statement, where we do not need to assume that the base field of schemes is $\mathbb{C}$ and the sheaf category is the category of D-modules. 

Let $\mathsf{k}$ denote the base field of schemes, and let $\mathsf{e}$ denote the coefficients field of sheaves. Note that \eqref{ques 2} only makes sense when $\mathsf{k}=\mathsf{e}=\mathbb{C}$, otherwise we cannot define D-modules and the mixed quantum group simultaneously. In this case, $\Rep_q^\mix(\check{G})$ can be realized as the category of \textit{factorization modules} (ref. \cite{[Ga6]}) over a \textit{factorization algebra} $\Omega_q^{{L,'}}$. Instead of working with $\Rep_q^\mix(\check{G})$, we will compare $\Whit(\ShvD_\kappa(\Fl_G))$ with this factorization module category.

The advantage of using factorization modules lies in the fact that the statement involving factorization modules is geometric (i.e., motivic). Rather than D-modules, we can also consider factorization modules in the settings of the $\ell$-adic sheaves and the constructible sheaves with arbitrary coefficients $\mathsf{e}$. 

Theorem \ref{strong thm} claims that there is a t-exact equivalence of highest weight categories
\begin{equation}
    \Whit_q(\Fl_G)\simeq \Omega_q^{L,'}-\Fact.
\end{equation}
In the above formula, 
\begin{itemize}
    \item $\Whit_q(\Fl_G)$ is the category of twisted Whittaker sheaves on $\Fl_G$,
    \item $\Omega_q^{L,'}-\Fact$ is the category of factorization modules over $\Omega_q^{L,'}$.
\end{itemize}
This theorem holds in a greater generality: besides D-modules, it is true for all the sheaves contexts listed in Section \ref{sheaf context} and any $q$. 

\subsection{Other Motivations}
In this section, we will provide motivations for Theorem \ref{thm 2.1} other than those coming from the work of \cite{[AB]}. We assume $\mathsf{k}=\mathsf{e}=\mathbb{C}$ in this section.
\subsubsection{Fundamental local equivalence}
Another main idea that motivates this work comes from the quantum local Langlands conjecture. We will explain it in this section.

In \cite{Ga13}, D. Gaitsgory proposed a very general conjecture of the quantum Langlands program. 

Consider the category $\ShvD_\kappa(G(\mathcal{K}))$ of $\kappa$-twisted D-modules on the loop group $G(\mathcal{K})$. The group structure on $G(\mathcal{K})$ induces a monoidal structure on this category. We denote by $G(\mathcal{K})\modu_\kappa$ the 2-category of module categories over this monoidal category. The quantum local Langlands conjecture asserts the following equivalences.

\begin{conj}\label{quantum local Langlands}\
\begin{enumerate}[label={\upshape(\arabic*)}]
    \item There is an equivalence of categories,
 \begin{equation}\label{quantum local}
    G(\mathcal{K})\modu_{-\kappa}\simeq \check{G}(\mathcal{K})\modu_{\check{\kappa}}.
\end{equation}
Here  $\check{\kappa}$ denotes the dual level of $\kappa$ (ref. \cite[0.1.1]{[GL1]}).
\item  If $\mathcal{C}\in G(\mathcal{K})\modu_{-\kappa}$ goes to $\check{\mathcal{C}}\in \check{G}(\mathcal{K})\modu_{\check{\kappa}}$ under the equivalence \eqref{quantum local}, then their Iwahori strong invariants (ref. \cite[Section 4]{[Ber]}) are equivalent
\begin{equation}\label{1.4+}
   \mathcal{C}^{I}\simeq \check{\mathcal{C}}^{\check{I}}.
\end{equation}

Namely, the category of $\ShvD_{-\kappa}(G(\mathcal{K}))$-equivariant functors from $\ShvD_{-\kappa}(\Fl_G))$ to $\mathcal{C}$, is equivalent to the category of $\ShvD_{\check{\kappa}}(\check{G}(\mathcal{K}))$-equivariant functors from $\ShvD_{\check{\kappa}}(\Fl_{\check{G}})$ to $\check{\mathcal{C}}$.

\end{enumerate}
\end{conj}

 Conjectural \ref{quantum local Langlands} is supposed to be characterized by the property that it intertwines the \emph{Whittaker model} and the \emph{Kac-Moody model}. The functor sending $\mathcal{C}\in G(\mathcal{K})\modu_{-\kappa}$ to its Whittaker model is co-represented by $\ShvD_{-\kappa}(G(\mathcal{K})/N(\mathcal{K}),\chi)$. The functor sending $\check{\mathcal{C}}\in \check{G}(\mathcal{K})\modu_{\check{\kappa}}$ to its Kac-Moody model is co-represented by the category of Kac-Moody representations $\hat{\check{\mathfrak{g}}}_{\check{\kappa}}\modu$. Hence,
\begin{equation}\label{assi}
    \ShvD_{-\kappa}(G(\mathcal{K})/N(\mathcal{K}),\chi)\mapsto \hat{\check{\mathfrak{g}}}_{\check{\kappa}}\modu
\end{equation}
under the equivalence (\ref{quantum local}). 

By applying the property (2) of Conjecture \ref{quantum local Langlands} to (\ref{assi}), we arrive the following conjectural equivalence,\footnotemark\
which was proposed by D. Gaitsgory and J. Lurie (\cite[Conjecture 3.11]{Ga12}).

\begin{conj}(Iwahori fundamental local equivalence)\label{conj 2}\label{iwahori con}
There is an equivalence of categories
\begin{equation}
    \Whit(\ShvD_{\kappa}(\Fl_G))\simeq \hat{\check{\mathfrak{g}}}_{\check{\kappa}}\modu^{\check{I}}
\end{equation}
Here $ \hat{\check{\mathfrak{g}}}_{\check{\kappa}}\modu^{\check{I}}$ denotes the category of Iwahori-integrable Kac-Moody representations.\footnotetext{It is proved in a recent paper \cite{[CDR]} by J. Campbell, G. Dhillon, S. Raskin using Soergel module techniques. }
\end{conj}

In the upcoming paper {\cite{[CF]}} of L. Chen and Y. Fu, the authors prove that $\Rep_q^{\mix}(\check{G})$ is equivalent to $\hat{\check{\mathfrak{g}}}_{\check{\kappa}}\modu^{\check{I}}$. Hence, the combination of our papers provides a new proof of Conjecture \ref{iwahori con}.

\subsubsection{Relation with Kazhdan-Lusztig}
By \cite{[KL]}, $\Rep_q(\check{G})$, the category of representations of the quantum group, is equivalent to the category of $\check{G}(\mathcal{O})$-integrable Kac-Moody representations. Furthermore, by \cite{[CDR]}, the latter is equivalent to the twisted Whittaker category on the affine Grassmannian $\Gr_G:= G(\mathcal{K})/G(\mathcal{O})$. Hence, there is an equivalence of categories
\begin{equation}\label{1.11}
    \Whit(\ShvD_{\kappa}(\Gr_G)):= \ShvD_{\kappa}(N(\mathcal{K),\chi}\backslash \Gr_G)\simeq \Rep_q(\check{G}).
\end{equation}
Theorem \ref{thm 2.1} provides a tamely ramified version of the above equivalence of categories.
\subsubsection{BGG Category $\mathsf{O}$}
When $q$ is generic, the Lusztig quantum group $U_q^L(\check{\mathfrak{n}}^-)$ is naturally isomorphic to the De Concini-Kac quantum group $U_q^{DK}(\check{\mathfrak{n}}^-)$. In particular, the category $\Rep_q^{\mix}(\check{G})$ is equivalent to the quantum category $\mathsf{O}$ in \cite{[BGG]} when $q$ is generic. So in this case, Theorem \ref{thm 2.1} gives a geometric realization of the quantum category $\mathsf{O}$.

In the case of root of unity, the category $\Rep_q^{\mix}(\check{G})$ is different from the quantum category $\mathsf{O}$. For example, the standard objects (i.e., Verma modules) and costandard objects (i.e., coVerma modules) of $\Rep_q^{\mix}(\check{G})$ are no longer of finite length. Nevertheless, $\Rep_q^{\mix}(\check{G})$ is still a highest weight category. The comparison of the highest weight category structures of $\Rep^{\mix}_q(\check{G})$ and $\Whit(\ShvD_\kappa(\Fl_G))$ plays an important role in our proof. 
\subsubsection{Casselman-Shalika theorem}\label{Casselman-Shalika}
 The original Casselman-Shalika theorem interprets the values of the spherical Whittaker function as characters of the irreducible representations of the Langlands dual group. 
 
 Let $\Bun_N$ be the algebraic stack classifying principal $N$-bundles on a smooth connected projective curve $X$. In \cite{[FGV]}, the authors proved a generalization of the geometric Casselman-Shalika formula.  It interprets the category of representations of the Langlands dual group as Whittaker D-modules (equivalently, $\ell$-adic sheaves) on $(\overline{\Bun_N^{\omega^\rho}})_{\infty\cdot x}$, where the algebraic stack $(\overline{\Bun_N^{\omega^\rho}})_{\infty\cdot x}$ denotes the Drinfeld compactification of $\Bun_N$ with a possible pole at a fixed point $x$ (see Section \ref{Drinfeld compactification sec} for definition). 

By a local-global comparison, \cite{[FGV]} actually proves that \eqref{1.11} is an equivalence when $q=1$.
 
The geometric Casselman-Shalika formula gives us a hint of how to construct a functor to relate the category of Whittaker sheaves to the category of representations. Namely, the ''integration'' of a Whittaker D-module along $G(\mathcal{O})$-orbits (or $N^-(\mathcal{K})$-orbits)  encodes representation-theoretic information. The construction of the functor $F^L$ (see Section \ref{subsection F} for definition) is inspired by this idea and S. Raskin's thesis \cite{[Ras1]}.

\subsubsection{Small quantum groups}
    In \cite{[GL1]}, a geometric realization of the category of representations of the small quantum groups is studied. In $loc.cit$, it is proved that the category of Hecke-eigensheaves of the twisted Whittaker categories on $\Gr_G$, is equivalent to the same category of representations of the small quantum groups. Our work adopts the strategy developed in \cite{[GL1]}{. This method is originated from \cite{[Ga2],[BFS],[FGV]}, and has proven to be a powerful method in geometric representation theory. The recent work of A.Braverman, M.Finkelberg, and R.Travkin on the Gaiotto conjecture for $\GL(N-1|N)$ (ref. \cite{[BFT]}), the work of R.Travkin and the author on the Gaiotto conjecture for $\GL(M|N)$ and the Iwahori Gaiotto conjecture also use this strategy.}

The small quantum group is very similar to the mixed quantum group: both of the categories of their representations $\Rep_q^{\sm}(\check{G})$ and $\Rep^\mix_q(\check{G})$ can be realized as categories of factorization modules. In particular, the method used in \cite{[GL1]} indicates us a strategy to prove Theorem \ref{thm 2.1} and offers us models for the constructions of the functors and stacks used in our paper. For example, the key step of the proof of our main theorem is to use the local-global equivalence of Whittaker categories and then prove the theorem in the global case. This idea comes from \cite{[GL1]}. 

In our case, there are some technical difficulties caused by the additional Iwahori structure. For example, in \cite{[GL1]}, standard objects of $\Whit_q(\Gr_G)$ are defined as $!$-pushforward of the unique irreducible twisted Whittaker sheaf supported on a single relevant $N(\mathcal{K})$-orbit and costandard objects are $*$-pushforward of that irreducible Whittaker sheaf. In our case, since coVerma modules of $\Rep_q^\mix(\check{G})$ are not compact, it is impossible to define standard objects and costandard objects similar to \cite{[GL1]} such that they match Verma modules and coVerma modules in $\Rep_q^\mix(\check{G})$. Instead, we define standard objects by $!$-Whittaker averaging of "Wakimoto sheaves" which are twisted and Iwahori equivariant against a character. However, this definition does not make sense because (classical) Wakimoto sheaves are Iwahori-equivariant and only defined in the non-twisted case. We need to extend the definition of Wakimoto sheaves to the twisted and $(I, b_\lambda)$-equivariant case. Here $b_\lambda$ is a character of the Iwahori subgroup. 

\subsection{Strategy of proof} The idea of proving Theorem \ref{thm 2.1} is to compare both sides with a factorization category. In \cite{[Ga5]}, the author proves that $\Rep_q^{\mix}(\check{G})$ can be realized as the category of factorization modules on the configuration space over a factorization algebra $\Omega_q^L$ (see Definition \ref{def lus}) in the D-module setting. Therefore, we need to construct an equivalence between $\Whit(\ShvD_\kappa(\Fl_G))$ and $\Omega_q^L-\Fact$.

A naive attempt is to use the pullback-pushforward functor $F^{DK}$,\footnotemark\ \footnotetext{The superscript $DK$ means De Concini-Kac, the reason is that this functor induces a functor to the category of representations of a group whose positive part is the De Concini-Kac quantum group.} namely, we $!$-pullback Whittaker D-modules on $\Fl_G$ to an Iwahori version Zastava space and then take $*$-pushforward to the configuration space. However, the image of $F^{DK}$ does not have a $\Omega_q^L$-factorization module structure (except when $\kappa$ is generic).  For a coweight $\lambda$, the $\lambda$-component of this functor is given by \[H(\Fl_G, \mathcal{F}\overset{!}{\otimes} j_*(\omega_{S_{\Fl,x}^{-,\lambda}})).\] 
Here $j_*(\omega_{S_{\Fl,x}^{-,\lambda}})$ denotes the $*$-extension of the dualizing D-module on the $N^-(\mathcal{K})$-orbit of $t^\lambda\in \Fl_G$.

The functor $F^L$ used in this paper is a modification of $F^{DK}$. The $\lambda$-component of $F^L$ is given by \[H(\Fl_G, \mathcal{F}\overset{!}{\otimes} j_!(\omega_{S_{\Fl,x}^{-,\lambda}})).\]Here $j_!(\omega_{S_{\Fl,x}^{-,\lambda}})$ denotes the $!$-extension of the dualizing D-module on the $N^-(\mathcal{K})$-orbit of $t^\lambda\in \Fl_G$. 

We will see that $F^L$ factors through $\Omega_q^{L,'}-\Fact$ (ref. Proposition \ref{prop 15.2}), and there is $\Omega_q^{L,'}\simeq \Omega_q^L$ if $q$ avoids small torsion. Theorem \ref{strong thm} claims that $F^L$ induces an equivalence between the Whittaker category on $\Fl_G$ and the category of factorization modules over $\Omega_q^{L,'}$.

To prove that $F^L$ is an equivalence, we will use some tautological arguments about highest weight categories. That is to say, we will construct standards and costandards of both sides, prove that $F^L$ preserves them and induces an isomorphism of Hom spaces. 

The compatibility of costandards is more or less trivial. It follows by a direct calculation of $!$-stalks of $F^L$ (ref. Corollary \ref{6.4.3}). The claim of fully faithfullness of $F^L$ follows from a calculation of Hom spaces, and the latter reduces to the problem of compatibility of standards (ref. Proposition \ref{imstan}). However, since the calculation of $*$-stalks is difficult (seems impossible), the proof of the compatibility of standards is not tautological at all. It is the main difficulty of the proof.

The method to overcome this difficulty is to define a duality functor to transfer the calculation of $*$-stalks to a calculation of $!$-stalks. This duality functor is not tautological. Since the Whittaker category used in this paper is defined as a category of invariants, its dual category is a category of coinvariants in nature. By a theorem of S. Raskin (see \cite[Theorem 2.1.1]{[Ras2]}), we can identify the invariant-Whittaker category as the coinvariant-Whittaker category. Hence, this duality functor can be defined.

We need to prove that this duality functor intertwines $F^L$ and $F^{DK}$. Following \cite{[Ga5]}, we can identify the (local) Whittaker category with the \textit{global} Whittaker category. By translating the problem into the global Whittaker category defined on a Drinfeld compactification, we need to prove that the Verdier duality functor intertwines the global functors corresponding to $F^L$ and $F^{DK}$ (ref. Theorem \ref{dualfunctor}). According to the constructions, we need to compare a $!$-tensor product and a $*$-tensor product. We solve this problem by using a universally locally acyclic property.
\subsection{Organization of the paper}\label{Road map}
\begin{enumerate}[label=(\arabic*)]
    \item In Section \ref{geometric prep}, we introduce some prestacks and gerbes used in this paper.
    \item In Section \ref{sec 2}, we explain the definitions of the Whittaker category on $\Fl_G$ and $\Rep_q^\mix(\check{G})$. Then we state Theorem \ref{thm 2.1}. 
    \item In Section \ref{Fact geo repl}, we review factorization algebras and factorization modules. We replace Theorem \ref{thm 2.1} by an equivalent statement Theorem \ref{main theorem 1}. 
    \item In Section \ref{local Whittaker category}, we study standards and duality functor of the Whittaker category, and show that standards compactly generate the Whittaker category.
    \item In Section \ref{try proof 1}, we construct a functor $F^L$ which goes from the Whittaker side to the factorization side. We show that $F^L$ is an equivalence functor modulo Proposition \ref{imstan} which is about the comparison of standards. 
    \item In Section \ref{Whittaker category: global definition}, the global Whittaker category is defined. By using the global Whittaker category, we reduce Proposition \ref{imstan} to  Proposition \ref{imstan glob} where we can use a {universally locally acyclic} property.
    \item In Section \ref{step 5}, we prove  Proposition \ref{imstan glob}.
\end{enumerate}

\subsection{Generality of our results}
\subsubsection{Sheaf theories} \label{sheaf context}
Let $\mathsf{k}$, $\mathsf{e}$ be algebraically closed fields, and $\textnormal{char}(\mathsf{e})=0$. The strong version of our main result (Theorem \ref{strong thm}) is true for any of the sheaf theories listed in \cite[Section 0.8.8]{[GL1]}:
\begin{enumerate}[label=(\arabic*)]
     \item (de Rham) Schemes are defined over $\mathsf{k}$ (assume $\textnormal{char}(\mathsf{k})=0$ here) and the sheaf category is the category of D-modules, or the ind-completion of the category of holonomic D-modules, or the ind-completion of the category of regular holonomic D-modules;
    \item (Betti) Schemes are defined over $\mathbb{C}$ and the sheaf category is the ind-completion of the category of constructible sheaves with respect to the classical topology with coefficients $\mathsf{e}$;
    \item ($\ell$-adic) Schemes are defined over $\mathsf{k}$ and the sheaf category is the ind-completion of the category of constructible $\bar{\mathbb{Q}}_\ell$-adic sheaves. Here $\mathsf{e}=\bar{\mathbb{Q}}_\ell$.
\end{enumerate}
We denote by $\Shv$ any sheaf theory listed above.

Note that the Whittaker category is not always well-defined for the sheaf theories above, such as the Betti setting and the $\ell$-adic setting on schemes defined over a field $\mathsf{k}$ of characteristic 0. In these cases, neither the exponential D-module nor the Artin-Schreier sheaf makes sense, so we are not allowed to talk about $(N(\mathcal{K}),\chi)$-equivariant sheaves. We need to replace the Whittaker category by the Kirillov model (ref. \cite[Appendix A]{[Ga1]}) in Theorem \ref{strong thm}. However, applying the Lefschetz principle and Riemann-Hilbert correspondence, the proof for these cases can be easily reduced to the setting of regular holonomic D-modules. In order to simplify the notations and not get distracted, we will only focus on the sheaf theories such that the Whittaker category makes sense.

\subsubsection{Deformation parameters}
In the general case, we need to use gerbes to twist a sheaf category. Given a gerbe $\mathcal{G}$ with respect to the multiplicative group $\mathsf{e}^\times$, one can twist a sheaf category with coefficients in $\mathsf{e}$. We refer readers to \cite[Section 1.7]{[GL2]} for the definition of the category of $\mathcal{G}$-twisted sheaves. 

Let us be more precise about which kind of gerbes are used in different sheaf theories. In the Betti setting and $\ell$-adic setting, we use the gerbes with respect to the torsion multiplicative group $\mathsf{e}^{\mathsf{torsion},\times}$ (see \cite[Section 1.3]{[GL2]}). In the D-module setting, we use the tame gerbes (\cite[Section 3.3]{[Zh]}). 

If readers are not familiar with how to twist a sheaf category with a gerbe, we advise to think about the case of D-modules, and $G$ is a simple and simply-connected group over $\mathbb{C}$. In this case, the twisting parameter $q$ is just a non-zero complex number.  Twisted D-modules on $\Fl_G$ are those D-modules on the canonical line bundle of $\Fl_G$ which are $\mathbb{G}_m$-monodromic along the fiber with the monodromy $q^2$. This restriction of generality will not lose the main interest of this paper. 

\subsection{Conventions and notations}

In the main body of the paper, to simplify, we assume that $G$ is a reductive group defined over any algebraically closed field $\mathsf{k}$, and the derived subgroup $[G,G]$ is simply connected. Let $T:= B^-\cap B$ be the Cartan subgroup of $G$.

We denote by $\Lambda$ the coweight lattice of $G$ and by $\check{\Lambda}$ the weight lattice. Let $\Lambda^{\textnormal{{neg}}}$ be the semi-group spanned by negative simple coroots. Its inverse is denoted by $\Lambda^{\textnormal{pos}}$. Set $\Lambda^+$  (resp. $\check{\Lambda}^+$) the semi-group of dominant coweights (resp. dominant weights), $\Delta$ the root system of $\check{G}$, and $\alpha_1, \alpha_2, \cdots, \alpha_r$ the simple coroots. Let $W$ denote the finite Weyl group and $W^{\ext}$ denote the extended affine Weyl group. 

The theory of sheaves '$\Shv$' on infinite-dimensional schemes (also prestacks) used in this paper is developed in \cite{[Ber]}, \cite{[GR1]}, \cite{[GR2]}, \cite{[Ras3]}, etc. When we talk about the Whittaker category, we assume that we are in the D-module setting or the $\ell$-adic setting.

In this paper, the categories considered are cocomplete $\mathsf{e}$-linear DG-categories (see \cite[Chapter 1, Section 10]{[GR1]}).
We need the theory of higher categories developed in \cite{[Lu1]} and \cite{[Lu2]} in this paper. 

Let $\Vect$ be the $(\infty,1)$-category of complexes of vector spaces over $\mathsf{e}$. Given a category $\mathcal{C}$ and $c_1, c_2\in \mathcal{C}$, we denote by $\mathcal{H}om_{\mathcal{C}}(c_1, c_2)\in \Vect$ the Hom space of $c_1$ and $c_2$, and denote by $\Hom_{\mathcal{C}}(c_1, c_2):= H^0(\mathcal{H}om_{\mathcal{C}}(c_1, c_2))$.

\section{Geometric Preparation}\label{geometric prep}
In this section, we will define some basic geometric objects used in this paper. We will, first of all, recall the definitions of \textit{Ran space} and \textit{Configuration space} (Sections \ref{Ran space sec} and \ref{ fixed point version configuration spaces}), and then review the definitions of Ran-ified (or, Beilinson-Drinfeld) affine flags and affine Grassmannian in Section  \ref{Ran version of affine flags and affine Grassmannian}. In Section \ref{gerbe used}, we will explain the gerbes used in this paper. 

\subsection{Ran space}\label{Ran space sec}
 The Ran space is important for us, since it is naturally factorizable. We need factorization prestacks over $\Ran$ to perform our construction of the equivalence in Theorem \ref{thm 2.1}. 

Let $X$ be a smooth connected projective curve defined over $\mathsf{k}$.
\begin{definition}
The Ran space $\Ran:=\Ran_X$ is defined as the prestack whose $S$-points classify non-empty finite sets $\mathcal{I}$ of $\Maps(S, X)$ for any affine scheme $S$ over $\mathsf{k}$. 

We denote by $(\Ran\times \Ran)_{disj}$ the open sub-prestack of $ \Ran\times \Ran$ with disjoint support condition. 

\end{definition}
The Ran space admits a (non-unital) semi-group structure by taking union
\begin{equation}\label{eq 2.1}
    \cup: \Ran\times \Ran\to \Ran
\end{equation}
\[\mathcal{I}_1, \mathcal{I}_2\mapsto \mathcal{I}_1\cup \mathcal{I}_2.\]

Let $\mathcal{D}_{\mathcal{I}}$ be the formal completion of $S\times X$ along the graph of $\mathcal{I}$, and denote by $\overset{\circ}{\mathcal{D}}_{\mathcal{I}}$ the open subscheme of $\mathcal{D}_{\mathcal{I}}$ obtained by removing the graph of $\mathcal{I}$.


\begin{definition}\label{def 2.3}
For any affine scheme $S$, the $S$-points of $\Gr_{T,\Ran}$  classify the triples $(\mathcal{I}, \mathcal{P}_T,\alpha)$, where $\mathcal{I}\in \Ran(S)$, $\mathcal{P}_T$ is a $T$-bundle on $\mathcal{D}_{\mathcal{I}}$, $\alpha$ is an isomorphism of $\mathcal{P}_T$ with the trivial $T$-bundle $\mathcal{P}^0_T$ on $\overset{\circ}{\mathcal{D}}_{\mathcal{I}}$.
\end{definition}
The prestack $\Gr_{T,\Ran}$ is called the Beilinson-Drinfeld (i.e., Ran-ified) affine Grassmannian, ref., \cite[Section 5.3.11]{[BD1]}.

\begin{rem}
By the Beauville-Laszlo theorem \cite{[BL]}, we can require that $\mathcal{P}_T$ is a $T$-bundle on $S\times X$, and $\alpha$ is an isomorphism of $\mathcal{P}_T$ with $\mathcal{P}^0_T$ on the complement of the graph of $\mathcal{I}$. The resulting prestack is the same.
\end{rem}

The important note here is that $\Gr_{T,\Ran}$ is factorizable over $\Ran$. Namely, we have an isomorphism
\begin{equation}
    \Gr_{T, \Ran}\mathop{\times}\limits_{\Ran}{(\Ran\times \Ran)_{disj}}\simeq \Gr_{T, \Ran}\boxtimes \Gr_{T, \Ran}\underset{\Ran\times \Ran}{\times}
    {(\Ran\times \Ran)_{disj}},
\end{equation}
with higher homotopy coherence ({See \cite[Section 5.1.2]{[GL1]}}).

Inside $\Gr_{T, \Ran}$, there is a closed factorization sub-prestack denoted by $\Gr^{\textnormal{neg}}_{T, \Ran}$ (\cite[4.6.2]{[GL1]}).

\begin{definition}
If $G$ is semi-simple and simply connected, then a $S$-point $(\mathcal{I}, \mathcal{P}_T,\alpha)$ of $\Gr_{T, \Ran}$ is in $\Gr^{\textnormal{neg}}_{T, \Ran}$ if 
\begin{itemize}
    \item for any dominant weight $\check{\lambda} \in\check{\Lambda}^+ $, the meromorphic map of the line bundles on $S\times X$
    \begin{equation}\label{2.4 eq}
        \check{\lambda}(\mathcal{P}_T) \to\check{\lambda}(\mathcal{P}^0_T )
    \end{equation}
induced by $\alpha$, is regular.
    \item  for any point $s \in S$ and any element $i \in \mathcal{I}$, there exists at least one $\check{\lambda} \in\check{\Lambda}^+$, such that \eqref{2.4 eq} has a zero at the point $s\to S \overset{i}{\to} X$.
\end{itemize}

For general reductive group $G$, we define $\Gr^{\textnormal{neg}}_{T, \Ran}$ as $\Gr^{\textnormal{neg}}_{T^{sc}, \Ran}$, where $T^{sc}$ is the Cartan subgroup of the simply connected cover of $[G,G]$.
\end{definition}

\subsubsection{}We also need the Ran space with a marked point.
\begin{definition}
Fix $x\in X$. We denote by $\Ran_{x}:=\Ran_{X,x}$ the prestack whose $S$-points classify non-empty sets $\mathcal{I}$ of $\Maps(S, X)$ with a distinguished element $\widetilde{x}$, where $\widetilde{x}$ denotes the constant map $\widetilde{x}: S\to x\to X$.
\end{definition}
Taking union defines a map
\begin{equation}
    \cup_x: \Ran\times \Ran_{x}\to \Ran_{x}.
\end{equation}
It equips $\Ran_{x}$ with a structure of module space over $\Ran$.

Let $\Gr_{T, \Ran_x}$ be $\Ran_x\underset{\Ran}{\times}\Gr_{T, \Ran}$. It has a closed sub-prestack $(\Gr^{\textnormal{neg}}_{T, \Ran_x})_{\infty\cdot x}$.
\begin{definition}
For any affine scheme $S$, a $S$-point $(\mathcal{I}, \mathcal{P}_T,\alpha)$ of $\Gr_{T, \Ran_x}$ belongs to  the sub-prestack $(\Gr^{\textnormal{neg}}_{T, \Ran_x})_{\infty\cdot x}$,
if  there exists a $T$-bundle $\mathcal{P}_{T,1}$  on $S \times X$ and an isomorphism \[\alpha': \mathcal{P}_{T,1} |_{ S \times (X\setminus x)} \simeq \mathcal{P}_{T} |_{ S \times (X\setminus x)},\] such that the resulting point $(\mathcal{I}, \mathcal{P}_{T,1},\alpha\circ\alpha')$ of $\Gr_{T, \Ran_x}$ belongs to $\Ran_x\mathop{\times}\limits_{\Ran}\Gr^{\textnormal{neg}}_{T, \Ran}$.
\end{definition}



\subsection{Configuration spaces}\label{ fixed point version configuration spaces}

\begin{definition}\label{def 2.7}
The configuration space $\Conf:=\Conf(X, \Lambda^{\textnormal{neg}})$ is defined as the scheme classifying colored divisors of $X$ with coefficients in $\Lambda^{\textnormal{neg}}\setminus\{0\}$, i.e., it classifies
\begin{equation}\label{D}
    D= \mathop{\sum}\limits_k \lambda_k\cdot x_k,\qquad \lambda_k\in \Lambda^{\textnormal{neg}}\setminus\{0\}.
\end{equation}
\end{definition}

\subsubsection{Connected components}
Connected components of $\Conf$ are indexed by $\Lambda^{\textnormal{neg}}\setminus\{0\}$, \[\Conf= \mathop{\bigsqcup}_{\Lambda^{\textnormal{neg}}\setminus\{0\}} \Conf^\lambda.\]
Here $\Conf^\lambda$ denotes the subscheme of $\Conf$ where we require the total degree of $D$ in \eqref{D} (i.e., $\sum_k \lambda_k$) to be $\lambda$. If $\lambda=-\sum_i n_i\cdot \alpha_i$, then ${\Conf^\lambda}$  is isomorphic to $\prod_i {X^{(n_i)}}$, where ${X^{(n_i)}}$ classifies unordered $n_i$ points in $X$.

Similar to the Ran space, $\Conf$ is equipped with a structure of non-unital commutative semi-group. There is a map
\begin{equation}\label{add conf}
    \add_\Conf: \Conf\times \Conf\to \Conf
\end{equation}
\[D_1, D_2\mapsto D_1+D_2.\]

If we restrict this map to the open subscheme $(\Conf\times \Conf)_{disj}$ with disjoint support condition, then it is \'{e}tale. 

Configuration space is essentially the same as $\Gr^{\textnormal{neg}}_{T, \Ran}$. The following lemma is from \cite[Lemma 4.6.4]{[GL1]}.

\begin{lem}\label{shf eq conf Gr}
Evaluation on fundamental weights gives rise to a morphism \begin{equation}\label{evalutiongrt}
    \Gr^{\textnormal{neg}}_{T, \Ran}\to \Conf.
\end{equation}
It induces an isomorphism of the sheafifications in the topology generated by finite surjective maps.
\end{lem}

In particular, (\ref{evalutiongrt}) induces an equivalence between categories of gerbes on $\Gr^{\textnormal{neg}}_{T, \Ran}$ and $\Conf$. Furthermore, it induces an equivalence of corresponding categories of twisted sheaves.

\subsubsection{}Similarly, we define the configuration space with a marked point.
\begin{definition}
Fix $x\in X$. We denote by $\Conf_{\infty\cdot x}$ the ind-scheme classifying the colored divisors on $X$ with $\Lambda$-coefficient
\begin{equation}\label{Dx}
 D= \lambda_x\cdot x+ \mathop{\sum}\limits_k \lambda_k\cdot x_k,
\end{equation}
such that $\lambda_k\in \Lambda^{\textnormal{neg}}$, $\lambda_x\in \Lambda$ and $x_k\neq x$.

\end{definition}

Regard $\Conf$ as a (non-unital) algebra in the category of prestacks, the addition map 
\begin{equation}\label{2.14}
    \add_{\Conf_x}: \Conf\times \Conf_{\infty\cdot x}\to \Conf_{\infty\cdot x}
\end{equation}
gives $\Conf_{\infty\cdot x}$ a module structure over $\Conf$.

If we restrict  $\add_{\Conf_x}$ to $(\Conf\times \Conf_{\infty\cdot x})_{disj}$, the open ind-scheme with disjoint support condition, then it is \'{e}tale. In particular, $\add_{\Conf_x}^!\simeq \add_{\Conf_x}^*$ on $(\Conf\times \Conf_{\infty\cdot x})_{disj}$.

Similar to (\ref{evalutiongrt}),  there is a map of prestacks
\begin{equation}\label{evalutiongrtinf}
    (\Gr^{\textnormal{neg}}_{T, \Ran})_{\infty\cdot x}\to \Conf_{\infty\cdot x}.
\end{equation}

\begin{lem}[{\cite[4.6.7]{[GL1]}}]\label{lemma 5.2}
The morphism (\ref{evalutiongrtinf}) induces an isomorphism of the sheafifications in the topology generated by finite surjective maps.
\end{lem}

\subsection{Ran-ified Fl and Gr}\label{Ran version of affine flags and affine Grassmannian}
The Beilinson-Drinfeld affine Grassmannian $\Gr_{G, \Ran}$ is similarly defined as $\Gr_{T, \Ran}$ (only replace $T$-bundles by $G$-bundles in the definition). By adding an Iwahori structure at $x$, we arrive the definition of the Beilinson-Drinfeld affine flags.

\begin{definition}\label{Ran version Fl}
$\Fl_{G,\Ran_x}$ is the prestack whose $S$-points classify the quadruples $(\mathcal{I}, \mathcal{P}_G, \alpha, \epsilon)$, where $(\mathcal{I}, \mathcal{P}_G, \alpha)\in \Gr_{G, \Ran}(S)$, and $\epsilon$ is a $B$-reduction of $\mathcal{P}_G$ over $S\times x$.
\end{definition}

An important feature of $\Gr_{G, \Ran}$ is that it is factorizable over $\Ran$. That is to say, we have an isomorphism
\begin{equation}\label{2.12 h}
    \Gr_{G, \Ran}\mathop{\times}\limits_{\Ran}{(\Ran\times \Ran)_{disj}}\simeq \Gr_{G, \Ran}\boxtimes \Gr_{G, \Ran}\underset{\Ran\times \Ran}{\times}
    {(\Ran\times \Ran)_{disj}},
\end{equation}
with higher homotopy coherence.

 $\Fl_{G, \Ran_x}$ is a factorization module space over $\Gr_{G, \Ran}$, i.e., for any  non-empty set with a distinguished point $(*\in {\cI})$, there is an isomorphism
\begin{equation}
\begin{split}
     \Fl_{G, \Ran_x}\mathop{\times}\limits_{\Ran_x}{(\Ran\times \Ran_x)_{disj}}\simeq \Gr_{G, \Ran}\boxtimes \Fl_{G, \Ran_x}\underset{\Ran\times \Ran_x}{\times}{(\Ran\times \Ran_x)_{disj}},
\end{split}
\end{equation}
with higher homotopy coherence.

\subsection{Gerbes used in this paper}\label{gerbe used}
In the Betti setting and the $\ell$-adic setting, we let $\mathcal{G}^G$ be a factorization $\mathsf{e}^{\mathsf{torsion},\times}$-gerbe  on $\Gr_{G, \Ran}$, which is compatible with the factorization structure on $\Gr_{G, \Ran}$ (see \eqref{2.12 h}). It is defined in \cite[Section 2.4]{[GL2]}, and is called metaplectic parameter. In the D-module setting, we should require $\mathcal{G}^G$ to be a \textit{tame} factorization gerbe on $\Gr_{G, \Ran}$ defined in \cite[Section 3.3]{[Zh]}. In this section, we will explain how to get gerbes on some prestacks from $\mathcal{G}^G$ on $\Gr_{G,\Ran}$.



\subsubsection{Gerbe on the Hecke prestack}
We denote by $\Hecke_G$ the Hecke prestack
which classifies the data: $(\mathcal{P}_{G,1}, \mathcal{P}_{G,2}, \alpha)$, where $\mathcal{P}_{G,1}$ and $\mathcal{P}_{G,2}$ are $G$-bundles on $X$ and $\alpha$ is an isomorphism of $\mathcal{P}_{G,1}$ and $\mathcal{P}_{G,2}$ over $X\setminus x$, $\alpha: \mathcal{P}_{G,1}|_{X\setminus x}\simeq \mathcal{P}_{G,2}|_{X\setminus x}$. Then by the $G(\mathcal{O})$-equivariance of $\mathcal{G}^G$ (see \cite[Section 7.3]{[GL2]}), $\mathcal{G}^G$ gives rise to a gerbe on the Hecke prestack. We denote the descent gerbe on $\Hecke_G$ by $\mathcal{G}^G_{\Hecke}$.

\subsubsection{$\omega^\rho$-twisted prestacks}\label{2.5.3}
Fix a square root of the canonical line bundle $\omega$ on $X$ and denote it by $\omega^{\otimes\frac{1}{2}}$. We define $\omega^\rho$ as the $T$-bundle induced from $\omega^{\otimes\frac{1}{2}}$ by the morphism of group schemes
\begin{equation}
    2\rho: \mathbb{G}_m\to T.
\end{equation}
Here $\rho$ is the sum of all fundamental coweights.

 If we replace the trivial $G$-bundle $\mathcal{P}_G^0$ by the $G$-bundle $\mathcal{P}_G^{\omega}:=\omega^\rho\overset{T}{\times}G$ in the definition of $\Gr_{G, \Ran}$, we will obtain the $\omega^\rho$-twisted Beilinson-Drinfeld affine Grassmannian. Let us denote it by $\Gr^{\omega^\rho}_{G, \Ran}$. It is still a factorization prestack over $\Ran$. Similarly, we can also define $\Fl^{\omega^\rho}_{G, \Ran_x}$, $\Gr_{G}^{\omega^\rho}$, $\Fl_{G}^{\omega^\rho}$, $G(\mathcal{K})^{\omega^\rho}$, $N(\mathcal{K})^{\omega^\rho}$,  $\Gr_{T, \Ran}^{\omega^\rho}$, $(\Gr_{T, \Ran}^{\omega^\rho})^{\textnormal{neg}}$, $(\Gr_{T, \Ran_x}^{\omega^\rho})_{\infty\cdot x}^{\textnormal{neg}}$, etc. Similar to the classical affine flags and affine Grassmannian, we have \[\Fl_{G}^{\omega^\rho}\simeq G(\mathcal{K})^{\omega^\rho}/I^{\omega^\rho},\] and \[\Gr_{G}^{\omega^\rho}\simeq G(\mathcal{K})^{\omega^\rho}/G(\mathcal{O})^{\omega^\rho}.\] Here $I^{\omega^\rho}$ and $G(\mathcal{O})^{\omega^\rho}$ are $\omega^\rho$-twisted version of $I$ and $G(\mathcal{O})$, respectively.

\subsubsection{Gerbes on $\Fl^{\omega^\rho}_G$ and $\Gr^{\omega^\rho}_G$}\label{2.5.4} By definition, taking $\mathcal{P}_{G,2}$ to be $\mathcal{P}_G^{\omega}$ defines a map $\Gr_{G, \Ran}^{\omega^\rho}\to \Hecke_G$. The pullback of $\mathcal{G}^G_{\Hecke}$ along this map is a factorization gerbe on $\Gr_{G, \Ran}^{\omega^\rho}$. With some abuse of notation, we denote it by $\mathcal{G}^G$. Its pullback to $\Fl_{G,\Ran_x}^{\omega^\rho}$ (resp. $\Fl_{G}^{\omega^\rho}$, $\Gr_{G}^{\omega^\rho}$, $G(\mathcal{K})^{\omega^\rho}$, etc) is also denoted by $\mathcal{G}^G$.

By \cite[Proposition 7.2.5]{[GL2]}, the pullback of $\mathcal{G}^G$ along $G(\mathcal{K})^{\omega^\rho}\to \Fl_{G,\Ran_x}^{\omega^\rho}$ is a multiplicative gerbe, i.e.,
\[m^!(\mathcal{G}^G)\simeq \mathcal{G}^G\boxtimes \mathcal{G}^G.\] Here $m$ denotes the multiplication map \[G(\mathcal{K})^{\omega^\rho}\times G(\mathcal{K})^{\omega^\rho}\to G(\mathcal{K})^{\omega^\rho}.\]

In particular, the gerbes $\mathcal{G}^G$ on $\Fl_{G}^{\omega^\rho}$ and $\Gr_{G}^{\omega^\rho}$ are equivariant with respect to the action of $G(\mathcal{K})^{\omega^\rho}$ against the gerbe $\mathcal{G}^G$.

\subsubsection{Gerbe on $\Conf$}Replace $G$ by $B$ in the definition of $\Gr_{G, \Ran}^{\omega^\rho}$, one can define a factorization prestack $\Gr_{B, \Ran}^{\omega^\rho}$. Consider the following diagram of prestacks:
\begin{center}
\begin{equation}
\begin{tikzpicture}[scale=1]
  \draw [->] (-1,1) -- (-2.5,0);
  \draw [->] (1,1) -- (2.5,0);
  \node at (0,1.2){$\Gr_{B, \Ran}^{\omega^\rho}$};
  \node [below] at (2.5,0){$\Gr_{T, \Ran}^{\omega^\rho}.$};
  \node [below] at (-2.5,0){$\Gr_{G, \Ran}^{\omega^\rho}$};
\end{tikzpicture}
\end{equation}
\end{center}
The pullback of $\mathcal{G}^G$ on $\Gr_{G,\Ran}^{\omega^\rho}$ along the left morphism gives a factorization gerbe on $\Gr_{B, \Ran}^{\omega^\rho}$. By \cite[Section 5.1]{[GL2]}, this factorization gerbe descends to a factorization gerbe on $\Gr_{T, \Ran}^{\omega^\rho}$. We denote the resulting gerbe by $\mathcal{G}^T$.

By constructions similar to (\ref{evalutiongrt}) and (\ref{evalutiongrtinf}), we have maps
\begin{equation}\label{omega to conf}
    (\Gr_{T, \Ran}^{\omega^\rho})^{\textnormal{neg}}\to \Conf,
\end{equation}
and
\begin{equation}\label{omega to confx}
    (\Gr_{T, \Ran_x}^{\omega^\rho})_{\infty\cdot x}^{\textnormal{neg}}\to \Conf_{\infty\cdot x}.
\end{equation}

By (a tiny modification of) Lemmas \ref{shf eq conf Gr} and \ref{lemma 5.2}, (\ref{omega to conf}) and (\ref{omega to confx}) induce equivalences of gerbes on the Beilinson-Drinfeld affine Grassmannians and Configuration spaces. Hence, we can descend the factorization gerbe $\mathcal{G}^T$ on $(\Gr_{T, \Ran}^{\omega^\rho})^{\textnormal{neg}}$ (resp. $(\Gr_{T, \Ran_x}^{\omega^\rho})_{\infty\cdot x}^{\textnormal{neg}}$) to a gerbe $\mathcal{G}^\Lambda$ on $\Conf$ (resp. $\Conf_{\infty\cdot x}$). Note that the above maps are compatible with the factorization structures, the gerbes on $\Conf$ and $\Conf_{\infty\cdot x}$ are factorizable. 

Following \cite[Section 4.2]{[GL2]}, we can get a quadratic form
\begin{equation}\label{quad f}
    q: \Lambda\to \mathsf{k}/\mathbb{Z}
\end{equation}
from a factorization gerbe $\mathcal{G}^\Lambda$ on $\Conf$ in the D-module setting. In the Betti setting and the $\ell$-adic setting, $q$ takes value in $\mathsf{e}^{\mathsf{torsion},\times}(-1):= \colim_n \Hom(\mu_n, \mathsf{e}^{\mathsf{torsion},\times})$. 
\subsubsection{Gerbe on $\Bun_G$}
Let $\Bun_G$ denote the algebraic stack classifying principal $G$-bundles on $X$. It is shown in \cite{[GL2]} that any factorization gerbe on $\Gr_{G, \Ran}$ descends to a gerbe on $\Bun_G$. We will also denote the resulting gerbe on $\Bun_G$ by $\mathcal{G}^G$.

\begin{rem}
The pullback of $\mathcal{G}^G$ on $\Bun_G$ along
\begin{equation}
    \Gr_{G,\Ran}^{\omega^\rho}\to \Bun_G
\end{equation}
is the gerbe $\mathcal{G}^G$ defined in Section \ref{2.5.4} tensored with the fiber $\mathcal{G}^G|_{\omega^\rho\in \Bun_G}$. 
\end{rem} 
\section{Statement of the main theorem}\label{sec 2}
In this section, we will introduce two sides of Theorem \ref{thm 2.1} explicitly:
\begin{enumerate}[label=(\arabic*)]
    \item the category of Whittaker sheaves on affine flags (Section \ref{Whittaker categories through invariant})
    \item the category of representations of the mixed quantum group (Section \ref{quantum groups}).
\end{enumerate}


\subsection{Definition of Whittaker category (through invariants)}\label{Whittaker categories through invariant}

In Section \ref{2.5.3}, we have already defined $\omega^\rho$-twisted prestacks. These objects have an advantage over the non-twisted prestacks: we can define the non-degenerate character $\chi$ of $N(\mathcal{K})^{\omega^\rho}$ and Whittaker sheaves on $\Fl^{\omega^\rho}_{G}$ canonically. In the rest of this paper, we will consider the $\omega^\rho$-twisted affine flags $\Fl_G^{\omega^\rho}$ and related geometric objects. One can show that the Whittaker category on $\Fl_G$ and the corresponding category on $\Fl_G^{\omega^\rho}$ are equivalent. 

Consider a non-degenerate character of $N(\mathcal{K})^{\omega^\rho}$,
\begin{equation}\label{naive chi}
    \chi:N(\mathcal{K})^{\omega^\rho}\overset{\projection}{\longrightarrow} N(\mathcal{K})^{\omega^\rho}/[N(\mathcal{K})^{\omega^\rho},N(\mathcal{K})^{\omega^\rho}]\overset{\sim}{\longrightarrow }\omega(\mathcal{K})^r\overset{\add}{\longrightarrow} \omega(\mathcal{K})\overset{\textnormal{residue}}{\longrightarrow} \mathbb{G}_a.
\end{equation}

The pullback of the exponential D-module along $\chi$ is a character D-module, we denote this character D-module by the same notation $\chi$. In the $\ell$-adic case, we use the Artin-Schreier sheaf instead of the exponential D-module here.



\begin{definition}\label{def Whit}
We define the Whittaker category on affine flags as
\begin{equation}
\Whit_q(\Fl^{\omega^\rho}_G):= \Shv_{\mathcal{G}^G}(\Fl^{\omega^\rho}_G)^{N(\mathcal{K})^{\omega^\rho},\chi}=\Shv_{\mathcal{G}^G}(N(\mathcal{K})^{\omega^\rho},\chi\backslash\Fl^{\omega^\rho}_G).
\end{equation}
Here $q$ is the quadratic form associated with $\mathcal{G}^G$ (ref. \eqref{quad f}).
\end{definition}
\subsubsection{} \label{N k}We note that Definition \ref{def Whit} involves taking invariants with respect to an ind-pro-group scheme, we need to be more precise about this definition. 

Write $N(\mathcal{K})^{\omega^\rho}$ as
\begin{equation}\label{eq 3.3}
    N(\mathcal{K})^{\omega^\rho}= \bigcup_{k\geq 0}N_k,
\end{equation}
where $N_k:= \textnormal{Ad}_{t^{-k\rho}}N(\mathcal{O})^{\omega^\rho}$.

First, by \cite[Section 4.4.3]{[Ber]}, we have
\begin{equation}\label{3.3}
    \Whit_q(\Fl^{\omega^\rho}_G)=\Shv_{\mathcal{G}^G}(\Fl^{\omega^\rho}_G)^{N(\mathcal{K})^{\omega^\rho},\chi}\simeq \mathop{\lim}\limits_{k, \oblv}\Shv_{\mathcal{G}^G}(\Fl^{\omega^\rho}_G)^{N_k,\chi}
\end{equation}


Fix a natural number $k\geq 0$. By (\ref{3.3}), we only need to define $\Shv_{\mathcal{G}^G}(\Fl^{\omega^\rho}_G)^{N_k,\chi} $.

Since $\Fl^{\omega^\rho}_G$ is an ind-scheme, we can write $\Fl^{\omega^\rho}_G$ as a colimit of finite-dimensional schemes $Y_i$. Furthermore, we can assume that each $Y_i$ is $N_k$-invariants. Then by
\[\Shv_{\mathcal{G}^G}(\Fl^{\omega^\rho}_G)\simeq \mathop{\lim}\limits_{i,!-\pullback}\Shv_{\mathcal{G}^G}(Y_i),\]
we have
\begin{equation}
    \Shv_{\mathcal{G}^G}(\Fl^{\omega^\rho}_G)^{N_k,\chi}\simeq \mathop{\lim}\limits_{i, !-\pullback}\Shv_{\mathcal{G}^G}(Y_i)^{N_k,\chi}.
\end{equation}

We note that $N_k$ is a pro-scheme of finite type, we can write it as
\[N_k= \mathop{\lim}\limits_{l}N^l_k,\] such that each $N^l_k$ is a finite-dimensional unipotent group scheme and the action of $N_k$ on $Y_i$ factors through $N^l_k$. Finally, we define
\begin{equation}
    \Shv_{\mathcal{G}^G}(Y_i)^{N_k,\chi}:= \Shv_{\mathcal{G}^G}(Y_i)^{N^l_k,\chi}.
\end{equation}
Since for any $l'\geq l$, the kernel of $N_k^{l'}\to N_k^{l}$ is unipotent, the above definition is independent of the choice of $N_k^l$.


\subsubsection{Averaging functors}
Denote by
\begin{equation}
    \oblv_{N(\mathcal{K})^{\omega^\rho},\chi}\colon \Whit_q(\Fl^{\omega^\rho}_G)\to \Shv_{\mathcal{G}^G}(\Fl^{\omega^\rho}_G)
\end{equation}
the fully faithful forgetful functor. 

It admits a (partially defined) left adjoint functor $\Av_!^{N(\mathcal{K})^{\omega^\rho},\chi}$. Since $\Av_!^{N(\mathcal{K})^{\omega^\rho},\chi}$ is a (partially defined) left adjoint functor, it commutes with filtered colimits. On the contrary, the right adjoint functor $\Av_*^{N(\mathcal{K})^{\omega^\rho},\chi}$ of $\oblv_{N(\mathcal{K})^{\omega^\rho},\chi}$ is discontinuous.

With the ind-pro-group scheme presentation \eqref{eq 3.3} of $N(\mathcal{K})^{\omega^\rho}$, we may write $\Av_!^{N(\mathcal{K})^{\omega^\rho},\chi}$ more precisely,
\begin{align}
    \Av_!^{N(\mathcal{K})^{\omega^\rho},\chi}= \mathop{\colim}\limits_{k, !-\averaging} \Av_!^{N_k,\chi}.
\end{align}
If $\Av_!^{N_k,\chi}(\mathcal{F})$ can be defined for any $k$, then $\Av_!^{N(\mathcal{K})^{\omega^\rho},\chi}(\mathcal{F})$ can be defined by taking colimit. In particular, $\Av_!^{N(\mathcal{K})^{\omega^\rho},\chi}$ can be defined in the $\ell$-adic setting and for ind-holonomic D-modules in the D-module setting.

\subsection{Mixed quantum groups}\label{quantum groups}
The quantum group used in this paper is not the classical quantum group. It is neither the Lusztig quantum group nor its graded dual. It is a combination of these two quantum groups: the positive part is the Lusztig quantum group $U^L_q(\check{\mathfrak{n}})$ and the negative part is the graded dual of it (i.e., the De Concini-Kac quantum group $U^{DK}_q(\check{\mathfrak{n}}^-)$).

\subsubsection{Mixed representation category}\label{Mixed representation category}
Let $\Vect_q^\Lambda$ denote $\Rep_q(\check{T})$, the braided monoidal category of $\mathsf{e}$-representations of the quantum torus $\check{T}$. We denote by $U_q^L(\check{\mathfrak{n}})\modu^{\mathsf{loc.nil}}$ the ind-completion of the derived category of finite-dimensional $U_q^L(\check{\mathfrak{n}})$-modules in $\Vect_q^\Lambda$. 

\begin{definition}
The category $\Rep_q^{\mix}(\check{G})$ is defined as $Z_{Dr, \Vect^\Lambda_q}(U_q^L(\check{\mathfrak{n}})\modu^{\mathsf{loc.nil}})$, which is the relative Drinfeld center of $U_q^L(\check{\mathfrak{n}})\modu^{\mathsf{loc.nil}}$ with respect to $\Vect_q^\Lambda$ (see \cite[27.2]{[GL1]}).
\end{definition}
\begin{rem}
At abelian category level, an object in $\Rep_q^{\mix}(\check{G})$  is a $\Lambda$-graded vector space with actions of $U^L_q(\check{\mathfrak{n}})$ (with the locally nilpotent condition) and $U^{DK}_q(\check{\mathfrak{n}}^-)$. 

\end{rem}

\subsubsection{Verma modules}
We denote by $\ind_{L\to Dr}$ the left adjoint functor of the forgetful functor from $\Rep_q^{\mix}(\check{G})$ to $U_q^L(\check{\mathfrak{n}})\modu^{\mathsf{loc.nil}}$, and by $\coind_{DK\to Dr}$ the right adjoint functor of the forgetful functor from $\Rep_q^{\mix}(\check{G})$ to $U_q^{DK}(\check{\mathfrak{n}}^-)\modu$. For $\lambda\in \Lambda$, we define Verma modules and coVerma modules in $\Rep_q^{\mix}(\check{G})$ as
\begin{align}\label{quan stan}
    V_\lambda^{\mix}:= &\ind_{L\to Dr}(\mathsf{e}^\lambda),\\
    V_\lambda^{\mix, \vee}:= &\coind_{DK\to Dr}(\mathsf{e}^\lambda).
\end{align}
Here $\mathsf{e}^\lambda$ denotes the 1-dimensional representation of $U_q^L(\check{\mathfrak{n}})$ (or, $U_q^{DK}(\check{\mathfrak{n}}^-)$) in $\Vect_q^\Lambda$, where the action of the torus corresponds to $\lambda$ and the action of the unipotent group is trivial.

Following \cite[5.3.2]{[Ga1]}, we have
\[\mathcal{H}om_{\Rep_q^{\mix}(\check{G})}(V_\lambda^{\mix}, V_\mu^{\mix, \vee})=\left\{\begin{aligned}
    \mathsf{e},\qquad \textnormal{if } \lambda=\mu,\\
    0,\qquad \textnormal{if } \lambda\neq\mu.
\end{aligned}
\right. \]
Furthermore, by construction, the objects $V_\lambda^{\mix}$ are compact and compactly generate $\Rep_q^{\mix}(\check{G})$. So there is a highest weight structure of $\Rep_q^{\mix}(\check{G})$ with standard objects $V_\lambda^{\mix}$ and costandard objects $V_\mu^{\mix, \vee}$.
\subsection{Statement of Theorem \ref{thm 2.1}}\label{statement of the main theorem}


\subsubsection{Avoid small torsion}\label{small torsion} A quadratic form $q$ avoids small torsion ({\cite[1.1.5]{[Ga6]}, \cite[35.1.2 (a)]{[Lus]}}) if for any long coroot $\alpha_l$ of a simple factor of $G$, there is \[\textnormal{ord}(q(\alpha_l)) \geq d_G + 1,\]
where $d_G$ = 1, 2, 3 is the lacing number (i.e., the
maximal number of edges in the Dynkin diagram).

\begin{thm}\label{thm 2.1}
In the setting of D-modules, when $q$ avoids small torsion, there exists a $t$-exact equivalence of highest weight categories
\begin{equation}\label{3.32}
\Whit_q(\Fl^{\omega^\rho}_G)\simeq \Rep_q^{\mix}(\check{G}).
\end{equation}
\end{thm}
\subsubsection{Highest weight category structure}
In Section \ref{standards and costandards}, we will define a collection of standard objects $\Delta_\lambda$ in $\Whit_q(\Fl^{\omega^\rho}_G)$ indexed by $\Lambda$. They are given by $!$-averaging the Wakimoto sheaves. Also, we will define a collection of costandrad objects $\nabla_\lambda\in \Whit_q(\Fl_G^{\omega^\rho})$. We will show that Verma modules $V_\lambda^{\mix}$ and $\Delta_\lambda$ match under the equivalence \eqref{3.32} and similarly for costandards.

\subsubsection{T-structure}
The equivalence $\eqref{3.32}$ is an equivalence at the derived level. We will define a t-structure on $\Whit_q(\Fl^{\omega^\rho}_G)$ in Section \ref{t structure} (see Definition \ref{t2}) using $\Delta_\lambda$, and we will show that it is compatible with the tautological t-structure on $\Rep^{\mix}_q(\check{G})$  under the equivalence.

\section{Factorization algebra and factorization module}\label{Fact geo repl}
Since \cite{[BFS]}, it is known that the category of modules over a Hopf algebra can be realized as the category of factorization modules over a factorization algebra. In this section, we review factorization modules and factorization algebras, and give an equivalent expression of Theorem \ref{thm 2.1} with factorization modules.
\subsection{Factorization algebras and factorization modules}\label{Factorization algebras} 

\begin{definition}
We call a twisted sheaf $\Omega\in \Shv_{\mathcal{G}^\Lambda}(\Conf)$ factorization algebra on $\Conf$ if it is compatible with the factorization property of $\Conf$, i.e., there is an isomorphism
\begin{equation}\label{4.2}
    \add_\Conf^!(\Omega)|_{(\Conf\times \Conf)_{disj}}\simeq \Omega\boxtimes \Omega|_{(\Conf\times \Conf)_{disj}},
\end{equation}
with higher homotopy coherence.
\end{definition}
The above definition makes sense in the $\mathcal{G}^\Lambda$-twisted case because \[\add_{\Conf}^!(\mathcal{G}^\Lambda)|_{(\Conf\times \Conf)_{disj}}\simeq \mathcal{G}^\Lambda\boxtimes \mathcal{G}^\Lambda|_{(\Conf\times \Conf)_{disj}}.\]

It is easy to see that the Verdier dual of a factorization algebra is also a factorization algebra.

\subsubsection{}
Given a factorization algebra $\Omega$ on $\Conf$, we can consider its module category on $\Conf_{\infty\cdot x}$.
\begin{definition}
We call a twisted sheaf $\mathcal{M}\in \Shv_{\mathcal{G}^\Lambda}(\Conf_{\infty\cdot x})$ factorization module over $\Omega$ if it is compatible with the factorization structure of $\Omega$, i.e., there is an isomorphism
\begin{equation}\label{4.3}
    \add_{\Conf_x}^!(\mathcal{M})|_{(\Conf\times \Conf_{\infty\cdot x})_{disj}}\simeq \Omega\boxtimes\mathcal{M}|_{(\Conf\times \Conf_{\infty\cdot x})_{disj}},
\end{equation}
with higher homotopy coherence.
\end{definition}
We denote by $\Omega-\Fact$ the category of factorization modules over $\Omega$.

\subsubsection{Structure of      \ {$\Omega-\Fact$}{}}\label{Structure of A}
Given $\lambda\in \Lambda$,  let $\Conf_{=\lambda\cdot x}$ be the locally closed subscheme of $\Conf_{\infty\cdot x}$ consisting of the points $D= \lambda\cdot x+\sum \lambda_i\cdot x_i$ such that $\lambda_i\in \Lambda^{\textnormal{neg}}\ \textnormal{and}\ x_i\neq x\  \forall\ i$. 

The restriction of (\ref{2.14}) to $\Conf\times \Conf_{=\lambda\cdot x}$ {induces a map \[\add_{\Conf_x}:(\Conf\times \Conf_{=\lambda\cdot x})_{disj}\longrightarrow \Conf_{=\lambda\cdot x}.\]} {We call $\cM\in \Shv_{\cG^\Lambda}(\Conf_{=\lambda\cdot x})$ a factorization module over $\Omega$, if there is $\add_{\Conf_x}^!(\mathcal{M})|_{(\Conf\times \Conf_{=\lambda\cdot x})_{disj}}\simeq \Omega\boxtimes\mathcal{M}|_{(\Conf\times \Conf_{=\lambda\cdot x})_{disj}}$}. We denote by $\Omega-\Fact_{=\lambda}$ the category of factorization modules on $\Conf_{=\lambda\cdot x}$ over $\Omega$.

Let $\mathcal{G}^\Lambda|_{\lambda\cdot x}$ be the fiber of $\mathcal{G}^\Lambda$ at $\lambda\cdot x$. The following lemma is from \cite[Lemma 5.3.5]{[GL1]}.

\begin{lem}\label{4.1.1}
Taking $!$-stalks at $\lambda\cdot x$ defines a $t$-exact equivalence
\begin{equation}
i_{\lambda}^!: \Omega-\Fact_{=\lambda}\simeq \Vect_{\mathcal{G}^\Lambda|_{\lambda\cdot x}}.
\end{equation}
 Here $\Vect_{\mathcal{G}^\Lambda|_{\lambda\cdot x}}$ denotes the category of $\mathcal{G}|_{\lambda\cdot x}$-twisted vector spaces.
\end{lem}
\begin{definition}\label{def 4.4}
Assume that $\Omega$ is in the heart of $\Shv_{\mathcal{G}^\Lambda}(\Conf)$. Consider the perverse\footnotemark\ \footnotetext{Here 'perverse' means this object is concentrated in degree 0.}generator of $\Omega-\Fact_{=\lambda }$, and we denote its $*$ (resp, $!$)-pushforward along the locally closed embedding \[\Conf_{=\lambda\cdot x}\to \Conf_{\infty\cdot x}\] by $\nabla_{\lambda, \Omega}$ (resp, $\Delta_{\lambda, \Omega}$).

$\nabla_{\lambda, \Omega}$ is called costandard object of $\Omega-\Fact$ and $\Delta_{\lambda, \Omega}$ is called standard object.
\end{definition}

By definition, standard objects $\Delta_{\lambda, \Omega}$ are compact and generate $\Omega-\Fact$. Standard objeects $\Delta_{\lambda, \Omega}$ and costandard objects $\nabla_{\lambda, \Omega}$ are perverse.
\subsubsection{}
The Verdier duality functor is well-defined for both $\Delta_{\lambda, \Omega}$ and $\nabla_{\lambda, \Omega}$. Furthermore, we have
\begin{equation}\label{dualfact}
\mathbb{D}^{\verdier}(\Delta_{\lambda, \Omega})\simeq \nabla_{\lambda, \mathbb{D}^{\verdier}(\Omega)}\qquad \forall \lambda\in \Lambda.
\end{equation}
Here $\Delta_{\lambda, \Omega}\in \Shv_{\mathcal{G}^\Lambda}(\Conf_{\infty\cdot x})$ and $\nabla_{\lambda, \mathbb{D}^{\verdier}(\Omega)}\in \Shv_{(\mathcal{G}^\Lambda)^{-1}}(\Conf_{\infty\cdot x})$.

By definition, $\Delta_{\lambda, \Omega}$ and $\nabla_{\mu, \Omega}$ satisfy the following orthogonality property.
\begin{lem}\label{prop 4.1}
For $\lambda,\mu\in \Lambda$, we have
\[\mathcal{H}om_{\Omega-\Fact}(\Delta_{\lambda,\Omega}, \nabla_{\mu, \Omega})=\left\{\begin{aligned}
    \mathsf{e},\qquad \textnormal{if}\ \lambda=\mu,\\
    0, \qquad \textnormal{if}\ \lambda\neq\mu.
\end{aligned}\right. .\]
\end{lem}
\begin{proof}
It follows from the adjointness of $!$-pushforward and $!$-pullback along the locally closed embedding $\Conf_{=\lambda\cdot x}\to \Conf_{\infty\cdot x}$, and Lemma \ref{4.1.1}.
\end{proof}

\subsubsection{}
The tautological t-structure on $\Shv_{\mathcal{G}^\Lambda}(\Conf_{\infty\cdot x})$ gives a t-structure on $\Omega-\Fact$. In fact, by \cite[Proposition 5.4.2]{[GL1]}, we can describe this t-structure on $\Omega-\Fact$ by $\Delta_{\lambda,\Omega}$. To be more precise, an $\Omega$-factorization module $\mathcal{M}$ is coconnective as a twisted sheaf if and only if for any $\lambda\in \Lambda$,
\begin{equation}\label{factt}
    \Hom_{\Omega-\Fact}(\Delta_{\lambda,\Omega}[k], \mathcal{M})=0,\qquad\textnormal{if} \ k> 0.
\end{equation}

\subsection{An explicit description of {$\Omega_q^L$}{}}\label{An explicit description}

In this section, we will recall the factorization algebra $\Omega_q^L$ given in \cite[Section 2.3]{[Ga6]}. It is the category of modules over this factorization algebra that is expected to be equivalent to the Whittaker category on affine flags. 

$\Conf$ has an open subscheme  $\overset{\circ}{\Conf}$  removing all diagonals. A point $ D= \mathop{\sum}_k \lambda_k\cdot x_k$ belongs to $\overset{\circ}{\Conf}$ if and only if the coefficient of any point $x_k$ is a negative coroot. If $\lambda=-\sum_i n_i\cdot\alpha_i$, then the $\lambda$ connected component $\overset{\circ}{\Conf^\lambda}$  is isomorphic to $\prod_i \overset{\circ}{X^{(n_i)}}$. Here $\overset{\circ}{X^{(n_i)}}$ classifies unordered $n_i$ different points in $X$.

Following \cite[Section 17.1.2]{[GL1]}, $\mathcal{G}^\Lambda|_{\overset{\circ}{\Conf}}$ is canonically trivialized. In particular,
\begin{equation}\label{4.25}
\Shv_{\mathcal{G}^\Lambda}(\overset{\circ}{\Conf})\simeq \Shv(\overset{\circ}{\Conf}).
\end{equation}

The product of sign local systems on each $\overset{\circ}{X^{(n_i)}}$ gives rise to a factorization algebra on $\overset{\circ}{\Conf}$. Under the equivalence \eqref{4.25}, it can be regarded as a twisted factorization algebra on $\overset{\circ}{\Conf}$. We denote it by $\overset{\circ}{\Omega}_q$, where $q$ is the quadratic form in \eqref{quad f}.
\subsubsection{}
By factorization (i.e., \eqref{4.2}), we only need to indicate how to extend $\Omega^L_q$ from its restriction on $\Conf^\lambda\setminus X$ to $\Conf^\lambda$, for any $\lambda\in \Lambda^{\textnormal{neg}}\setminus \{0\}$. Here $X$ embeds into $\Conf^\lambda$ by assigning $x$ to $\lambda\cdot x$.

We denote by $\jmath_\lambda$ the open embedding from $\Conf^{\lambda}\setminus X$ to $\Conf^\lambda$ and by $l$ the length function of the Weyl group.

\begin{definition}\label{def lus}
The factorization algebra $\Omega^L_q$ is defined inductively as follows:
\begin{enumerate}[label=(\arabic*)]
    \item If $\lambda= w(\rho)-\rho$ and $l(w)=2$,
then \[\jmath_{\lambda,!}\circ \jmath_\lambda^!(\Omega^L_q)\overset{\sim}{\longrightarrow} \Omega^L_q.\]
    \item If $\lambda=  w(\rho)-\rho$ and $l(w)\leq 3$,

\[H^0(\jmath_{\lambda,!}\circ \jmath_\lambda^!(\Omega^L_q))\simeq \Omega^L_q.\]
    \item If $\lambda$ is not of the form $w(\rho)-\rho$, then \[\Omega^L_q\simeq \jmath_{\lambda, !*}\circ \jmath_\lambda^!(\Omega^L_q).\]
\end{enumerate}

\end{definition}
{The following lemma is indicated in \cite[Section 29]{[GL1]}, see \cite[Theorem 1.2.1]{[CF]} for a precise statement and proof.}
\begin{lem}\label{replacement}
In the D-module setting and the Betti setting, there is
\begin{equation}
     \Rep_q^{\mix}(\check{G}):= Z_{Dr, \Vect_q^\Lambda}(U_q^{L}(\check{\mathfrak{n}})\modu^{\mathsf{loc.nil}})\simeq \Omega_{q}^L-\Fact.
\end{equation}
Furthermore, under the above equivalence, $V_\lambda^{\mix}$ (resp. $V_\lambda^{\mix,\vee}$) corresponds to $\Delta_{\lambda, \Omega_q^L}$  (resp. $\nabla_{\lambda, \Omega_q^L}$). 
\end{lem}
\begin{proof}
By Lefschetz principle and Riemann-Hilbert correspondence, we only need to prove the claim in the setting of constructible sheaves with coefficients in $\mathbb{C}$ in classical topology. 

Following \cite[Section 29.5.1]{[GL1]}, one can associate a factorization algebra with a Hopf algebra $A$. Furthermore, the relative Drinfeld center of $A$-mod is equivalent to the category of factorization modules over the corresponding factorization algebra. Apply it to our case, there is an equivalence of categories
\begin{equation}
    \Rep_q^\mix(\check{G})\simeq \Omega_{q, \textnormal{quant}}^L-\Fact.
\end{equation}
Here $\Omega_{q, \textnormal{quant}}^L$ denotes the $\mathcal{G}^\Lambda$-twisted factorization algebra associated with $U_q^L(\check{\mathfrak{n}})$.
for a certain  factorization algebra . 

However, by the Verdier dual of \cite[Section 2.3.8, Theorem 3.6.2]{[Ga6]}, there is $\Omega_{q, \textnormal{quant}}^L\simeq \Omega^L_q$.
\end{proof}

\subsection{Restatement of Theorem \ref{thm 2.1}}\label{Resta}
The following theorem is equivalent to Theorem \ref{thm 2.1}.
\begin{thm}\label{main theorem 1}
In the setting of D-modules, when $q$ avoids small torsion, there is a functor $F^L$ which establishes a t-exact\footnotemark\ equivalence

\[F^L: \Whit_q(\Fl^{\omega^\rho}_{G})\overset{\sim}{\longrightarrow} \Omega_{q}^{L}-\Fact,\]
and preserves standard objects and costandard objects. \footnotetext{The t-structure on $\Whit_q(\Fl^{\omega^\rho}_{G})$ will be given in Section \ref{t structure}.}
\end{thm}
\begin{rem}
The above theorem holds in the Betti setting if we replace the Whittaker category by the Kirillov model.
\end{rem}

\section{Standard object and duality of Whittaker category}\label{local Whittaker category}

The principal goal of this section is to construct standard objects of $\Whit_q(\Fl^{\omega^\rho}_G)$ (Section \ref{standards and costandards}) and define the duality functor of $\Whit_q(\Fl^{\omega^\rho}_G)$ using a lemma from \cite{[Ras2]} (Section \ref{coinvariant}). They will play an important role in the proof of our main theorem.
\subsection{Relevant orbits}\label{relevant}
To study $\Whit_q(\Fl_G^{\omega^\rho})$, we should at first study  when a $N(\mathcal{K})^{\omega^\rho}$-orbit admits non-zero Whittaker sheaves support on it. 

It is known that the $N(\mathcal{K})^{\omega^\rho}$-orbits in $\Fl_G^{\omega^\rho}$ are indexed by the extended affine Weyl group $W^{\ext}:= W\ltimes \Lambda$. For any $\widetilde{w}\in W^{\ext}$, we denote by $S_{\Fl}^{\widetilde{w}}$ the $N(\mathcal{K})^{\omega^\rho}$-orbit in $\Fl_G^{\omega^\rho}$ of $\widetilde{w}\cdot I^{\omega^\rho}/I^{\omega^\rho}$. Let $\overline{S}_{\Fl}^{\widetilde{w}}$ denote the closure of $S_{\Fl}^{\widetilde{w}}$ in $\Fl_G^{\omega^\rho}$. 

The orbits which admit non-zero $\mathcal{G}^G$-twisted Whittaker sheaves are called relevant orbits. Since $N(\mathcal{K})^{\omega^\rho}$ is ind-pro-unipotent and $\mathcal{G}^G$ on $\Fl_G^{\omega^\rho}$ is $\mathcal{G}^G$-equivariant with respect to the action of $G(\mathcal{K})^{\omega^\rho}$, the gerbe $\mathcal{G}^G$ on any single $N(\mathcal{K})^{\omega^\rho}$-orbit admits a $N(\mathcal{K})^{\omega^\rho}$-equivariant trivialization. As a result, the necessary and sufficient condition for a $N(\mathcal{K})^{\omega^\rho}$-orbit ${S}^{\widetilde{w}}_{\Fl}$ to be relevant is that it is relevant in the non-twisted case. Namely,
\begin{equation}\label{5.1.1}
    \Stab_{N(\mathcal{K})^{\omega^\rho}}(\widetilde{w}I^{\omega^\rho}/I^{\omega^\rho})\subset \Ker(\chi).
\end{equation}

\begin{definition}
Given an element $\widetilde{w}\in W^{\ext}$, we denote by $l(\widetilde{w})$ the dimension of its Iwahori orbit in the affine flags (or the length of $\widetilde{w}$). To be more precise, if $\widetilde{w}=t^\lambda w$, then
\begin{equation}\label{length}
    l(\widetilde{w})=\sum_{\check{\alpha}\in \check{\Delta}^+, w^{-1}(\check{\alpha})>0}|\langle\check{\alpha}, \lambda\rangle|+\sum_{\check{\alpha}\in \check{\Delta}^+,w^{-1}(\check{\alpha})<0}|\langle\check{\alpha},\lambda\rangle+1|.
\end{equation}

\end{definition}

\begin{prop}
A $N(\mathcal{K})^{\omega^\rho}$-orbit $S_{\Fl}^{\widetilde{w}}\subset \Fl^{\omega^\rho}_G$ is relevant if and only if $t^\rho \widetilde{w}$ is the maximal length element in the left coset  $Wt^\rho \widetilde{w}\subset W^{\ext}$.

In particular, for any dominant coweight $\lambda\in \Lambda^+$, $t^\lambda\in W^\ext$ is relevant.
\end{prop}
\begin{proof}
Denote by \[{\widetilde{\check{\Delta}}}^+=\{ \check{\alpha}+ n\delta|\ \check{\alpha}\in \check{\Delta},\ n\geq 0\ \textnormal{if}\ \check{\alpha} \textnormal{ is positive, and }n\geq 1 \textnormal{if } \check{\alpha} \textnormal{ is negative} \}\cup \{n\delta|n\geq 0\}\] the set of  the roots of $G(\mathcal{K})^{\omega^\rho}$ corresponding to $I^{\omega^\rho}$. Here $\delta$ is the positive imaginary root generator. We denote by $\check{\Delta}^+$ (resp. $\check{\Pi}^+$) the set of positive (resp. positive simple) roots of $G$. 

If the $N(\mathcal{K})^{\omega^\rho}$-orbit of $\widetilde{w}\cdot I^{\omega^\rho}/I^{\omega^\rho}\in \Fl_G^{\omega^\rho}$ is relevant, we need the formula \eqref{5.1.1} holds. By a straightforward calculation
\begin{equation}
    \begin{split}
        &\Stab_{N(\mathcal{K})^{\omega^\rho}}(\widetilde{w}\cdot I^{\omega^\rho}/I^{\omega^\rho})= \widetilde{w}I^{\omega^\rho} \widetilde{w}^{-1}\cap N(\mathcal{K})^{\omega^\rho}\subset \Ker(\chi)\\
\Leftrightarrow\quad &\widetilde{w}({\widetilde{\check{\Delta}}}^+)\cap (\check{\Pi}^+-\delta)=\emptyset\\
\Leftrightarrow\quad & \widetilde{w}({\widetilde{\check{\Delta}}}^+)\cap t^{-\rho}(\check{\Pi}^+)=\emptyset\\
\Leftrightarrow\quad & t^{\rho}\widetilde{w}({\widetilde{\check{\Delta}}}^+)\cap \check{\Pi}^+=\emptyset\\
\Leftrightarrow\quad & -\{\check{\Pi}^+\}\subset t^{\rho}\widetilde{w}({\widetilde{\check{\Delta}}}^+)\\
\Leftrightarrow\quad & \widetilde{w}^{-1}t^{-\rho}(\check{\Delta}^+)< 0 .
\end{split}
\end{equation}
According to \cite[Lemma 3.11 a.]{[Kac]}, it means that for any simple reflection $r_i$ of $W$, there is $l(\widetilde{w}^{-1} t^{-\rho} r_i)\leq l(\widetilde{w}^{-1} t^{-\rho})$. We conclude that for any simple reflection $r_i$ of $W$, there is $l(r_i t^{\rho} \widetilde{w})\leq l(t^{\rho} \widetilde{w})$. It means $t^{\rho} \widetilde{w}$ is the unique maximal length element in $W t^{\rho} \widetilde{w}\subset W^{\ext}$.
\end{proof}
\begin{rem}
Relevant orbits of $\Fl_G$ are naturally indexed $\Lambda$. One can check that 
\begin{equation}
    \begin{split}
        \phi: W^{\ext}\to W^{\ext}\\
        \widetilde{w}\mapsto t^{-\rho} \widetilde{w}_l,
    \end{split}
\end{equation}
induces a bijection between $\Lambda$ and the set $\{\widetilde{w},\ \widetilde{w}\ \textnormal{is relevant}\}$. Here $\widetilde{w}_l$ denotes the unique maximal length element in the left coset $W\widetilde{w}$. For example, if $\widetilde{w}=t^0$, then $\phi(\widetilde{w})= t^{-\rho}w_0$; if $\widetilde{w}=t^{\rho+\lambda}$ and $\lambda$ is dominant, then $\phi(\widetilde{w})= t^\lambda$.
\end{rem}
\begin{prop}\label{prop 5.2}
The category of twisted Whittaker sheaves on $S^{\widetilde{w}}_{\Fl}$ is equivalent to $\Vect$ or $0$ depends on if $\widetilde{w}$ is relevant or not, i.e.,
$$ \Whit_q(S_{\Fl}^{\widetilde{w}})\simeq\left\{
\begin{aligned}
&\Vect, \qquad  &\textnormal{if}\ \widetilde{w}\ \textnormal{is relevant,}\\
&0,\qquad  &\textnormal{otherwise}.
\end{aligned}
\right.
$$
\end{prop}
\begin{proof}

There is a $N(\mathcal{K})^{\omega^\rho}$-equivariant trivialization of $\mathcal{G}^G$ on $S_{\Fl}^{\widetilde{w}}$, we have $\Whit_q(S_{\Fl}^{\widetilde{w}})\simeq \Whit(S_{\Fl}^{\widetilde{w}})$. Since $\Stab_{N(\mathcal{K})^{\omega^\rho}}(\widetilde{w}\cdot I^{\omega^\rho}/I^{\omega^\rho})$ is a connected pro-unipotent group, $\Whit(S_{\Fl}^{\widetilde{w}})$ is a full subcategory of $\Shv^{N(\mathcal{K})^{\omega^\rho},\chi}(N(\mathcal{K})^{\omega^\rho})$. The latter category is equivalent to $\Vect$, we only need to show that the fully faithful embedding \[\Whit(S_{\Fl}^{\widetilde{w}})\to \Shv^{N(\mathcal{K})^{\omega^\rho},\chi}(N(\mathcal{K})^{\omega^\rho})\] is actually an equivalence if $\widetilde{w}$ is relevant and is $0$ otherwise. 

If $\widetilde{w}$ is relevant,  $\Stab_{N(\mathcal{K})^{\omega^\rho}}(\widetilde{w}\cdot I^{\omega^\rho}/I^{\omega^\rho})\subset \Ker (\chi)$, so any Whittaker sheaf on $N(\mathcal{K})^{\omega^\rho}$ descends to a Whittaker sheaf on $S_{\Fl}^{\widetilde{w}}\simeq N(\mathcal{K})^{\omega^\rho}/ \Stab_{N(\mathcal{K})^{\omega^\rho}}(\widetilde{w}\cdot I^{\omega^\rho}/I^{\omega^\rho})$. If $\widetilde{w}$ is not relevant, we should prove that for any $\mathcal{F}\in \Whit(S_{\Fl}^{\widetilde{w}})$, there is $\mathcal{F}=0$. Since the action of $N(\mathcal{K})^{\omega^\rho}$ is transitive, we only need to show that the stalk of $\mathcal{F}$ is $0$ at $\widetilde{w}\cdot I^{\omega^\rho}/I^{\omega^\rho}$. It follows immediately from the fact that this fiber is equivariant with respect to $\Stab_{N(\mathcal{K})^{\omega^\rho}}(\widetilde{w}\cdot I^{\omega^\rho}/I^{\omega^\rho})$ against a non-trivial character.


\end{proof}
\begin{definition}\label{ver st}
Assume that $\widetilde{w}$ is relevant, we define

\[\Delta_{\widetilde{w}}^{\ver}:= \Av_!^{N(\mathcal{K})^{\omega^\rho},\chi}(\delta_{\widetilde{w}})[-l(w_0 t^\rho\widetilde{w})+ l(t^{-\rho}w_0)],\]
and
\[\nabla_{\widetilde{w}}^{\ver}:= \bar{j}_{{\widetilde{w}}, \Fl,*}\circ \bar{j}_{{\widetilde{w}}, \Fl}^!(\Delta_{\widetilde{w}}^{\ver}),\]
where $j_{{\widetilde{w}}, \Fl}$ is the locally closed embedding \[ j_{{\widetilde{w}}, \Fl}: S_{\Fl}^{\widetilde{w}}\to \Fl^{\omega^\rho}_G.\]
{They are defined up to tensoring by a line, i.e., depending on the trivialization of $\cG^G$ at $\widetilde{w}\in \Fl_G^{\omega^{\rho}}$.}

\begin{rem}
Here the superscript ''${\ver}$'' means \textit{Verma}. We expect, under Conjecture \ref{conj 2}, $\{\Delta_{\widetilde{w}}^{\ver}\}$ correspond to Verma modules in the category of modules over Kac-Moody Lie algebra.
\end{rem}
\end{definition}
The object $\Delta_{\widetilde{w}}^{\ver}$ is the $!$-pushforward of the generator of $\Whit_q(S_{\Fl}^{\widetilde{w}})$, and $\nabla_{\widetilde{w}}^{\ver}$ is the $*$-pushforward of the generator of $\Whit_q(S_{\Fl}^{\widetilde{w}})$. They have the following properties.

\begin{prop}\label{naive standard generate}\leavevmode
\begin{enumerate}[label={\upshape(\arabic*)}]
    \item $\{\Delta_{\widetilde{w}}^{\ver}, \widetilde{w}\ \textnormal{is relevant}\}$ (resp. $\{\nabla_{\widetilde{w}}^{\ver}, \widetilde{w}\ \textnormal{is relevant}\}$) compactly generate $\Whit_q(\Fl_G^{\omega^\rho})$.
    \item \[\mathcal{H}om_{\Whit_q(\Fl_G^{\omega^\rho})}(\Delta_{\widetilde{w}}^{\ver}, \nabla_{\widetilde{w}'}^{\ver})\simeq\left\{ 
\begin{aligned}
    0,\qquad \textnormal{if}\ \widetilde{w}\neq \widetilde{w}',\\
    \mathsf{e}, \qquad \textnormal{if}\ \widetilde{w}= \widetilde{w}'.
\end{aligned}
\right.  \] 
\end{enumerate}
\end{prop}
\begin{proof}
Since $j_{\widetilde{w}, \Fl,!}$ is the left adjoint functor of a continuous functor (i.e., $j^!_{\widetilde{w}, \Fl}$), it preserves compactness. In particular, $\Delta_{\widetilde{w}}^{\ver}$ is compact. Note that for $\mathcal{F}\in \Whit_q(\Fl_G)$, $\mathcal{H}om_{\Whit_q(\Fl_G^{\omega^\rho})}(\Delta_{\widetilde{w}}^{\ver}, \mathcal{F})$ is isomorphic to the $!$-stalks of $\mathcal{F}$ at $\widetilde{w}$ up to a shift. Hence, if $\mathcal{H}om_{\Whit_q(\Fl_G^{\omega^\rho})}(\Delta_{\widetilde{w}}^{\ver}, \mathcal{F})=0$ for any relevant $\widetilde{w}$, then $\mathcal{F}=0$.

To show the claim for $\{\nabla_{\widetilde{w}}^{\ver}, \widetilde{w}\ \textnormal{is relevant}\}$, note that the closure of any $N(\mathcal{K})^{\omega^\rho}$-orbit in $\Fl^{\omega^\rho}_G$ only contains finite many relevant orbits. Hence, any $\nabla_{\widetilde{w}}^{\ver}$ is a finite extension of objects in $\{\Delta_{\widetilde{w}}^{\ver}[k], \widetilde{w}\ \textnormal{is relevant}, k\in \mathbb{Z}\}$. In particular, $\nabla_{\widetilde{w}}^{\ver}$ is compact. Similarly, any $\Delta_{\widetilde{w}}^{\ver}$ is a finite extension of objects in $\{\nabla_{\widetilde{w}}^{\ver}[k], \widetilde{w}\ \textnormal{is relevant}, k\in \mathbb{Z}\}$. Hence, the objects $\nabla_{\widetilde{w}}^{\ver}$ compactly generate $\Whit_q(\Fl_G^{\omega^\rho})$ as the objects $\Delta_{\widetilde{w}}^{\ver}$ do.

The claim \textup{(2)} directly follows from the adjointness of $!$-pushforward and $!$-pullback, and Proposition \ref{prop 5.2}.
\end{proof}

As a corollary, an object $\mathcal{F}\in \Whit_q(\Fl_G^{\omega^\rho})$ is compact if and only if $\mathcal{F}$ is supported on {finitely many ${S}^{\widetilde{w}}_{\Fl}$}, and its restriction is compact in  {$\Whit_q({S}_{\Fl}^{\widetilde{w}})$, for any $\widetilde{w}$}.

\begin{rem}
We can define a highest weight category structure on $\Whit_q(\Fl_G^{\omega^\rho})$ with standards $\{\Delta_{\widetilde{w}}^{\ver}\}$ and costandards $\{\nabla_{\widetilde{w}}^{\ver}\}$. It is the highest weight category structure in \cite{[CDR]}. However, it is different from the one used in this paper. For generic $q$, we expect that they are essentially the same (up to a convolution).
\end{rem}


\subsection{Right equivariant sheaf}\label{Right monodromic D-modules}
Recall the main theorem of \cite{[AB]},
\[\Whit(\Fl_G)\simeq \QCoh(\check{\mathfrak{n}}/\check{B})\simeq \Rep^{\mix}(\check{G}).\]
It is not only an equivalence of plain categories, but also compatible with highest weight category structures on both sides. The standard objects in $\Rep^{\mix}(\check{G})$ are Verma modules $V^{\mix}_\lambda$ and the standard objects in $\Whit(\Fl_G)$ are given by the !-averaging of the BMW\footnotemark\ sheaves ${{J}}_\lambda$ \footnotetext{$BMW$ stands for R.Bezrukavnikov-I.Mirkovi\'c-M.Wakimoto}(i.e., Wakimoto sheaves). Hence, in the twisted case, we need to  define the twisted BMW sheaves ${{J}}_\lambda$ and use them to define standard objects $\Delta_\lambda$ in $\Whit_q(\Fl_G^{\omega^\rho})$ by $!$-averaging.
\subsubsection{}
Let $I^0$ be the unipotent radical of the $\omega^\rho$-twisted Iwahori subgroup $I^{\omega^\rho}$, and let $T^{\omega^\rho}$ be $I^{\omega^\rho}/I^0$. When there is no twisting, ${{J}}_\lambda$ is Iwahori-equivariant. However, there are 'much fewer' Iwahori-equivariant objects in the twisted case. 

Indeed, since $I^0$ is pro-unipotent, there is a unique (up to a non-canonical isomorphism) $I^0$-equivariant trivialization of the gerbe $\mathcal{G}^G$ on any Iwahori orbit \[\Fl_G^{\widetilde{w}}:= I^{\omega^\rho}\cdot \widetilde{w} \cdot I^{\omega^\rho}/I^{\omega^\rho}\] of $\Fl^{\omega^\rho}_G$. In the twisted case, the $I^0$-equivariant trivialization of $\mathcal{G}^G$ on $\Fl_G^{\widetilde{w}}$ is not necessarily  Iwahori-equivariant. Instead, it is equivariant with respect to a certain character of $T^{\omega^\rho}$.  For example,  when the quadratic form $q$ is generic, there is no $I^{\omega^\rho}$-equivariant sheaf on $\Fl_G^{t^\lambda w}$ unless $\lambda=0$. 

The solution to fix this problem is to consider $I^0$-equivariant sheaves on $\widetilde{\Fl}:=G(\mathcal{K})^{\omega^\rho}/I^0$, rather than Iwahori-equivariant sheaves on $\Fl_G$.


\subsubsection{}\label{5.2.2}
The exact sequence
\[1\to I^0\to I^{\omega^\rho}\to T^{\omega^\rho}\to 1\]
is split. Hence, we may consider the right action of $T^{\omega^\rho}$ on $\widetilde{\Fl}:=G(\mathcal{K})^{\omega^\rho}/I^0$. In order to define BMW sheaves in the twisted case, we need to consider sheaves on $\widetilde{\Fl}$ which are right $T^{\omega^\rho}$-equivariant against a character. This idea appears in \cite{[Be3]} and \cite{[LY]}.

Let $b(-,-)$ be the symmetric bilinear form associated with the quadratic form $q$, i.e., in the D-module setting,
\begin{equation}
    b(\lambda_1,\lambda_2):= q(\lambda_1+\lambda_2)- q(\lambda_1)-q(\lambda_2);
\end{equation}
in the Betti setting and the $\ell$-adic setting, 
\begin{equation}
    b(\lambda_1,\lambda_2):= q(\lambda_1+\lambda_2)\cdot q(\lambda_1)^{-1}\cdot q(\lambda_2)^{-1}.
\end{equation}
 In the D-module setting (resp. Betti setting and $\ell$-adic setting), we denote by 
    \begin{equation}
    b_\lambda:=b(\lambda,-)\colon\ \Lambda\to \mathsf{k}/\mathbb{Z}\ (\textnormal{resp.}\ \mathsf{e}^{\mathsf{torsion},\times}(-1)),
\end{equation}
the associated character of $T^{\omega^\rho}$. With some abuse of notations, we denote by $b_\lambda$ the associated Kummer sheaf on $T^{\omega^\rho}$. Its pullback to $I^{\omega^\rho}$ is also denoted by $b_\lambda$.
\begin{definition}
We define \[\Shv_{\mathcal{G}^G}(\widetilde{\Fl})^{I,\lambda}_\mu:= \Shv_{\mathcal{G}^G}(I^{\omega^\rho},b_\lambda\backslash \widetilde{\Fl}/ T^{\omega^\rho},b_\mu)\] as the category of left $(I^{\omega^\rho},b_\lambda)$-equivariant and right $(T^{\omega^\rho},b_\mu)$-equivariant twisted sheaves on $\widetilde{\Fl}$. 

Similarly, we define \[\Shv_{\mathcal{G}^G}(\widetilde{\Fl})^{I^0}_\mu:=\Shv_{\mathcal{G}^G}(I^0\backslash \widetilde{\Fl}/ T^{\omega^\rho},b_\mu) \] as the category of left $I^0$-equivariant and right $(T^{\omega^\rho},b_\mu)$-equivariant twisted sheaves on $\widetilde{\Fl}$.

When $\mu=0$, we omit $\mu$ in the notation. In this case, we can realize right $T^{\omega^\rho}$-equivariant objects on $\widetilde{\Fl}$ as  objects on $\Fl_G^{\omega^\rho}$.
\end{definition}




Let $\widetilde{\Fl}^{\widetilde{w}}$ be the preimage of ${\Fl}^{\widetilde{w}}_G$ in $\widetilde{\Fl}$. {We have equivalences
\[\Shv_{\mathcal{G}^G}(\widetilde{\Fl}^{\widetilde{w}})_{\mu}^{I^0}\simeq \Shv_{\mathcal{G}^G}((\widetilde{w}I^{\omega^\rho}/I^0)/T^{\omega^\rho}, b_\mu)\simeq \Shv_{\mathcal{G}^G}(\widetilde{w}I^{\omega^\rho}/I^{\omega^\rho}, b_\mu)\simeq \Shv_{\mathcal{G}^G}(\widetilde{w}T(\cO)^{\omega^\rho}/T(\cO)^{\omega^\rho}, b_\mu).\]
Given an identification of $\Shv_{\mathcal{G}^G}(\widetilde{w}T(\cO)^{\omega^\rho}/T(\cO)^{\omega^\rho}, b_\mu)$ with $\Vect$ which preserves the cohomological degree of $!$-fiber,} {w}e denote by $(c_{\widetilde{w}})_{\mu}$ the twisted sheaf on $\widetilde{\Fl}^{\widetilde{w}}$ {corresponding to $\mathsf{e}\in \Vect$. Different identifications will change the resulting twisted sheaf $(c_{\widetilde{w}})_{\mu}$ by tensoring by a line in cohomological degree $0$, so we can regard $(c_{\widetilde{w}})_{\mu}$ as an object defined up to tensoring by a line. } 
\begin{rem}
In the case $\mu=0$, $\Shv_{\mathcal{G}^G}(\widetilde{\Fl}^{\widetilde{w}})_{\mu}^{I^0}\simeq \Shv_{\mathcal{G}^G}({\Fl^{\widetilde{w}}_G})^{I^0}\simeq \Shv({\Fl^{\widetilde{w}}_G})^{I^0}$, and $(c_{\widetilde{w}})_{\mu}=c_{\widetilde{w}}$ is just the (twisted) {dualizing sheaf} on $\Fl_G^{\widetilde{w}}$ {with respect to the unique (up to a non-canonical isomorphism) $I^0$-equivariant trivialization on ${\Fl^{\widetilde{w}}_G}$}.
\end{rem}
\begin{definition}
Let $({{J}}_{\widetilde{w},!})_{\mu}\ (resp, ({{J}}_{\widetilde{w},*})_{\mu})$ be the ! (resp, *)-extension of $(c_{\widetilde{w}})_{\mu}[{-}l(\widetilde{w})]$ along the locally closed embedding \[ \widetilde{\Fl}^{\widetilde{w}}\to \widetilde{\Fl}.\]

\end{definition}



\subsection{Convolution product}\label{convolution product}
 Denote by
 \begin{equation}
     \widetilde{\pi}: G(\mathcal{K})^{\omega^\rho}\to G(\mathcal{K})^{\omega^\rho}/I^0\simeq \widetilde{\Fl}
 \end{equation}  the natural projection from $G(\mathcal{K})^{\omega^\rho}$ to $\widetilde{\Fl}$.

For a right $(I^{\omega^\rho}, b_\lambda)$-equivariant sheaf $\mathcal{F}_1\in \Shv_{\mathcal{G}^G}(\widetilde{\Fl})_\lambda$ and a left $(I^{\omega^\rho}, b_\lambda)$-equivariant and right $(I^{\omega^\rho}, b_\eta)$-equivariant sheaf $\mathcal{F}_2\in \Shv_{\mathcal{G}^G}(\widetilde{\Fl})^{I,\lambda}_{\eta}$. Consider the following diagram:
\begin{equation*}
\xymatrix{
  \cdots \ar@<-1ex>[r] \ar@<1ex>[r]\ar[r]&G(\mathcal{K})^{\omega^\rho}\mathop{\times} I^{\omega^\rho}\times \widetilde{\Fl}\ar@<-.5ex>[r] \ar@<.5ex>[r] & G(\mathcal{K})^{\omega^\rho}\times \widetilde{\Fl}\ar[r]\ar[r]^\pi&  G(\mathcal{K})^{\omega^\rho}\mathop{\times}\limits^{I^{\omega^\rho}} \widetilde{\Fl}.
}
\end{equation*}
Since the equivariant conditions translate to the descent condition, the external product $\widetilde{\pi}^!(\mathcal{F}_1)\boxtimes \mathcal{F}_2$ on $G(\mathcal{K})^{\omega^\rho}\times \widetilde{\Fl}$ descends to a twisted sheaf $\widetilde{\pi}^!(\mathcal{F}_1)\widetilde{\boxtimes} \mathcal{F}_2$ on $G(\mathcal{K})^{\omega^\rho}\mathop{\times}\limits^I \widetilde{\Fl}$, such that \[\pi^!(\widetilde{\pi}^!(\mathcal{F}_1)\widetilde{\boxtimes} \mathcal{F}_2)\simeq \widetilde{\pi}^!(\mathcal{F}_1)\boxtimes \mathcal{F}_2.\]

By the multiplicative property of $\mathcal{G}^G$, the pullback of $\mathcal{G}^G$ along the multiplication map \[m: G(\mathcal{K})^{\omega^\rho}\times G(\mathcal{K})^{\omega^\rho}/I^0\to G(\mathcal{K})^{\omega^\rho}/I^0\] is $\mathcal{G}^G\boxtimes \mathcal{G}^G$. In particular, the pushforward of a $\mathcal{G}^G\boxtimes \mathcal{G}^G$-twisted sheaf along the multiplication map is $\mathcal{G}^G$-twisted.
\begin{definition}

Set
\begin{equation}
    \mathcal{F}_1\mathop{\star}\limits \mathcal{F}_2:= m_*((\widetilde{\pi}^!(\mathcal{F}_1)\widetilde{\boxtimes} \mathcal{F}_2))\in \Shv_{\mathcal{G}^G}(\widetilde{\Fl})_{\eta}.
\end{equation}
\end{definition}


\begin{rem}
Since $m$ is (ind-) proper, it does not matter whether we consider $!$- or $*$-pushforward of $m$.
\end{rem}

 We denote by $\overline{\widetilde{w}}$ the image of $\widetilde{w}$ under the first projection $W^{\ext}= \Lambda\rtimes W\to \Lambda$. For $\widetilde{w}=t^\lambda w\in W^{\ext}$, $\overline{\widetilde{w}}=\lambda$.
\begin{rem}
$({{J}}_{\widetilde{w},?})_{\mu}$ ( ?=  ! or *) is left $(I^{\omega^\rho}, b_{\overline{\widetilde{w}t^\mu}})$-equivariant and right $(T^{\omega^\rho}, b_\mu)$-equivariant.
\end{rem} 

In order to define the twisted BMW sheaf for any coweight $\lambda\in \Lambda$, we need the following lemma which is an analog of \cite[Lemma 8]{[AB]}, and \cite[Lemma 3.4, 3.5]{[LY]}.

\begin{lem}\label{BMW}
For $\widetilde{w}, \widetilde{w}'\in W^{\ext}$ and $\mu\in \Lambda$, if $l(\widetilde{w}\widetilde{w}')=l(\widetilde{w})+l(\widetilde{w}')$,  then 

\textup{(1)} \[({{J}}_{\widetilde{w},!})_{\overline{\widetilde{w}'t^\mu}}\star ({{J}}_{\widetilde{w}',!})_{\mu}\simeq ({{J}}_{\widetilde{w}\widetilde{w}',!})_{\mu},\]
\[({{J}}_{\widetilde{w},*})_{\overline{\widetilde{w}'t^\mu}}\star ({{J}}_{\widetilde{w}',*})_{\mu}\simeq ({{J}}_{\widetilde{w}\widetilde{w}',*})_{\mu}.\]

\textup{(2)} $ ({{J}}_{\widetilde{w},!})_{\overline{\widetilde{w}^{-1}\mu}}\star ({{J}}_{\widetilde{w}^{-1},*})_{\mu}\simeq ({{J}}_0)_\mu\simeq  ({{J}}_{\widetilde{w},*})_{\overline{\widetilde{w}^{-1}\mu}}\star ({{J}}_{\widetilde{w}^{-1},!})_{\mu}$.

\end{lem}

\begin{proof}

The proof is similar to the non-twisted case. Here we sketch the proof.

For (1), we only prove the first claim, the second one follows from the same argument.


By the Cartan decomposition of $G(\mathcal{K})^{\omega^\rho}$ by the Iwahori subgroup $I^{\omega^\rho}$, the multiplication map 
\begin{equation}\label{mult m}
    G(\mathcal{K})^{\omega^\rho}\mathop{\times}\limits^{I^{\omega^\rho}} G(\mathcal{K})^{\omega^\rho}/I^0\to G(\mathcal{K})^{\omega^\rho}/I^0
\end{equation} is an isomorphism after restricting to $I^{\omega^\rho} \widetilde{w} I^{\omega^\rho}\mathop{\times}\limits^{I^{\omega^\rho}} I^{\omega^\rho} \widetilde{w}' I^{\omega^\rho}/I^0\to I^{\omega^\rho} \widetilde{w}\widetilde{w}'I^{\omega^\rho}/I^0$. Note that the trivializations of gerbes on both sides are compatible. Now the first assertion follows from the fact that both $ ({{J}}_{\widetilde{w},!})_{\overline{\widetilde{w}'t^\mu}}\star ({{J}}_{\widetilde{w}',!})_{\mu}$ and  $({{J}}_{\widetilde{w}\widetilde{w}',!})_{\mu}$ have zero $*$-stalks outside $I^{\omega^\rho} \widetilde{w}\widetilde{w}'I^{\omega^\rho}/I^{\omega^\rho}\subset \Fl_G^{\omega^\rho}$.

To prove \textup{(2)}, by \textup{(1)}, we can reduce the question to the case when $\widetilde{w}$ is a simple reflection. In this case, $\overline{I^{\omega^\rho} \widetilde{w}I^{\omega^\rho}/I^{\omega^\rho}}\subset \Fl_G^{\omega^\rho}$ is isomorphic to $\mathbb{P}^1$.

The support of $({{J}}_{\widetilde{w},!})_{\overline{\widetilde{w}^{-1}\mu}}\star ({{J}}_{\widetilde{w}^{-1},*})_{\mu}$ is contained in $\overline{I^{\omega^\rho}\widetilde{w}I^{\omega^\rho}\cdot I^{\omega^\rho} \widetilde{w}^{-1}I^{\omega^\rho}/I^0}$, which is equal to the disjoint union $I^{\omega^\rho}\widetilde{w}I^{\omega^\rho}/I^0\sqcup I^{\omega^\rho}/I^0$. By the equivariance property of $({{J}}_{\widetilde{w},!})_{\overline{\widetilde{w}^{-1}\mu}}\star ({{J}}_{\widetilde{w}^{-1},*})_{\mu}$, we should prove that the $!$-stalks of $ ({{J}}_{\widetilde{w},!})_{\overline{\widetilde{w}^{-1}\mu}}\star ({{J}}_{\widetilde{w}^{-1},*})_{\mu}$ is zero at $\widetilde{w}$ and is $\mathsf{e}$ at $1$.

For $g_0\in G(\mathcal{K})^{\omega^\rho}$, taking pushforward along the map $G(\mathcal{K})^{\omega^\rho}\to G(\mathcal{K})^{\omega^\rho}: g\mapsto g_0g^{-1}$ induces a functor
\begin{equation}
    \iota_{g_0}: \Shv_{\mathcal{G}^G}(\widetilde{\Fl})^{I,\lambda}\to \Shv_{(\mathcal{G}^{G})^{-1}\otimes \mathcal{G}^G_{g_0}}(\widetilde{\Fl})_{-\lambda},
\end{equation}
where $\mathcal{G}_{g_0}^G$ denotes the fiber of $\mathcal{G}^G$ at $g_0$.

Note that $({{J}}_{\widetilde{w},!})_{\overline{\widetilde{w}^{-1}\mu}}\mathop{\otimes}\limits^! \iota_{g^0}(({{J}}_{\widetilde{w}^{-1},*})_{\mu})$ is right $T^{\omega^\rho}$-equivariant, so it {descends} to a twisted sheaf on ${\Fl}_G^{\omega^\rho}$. We denote the resulting sheaf by the same notation.



    The projection of $g_0$ in $G(\mathcal{K})^{\omega^\rho}/I^0$ is denoted by $\bar{g}_0$. The preimage of $\bar{g}_0$ along \eqref{mult m} is identified with $\Fl_G^{\omega^\rho}$ via the composition of the first projection map $p_1: G(\mathcal{K})^{\omega^\rho}\mathop{\times}\limits^{I^{\omega^\rho}} G(\mathcal{K})^{\omega^\rho}/I^0\to G(\mathcal{K})^{\omega^\rho}/I^0$ and $G(\mathcal{K})^{\omega^\rho}/I^0\to G(\mathcal{K})^{\omega^\rho}/I^{\omega^\rho}$.

    Under this identification, the $!$-restriction of $ ({{J}}_{\widetilde{w},!})_{\overline{\widetilde{w}^{-1}\mu}}\star ({{J}}_{\widetilde{w}^{-1},*})_{\mu}$ to the preimage of $\bar{g}_0$ is identified with

\[ ({{J}}_{\widetilde{w},!})_{\overline{\widetilde{w}^{-1}\mu}}\mathop{\otimes}\limits^! \iota_{g^0}(({{J}}_{\widetilde{w}^{-1},*})_{\mu})\in \Shv_{\mathcal{G}^G_{g_0} }({\Fl}_G^{\omega^\rho}).\]

By the base change theorem, the $!$-stalks of $({{J}}_{\widetilde{w},!})_{\overline{\widetilde{w}^{-1}\mu}}\star ({{J}}_{\widetilde{w}^{-1},*})_{\mu}$ at the point $\bar{g}_0$ is isomorphic to \[H(\Fl^{\omega^\rho}_G, ({{J}}_{\widetilde{w}},!)_{\overline{\widetilde{w}^{-1}\mu}}\mathop{\otimes}\limits^! \iota_{g^0}(({{J}}_{\widetilde{w}^{-1},*})_{\mu}))\in \Vect_{\mathcal{G}^G_{g_0} }.\] 

In particular, the $!$-stalks of $ ({{J}}_{\widetilde{w},!})_{ \overline{\widetilde{w}^{-1}\mu}}\star ({{J}}_{\widetilde{w}^{-1},*})_{\mu}$ at $\widetilde{w}$ is isomorphic to
\[H(\Fl^{\omega^\rho}_G, ({{J}}_{\widetilde{w},!})_{ \overline{\widetilde{w}^{-1}\mu}}\mathop{\otimes}\limits^!\iota_{\widetilde{w}}(({{J}}_{\widetilde{w}^{-1},*})_{\mu})).\]

 Under the identification $\overline{I^{\omega^\rho} \widetilde{w} I^{\omega^\rho}/I^{\omega^\rho}}\simeq \mathbb{P}^1$, $({{J}}_{\widetilde{w},!})_{ \overline{\widetilde{w}^{-1}\mu}}\mathop{\otimes}\limits^!\iota_{\widetilde{w}}({}_{ }({{J}}_{\widetilde{w}^{-1},*})_{\mu})$ is identified with the
$*$-extension from a $\mathbb{G}_m$-equivariant (with respect to certain Kummer sheaf) sheaf on $\mathbb{G}_m:= \mathbb{P}^1\setminus \{0,\infty\}$ to $\mathbb{A}^1:=\mathbb{P}^1\setminus \{\infty\}$, and then $!$-pushforward to $\mathbb{P}^1$. By Braden's theorem (see \cite[Proposition 3.2.2]{[DG]}), its cohomology equals its $*$-stalks at $\infty$, which equals zero. 

Over the point $1\in \Fl_G^{\omega^\rho}$, the $!$-stalks can be calculated by \[H(\Fl^{\omega^\rho}_G, ({{J}}_{\widetilde{w},!})_{ \overline{\widetilde{w}^{-1}\mu}}\mathop{\otimes}\limits^!\iota_{1}({}_{ }({{J}}_{\widetilde{w}^{-1},*})_{\mu})).\]

Note that the restriction of $({{J}}_{\widetilde{w},!})_{ \overline{\widetilde{w}^{-1}\mu}}\mathop{\otimes}\limits^!\iota_{1}({}_{ }({{J}}_{\widetilde{w}^{-1},*})_{\mu})$ to $\mathbb{A}^1$ is isomorphic to the constant object. By the projection formula, $({{J}}_{\widetilde{w},!})_{ \overline{\widetilde{w}^{-1}\mu}}\mathop{\otimes}\limits^!\iota_{\widetilde{w}}({}_{ }({{J}}_{\widetilde{w}^{-1},*})_{\mu})$ is the $*$-pushforward of the constant sheaf on $\mathbb{A}^1$, and its cohomology is $\mathsf{e}$.
\end{proof}

\subsection{Twisted BMW sheaf}\label{twisted BMW D-modules}
With the preparation given in the last several sections, finally, we can construct the BMW sheaves in the twisted cases.
\begin{definition}
Given $\lambda\in \Lambda $, such that $\lambda=\lambda_1-\lambda_2,  \lambda_1, \lambda_2\in \Lambda^+$. We define the twisted BMW sheaf $({{J}}_\lambda)_{\mu}\in \Shv_{\mathcal{G}^G}(\widetilde{\Fl})_{\mu}$ as
\begin{equation}\label{5.4.1}
    ({{J}}_\lambda)_{\mu}:= ({{J}}_{\lambda_1,!})_{-\lambda_2+\mu}\star ({{J}}_{-\lambda_2,*})_\mu.
\end{equation}

Similarly, we define the dual BMW sheaf as
\begin{equation}\label{5.4.2}
    ({{J}}_\lambda^{\mathbb{D}})_\mu:=({{J}}_{\lambda_{1,*}})_{-\lambda_2+\mu}\star ({{J}}_{-\lambda_{2,!}})_\mu.
\end{equation}

{
They are well-defined up to tensoring by a line in degree $0$, and according to Lemma \ref{BMW}, the definitions are independent of choices of $\lambda_1, \lambda_2$.

\begin{rem}
To be more precise, they are determined by $\Shv_{\cG^G}(t^\lambda T(\cO)^{\omega^\rho}/T(\cO)^{\omega^\rho},b_\mu)\simeq \Vect$ (equivalently, an identification $\Shv_{\cG^G}(t^{\lambda_1}T(\cO)^{\omega^\rho}/T(\cO)^{\omega^\rho},b_{-\lambda_2+\mu})\otimes \Shv_{\cG^G}(t^{-\lambda_2}T(\cO)^{\omega^\rho}/T(\cO)^{\omega^\rho},b_\mu)\simeq \Vect$, since there is a canonical equivalence \[\Shv_{\cG^G}(t^{\lambda_1}T(\cO)^{\omega^\rho}/T(\cO)^{\omega^\rho},b_{-\lambda_2+\mu})\otimes \Shv_{\cG^G}(t^{-\lambda_2}T(\cO)^{\omega^\rho}/T(\cO)^{\omega^\rho},b_\mu)\simeq \Shv_{\cG^G}(t^\lambda T(\cO)^{\omega^\rho}/T(\cO)^{\omega^\rho},b_\mu)).\] The identification determines the right-hand side of \eqref{5.4.1} (resp. \eqref{5.4.2}), and there are canonical isomorphisms between $({{J}}_{\lambda_1,!})_{-\lambda_2+\mu}\star ({{J}}_{-\lambda_2,*})_\mu$ (resp. $({{J}}_{\lambda_{1,*}})_{-\lambda_2+\mu}\star ({{J}}_{-\lambda_{2,!}})_\mu$) for different $\lambda_1,\lambda_2$.

\end{rem}
}

When $\mu=0$, we will omit the subscript $\mu$ in $({{J}}_\lambda)_{\mu}$ and $ ({{J}}_\lambda^{\mathbb{D}})_\mu$.

\end{definition}


By Lemma \ref{BMW},  twisted BMW sheaves admit a convolution product
\begin{equation}\label{5.4.3}
    ({{J}}_{\eta})_{\lambda+\mu}\mathop{\star}({{J}}_{\lambda})_\mu\simeq ({{J}}_{\eta+\lambda})_{\mu}\in \Shv_{\mathcal{G}^G }(\widetilde{\Fl})^{I,b_{\lambda+\mu+\eta}}_\mu.
\end{equation}

{

\subsection{Standards and Costandards}\label{standards and costandards}

\begin{definition}\label{stan def} 
{Given an identification $\Shv_{\cG^G}(t^\lambda T(\cO)^{\omega^\rho}/T(\cO)^{\omega^\rho})\simeq\Vect$ as before, i.e., a trivialization of $\cG^G|_{t^\lambda\in \Fl}$, } {w}e define
\begin{equation}
    \Delta_\lambda= \Av_!^{N(\mathcal{K})^{\omega^\rho},\chi}({{J}}_\lambda),
\end{equation}
and
\begin{equation}
    \nabla_\lambda=  \Av_*^{ren}(\mathop{\colim}\limits_{\mu{,\mu-\lambda}\in \Lambda^+}({{J}}_{\mu,*})_{\lambda-\mu}\mathop{\star}\limits{{J}}_{\lambda-\mu,*}),
\end{equation}
where $Av_*^{ren}$ is the renormalized averaging functor, \begin{equation}\label{Avren}
    \Av_*^{ren}:= \mathop{\colim}\limits_{k} \Av_*^{N_k,\chi}\otimes l_{0,k},
\end{equation}
$\Av_*^{N_k,\chi}$ is the $*$-averaging functor with respect to $(N_k, \chi)$ (see Section \ref{N k}) and $l_{k,k'}$ is the line of $*$-fiber of the dualizing sheaf of $N_{k'}/N_k$ at $1$ for $k'\geq k$. {Furthermore, up to a shift, $({{J}}_{\mu,*})_{\lambda-\mu}\mathop{\star}\limits{{J}}_{\lambda-\mu,*}$ is the $(N(\cO)^{\omega^\rho}$, $*)$-averaging of the $*$-extension of the twisted dualizing sheaf on $\Ad_{\mu}I^{\omega^\rho}t^\lambda I^{\omega^\rho}/I^{\omega^\rho}$ with respect to the unique $\Ad_{\mu}I^0$-equivariant trivialization. If we further require $\mu_2-\mu_1\in \Lambda^+$, we have $\Ad_{\mu_1}I^{\omega^\rho}t^\lambda I^{\omega^\rho}/I^{\omega^\rho}\subset \Ad_{\mu_2}I^{\omega^\rho}t^\lambda I^{\omega^\rho}/I^{\omega^\rho}$.  The transition maps between  $({{J}}_{\mu,*})_{\lambda-\mu}\mathop{\star}\limits{{J}}_{\lambda-\mu,*}$ are obtained from natural maps between the dualizing sheaves.

}

\end{definition}
{Similar to twisted BMW sheaves, we also regard $\Delta_\lambda$ and $\nabla_\lambda$ as objects which are well-defined up to tensoring by a line in degree $0$, i.e., depend on the choice of identifications.} We call $\{\Delta_\lambda,\ \lambda\in \Lambda\}$ standard objects of $\Whit_q(\Fl_G^{\omega^\rho})$ and $\{\nabla_\lambda,\ \lambda\in \Lambda\}$ costandard objects.

\begin{prop}
$\{\Delta_\lambda, \lambda\in \Lambda\}$ is a collection of compact generators of $\Whit_q(\Fl_G^{\omega^\rho})$.
\end{prop}

\begin{proof}
Taking convolution with BMW sheaf is invertible, in particular, it preserves compactness. Since $\Delta_\lambda\simeq\Av_!^{N(\mathcal{K})^{\omega^\rho},\chi}({{J}}_\lambda)\simeq \Av_!^{N(\mathcal{K})^{\omega^\rho},\chi}(({{J}}_0)_\lambda)\star {{J}}_\lambda$ and $\Av_!^{N(\mathcal{K})^{\omega^\rho},\chi}(({{J}}_0)_\lambda)$ is compact, the standard object $\Delta_\lambda$ is compact.

By an analysis of the support of the convolution product, one proves that $\supp(\Delta_\lambda)= \overline{S}^{\phi(\rho+\lambda)}_{\Fl}$. According to \cite[Lemma 11]{[Be2]}, we have \[\Delta_\lambda|_{{S}^{\phi(\rho+\lambda)}_{\Fl}}\simeq  \Av_!^{N(\mathcal{K})^{\omega^\rho},\chi}({{J}}_{\lambda,!})|_{{S}^{\phi(\rho+\lambda)}_{\Fl}}\simeq \Delta_{\phi(\rho+\lambda)}^{\ver}|_{{S}^{\phi(\rho+\lambda)}_{\Fl}}.\] Now the proposition directly follows from the fact that $\operatorname{rank}(\Delta_{\phi(\rho+\lambda)}^{\ver}|_{{S}^{\phi(\rho+\lambda)}_{\Fl}})=1$ (i.e., $\Delta_{\phi(\rho+\lambda)}^{\ver}|_{{S}^{\phi(\rho+\lambda)}_{\Fl}}$ is the generator of $\Whit_q(S_{\Fl}^{\phi(\rho+\lambda)})\simeq \Vect$).
\end{proof}

The standards $\Delta_\lambda$ and costandards $\nabla_\mu$ satisfy the orthogonality property.
\begin{prop}\label{stan}
\begin{equation}
\mathcal{H}om_{\Whit_q(\Fl_G^{\omega^\rho})}(\Delta_\lambda, \nabla_\mu)\simeq\left\{
\begin{aligned}
    0,\qquad \textnormal{if}\ \lambda\neq \mu, \\ \mathsf{e},\qquad \textnormal{if}\  \lambda=\mu.
\end{aligned}
\right.
\end{equation}
\end{prop}

\begin{proof}
\begin{equation}\label{5.16}
\begin{split}
     &\mathcal{H}om_{\Whit_q(\Fl_G^{\omega^\rho})}(\Delta_\lambda, \nabla_\mu)\\
    \simeq&\mathcal{H}om_{\Whit_q(\Fl_G^{\omega^\rho})}(\Av_!^{N(\mathcal{K})^{\omega^\rho},\chi}({{J}}_\lambda),  \Av_*^{ren}\mathop{\colim}\limits_{\alpha{,\alpha-\mu}\in \Lambda^+}(({{J}}_{\alpha,*})_{\mu-\alpha}\mathop{\star}{{J}}_{\mu-\alpha,*}))\\
    \simeq&\mathop{\colim}\limits_{\alpha{,\alpha-\mu}\in \Lambda^+}\mathcal{H}om_{\Whit_q(\Fl_G^{\omega^\rho})}(\Av_!^{N(\mathcal{K})^{\omega^\rho},\chi}({{J}}_\lambda), \Av_*^{ren}(({{J}}_{\alpha,*})_{\mu-\alpha}\mathop{\star}{{J}}_{\mu-\alpha,*}))\\
    \simeq&\mathop{\colim}\limits_{\alpha{,\alpha-\mu}\in \Lambda^+}\mathcal{H}om_{\Whit_q(\widetilde{\Fl})_{\mu-\alpha}}(\Av_!^{N(\mathcal{K})^{\omega^\rho},\chi}({{J}}_\lambda\star ({{J}}_{\alpha-\mu,!})_{\mu-\alpha}), 
    \Av_*^{ren}(({{J}}_{\alpha,*})_{\mu-\alpha} \\ &\mathop{\star}{{J}}_{\mu-\alpha,*}\star ({{J}}_{\alpha-\mu,!})_{\mu-\alpha})\\
    \simeq&\mathop{\colim}\limits_{\alpha{,\alpha-\mu}\in \Lambda^+}\mathcal{H}om_{\Whit_q(\widetilde{\Fl})_{\mu-\alpha}}(\Av_!^{N(\mathcal{K})^{\omega^\rho},\chi}( ({{J}}_{\lambda+\alpha-\mu,!})_{\mu-\alpha}), 
    \Av_*^{ren}(({{J}}_{\alpha,*})_{\mu-\alpha}))
    \end{split}
\end{equation}
By the following lemma \ref{prop 10.6}, the above colimit is isomorphic to \[\mathop{\colim}\limits_{\alpha{,\alpha-\mu}\in \Lambda^+}\mathcal{H}om_{\Whit_q(\widetilde{\Fl})_{\mu-\alpha}}(\Av_!^{N(\mathcal{K})^{\omega^\rho},\chi}(\delta_{\lambda+\alpha-\mu})_{\mu-\alpha}[\langle\mu-\lambda-\alpha, 2\check{\rho}\rangle], \Av_*^{ren}(\delta_{\alpha})_{\mu-\alpha}[\langle\alpha, 2\check{\rho}\rangle]).\]

{Up to tensoring by a line,} we denote by $(\delta_{\lambda+\alpha-\mu})_{\mu-\alpha}\in \Shv_{\mathcal{G}^G}(\widetilde{\Fl})_{\mu-\alpha}$ (resp. $(\delta_{\alpha})_{\mu-\alpha}\in \Shv_{\mathcal{G}^G}(\widetilde{\Fl})_{\mu-\alpha}$) the $!$ (equivalently, $*$)-extension of the Kummer sheaf corresponding to $\mu-\alpha$ on $\widetilde{\Fl}\mathop{\times}\limits_{\Fl_G}t^{\lambda+\alpha-\mu}$ (resp. $\widetilde{\Fl}\mathop{\times}\limits_{\Fl_G}t^\alpha$) to $\widetilde{\Fl}$. 

$\Av_!^{N(\mathcal{K})^{\omega^\rho},\chi}(\delta_{\lambda+\alpha-\mu})_{\mu-\alpha}[\langle\mu-\lambda-\alpha, 2\check{\rho}\rangle]$ is the $!$-extension of the generator of $\Whit_q(S^{\lambda+\alpha-\mu}_{\widetilde{\Fl}})_{\mu-\alpha}$ to $\widetilde{\Fl}$, and $\Av_*^{ren}(\delta_{\alpha})_{\mu-\alpha}[\langle\alpha, 2\check{\rho}\rangle]$ is the $*$-extension of the generator of $\Whit_q(S^{\lambda+\alpha-\mu}_{\widetilde{\Fl}})_{\mu-\alpha}$ to $\widetilde{\Fl}$. Hence, by adjointness, we have \eqref{5.16} is $\mathsf{e}$ if $\lambda=\mu$ and is $0$ otherwise.
\end{proof}


\begin{lem}\label{prop 10.6}
 For $\lambda\in \Lambda^+$, $\mu\in \Lambda$, we have

 \textup{(1)} \[\Av_*^{ren}((\delta_\lambda)_\mu)[\langle\lambda,2\check{\rho}\rangle]\simeq \Av_*^{ren}(({{J}}_{\lambda,*})_{\mu}).\]

 \textup{(2)} \[\Av^{N(\mathcal{K})^{\omega^\rho},\chi}_!((\delta_\lambda)_\mu)[-\langle\lambda,2\check{\rho}\rangle]\simeq \Av^{N(\mathcal{K})^{\omega^\rho},\chi}_!(({{J}}_{\lambda,!})_{\mu}).\]
\end{lem}

\begin{proof}

For \textup{(1)}, note that if $\lambda$ is dominant, then $N(\mathcal{K})^{\omega^\rho}t^\lambda I^{\omega^\rho}=N(\mathcal{K})^{\omega^\rho}I^{\omega^\rho} t^\lambda I^{\omega^\rho}$. Hence, both sheaves in \textup{(1)} are $*$-extensions from their restrictions on $S^\lambda_{\widetilde{\Fl}}\subset \widetilde{\Fl}$. As they are $(N(\mathcal{K})^{\omega^\rho},\chi)$-equivariant, we only need to prove that their $!$-stalks at any lift of $t^\lambda\in \widetilde{\Fl}$ coincide. It follows from the constructions that the $!$-stalks of both sides are isomorphic to $\mathsf{e}[-\langle\lambda,2\check{\rho}\rangle]$.


The proof of the second claim is absolutely similar.
\end{proof}

\begin{rem}
In particular, if $\lambda$ is dominant, there is an isomorphism  \[\Delta_\lambda\simeq \Delta_\lambda^{\ver}\overset{\eqref{ver st}}{:=} \Av_!^{N(\mathcal{K}), \chi}(\delta_\lambda)[-\langle\lambda,2\check{\rho}\rangle].\]
\end{rem}

\subsection{T-structure on Whittaker category}\label{t structure}

Recall the following lemma in \cite{[BR]}.
\begin{lem}
Let $\mathcal{C}$ be a compactly generated category with compact generators $\{c_i\}$, then there is a t-structure given by $$\mathcal{C}^{\geq 0}:=\{c|\ \Hom_{\mathcal{C}}(c_i[k], c)=0\ \ \forall \lambda\in \Lambda\ \textnormal{and}\ k>0\}.$$ 
\end{lem}
In particular, we define a new t-structure (different from the one from the tautological t-structure on $\Shv_{\mathcal{G}^G}(\Fl_G^{\omega^\rho})$) on $\Whit_q(\Fl_G^{\omega^\rho})$ by the compact generators $\Delta_\lambda$, i.e.,
\begin{definition}\label{t2}
$\mathcal{F}\in \Whit_q(\Fl_G^{\omega^\rho})^{\geq 0}$ if and only if $$\Hom_{\Whit_q(\Fl_G^{\omega^\rho})}(\Delta_\lambda[k],\mathcal{F})=0\qquad \forall \lambda\in \Lambda\ \textnormal{and}\ k>0.$$
\end{definition}

After we proved our main theorem, we will see that $\Delta_\lambda$ and $\nabla_\lambda$ are in the heart of this t-structure (see Corollary \ref{heart}). Furthermore, $\Delta_\lambda$ and $\nabla_\lambda$  are of finite length for all $\lambda\in \Lambda$ if and only if $q$ is generic. When $q$ is a root of unity, costandard objects and irreducible objects are not even compact, but standard objects are still compact.

\subsection{Coinvariants}\label{coinvariant}
 By considering the coinvariant-Whittaker category, we can obtain the definition of the Verdier duality functor for $\Whit_q(\Fl_G^{\omega^\rho})$.
\begin{definition}

We define $\Whit_{q}(\Fl_G^{\omega^\rho})_{co}$ as the quotient DG-category of $\Shv_{\mathcal{G}^G}(\Fl_G^{\omega^\rho})$ by the full subcategory generated by
\begin{equation}\label{co}
    \Fib(\Av_*^{N_k,\chi}(\mathcal{F})\to \mathcal{F})
\end{equation}
$\textnormal{for all}\ \mathcal{F}\in \Shv_{\mathcal{G}^G}(\Fl_G^{\omega^\rho})\ \textnormal{and}\ k\in\mathbb{Z}.$

\end{definition}

\subsubsection{}The functor $\Av_*^{ren}$ (see (\ref{Avren})) maps all morphisms of the form \eqref{co} to isomorphisms. In particular, it induces a functor $\Av_*^{ren}$ from $\Whit_q(\Fl_G^{\omega^\rho})_{co}$ to $\Whit_q(\Fl_G^{\omega^\rho})$. The following lemma is proved in \cite[Theorem 2.1.1]{[Ras2]}.

\begin{lem}\label{invcoinveq}
The functor
\begin{equation}\label{5.7.2}
    \Av_*^{ren}: \Whit_q(\Fl_G^{\omega^\rho})_{co}\to \Whit_q(\Fl_G^{\omega^\rho})
\end{equation}
is an equivalence of categories.
\end{lem}

\subsubsection{} 
For a DG-category $\mathcal{C}$, we denote by $\mathcal{C}^\vee$ the dual category of $\mathcal{C}$ if it is dualizable. By definition, it is given by the DG-category of functors $\Fun(\mathcal{C}, \Vect)$. Since $\Whit_q(\Fl_G^{\omega^\rho})$ is compactly generated and any compactly generated category is dualizable, $\Whit_q(\Fl_G^{\omega^\rho})$ is dualizable. By Lemma \ref{invcoinveq}, $\Whit_q(\Fl_G^{\omega^\rho})_{co}$ is also dualizable. 




By definition, $\Whit_q(\Fl_G^{\omega^\rho})_{co}^\vee$ is the full subcategory of $ \Shv_{\mathcal{G}^G}(\Fl_G^{\omega^\rho})^\vee$ spanned by the functors
$$\Shv_{\mathcal{G}^G}(\Fl_G^{\omega^\rho})\to \Vect$$
\begin{equation}
    \begin{split}
        \mathcal{F}\mapsto \langle \mathcal{F}, \mathcal{F}_0\rangle,
   \end{split}
\end{equation}
such that for any $\mathcal{F}\in \Shv_{\mathcal{G}^G}(\Fl_G^{\omega^\rho})$ and $k$,

$$\langle \Av_*^{N_k,\chi}(\mathcal{F}), \mathcal{F}^0 \rangle\overset{\sim}{\longrightarrow} \langle\mathcal{F}, \mathcal{F}^0 \rangle,$$
i.e.,
$$\langle \mathcal{F}, \Av_*^{N_k,\chi}(\mathcal{F}^0) \rangle\overset{\sim}{\longrightarrow} \langle\mathcal{F}, \mathcal{F}^0 \rangle.$$

In other words, we require $\Av_*^{N_k,\chi}(\mathcal{F}^0)\simeq \mathcal{F}^0$ for any $k$. Hence, the duality functor \begin{equation}\label{5.40}
\begin{split}
      \Shv_{(\mathcal{G}^G)^{-1}}(\Fl_G^{\omega^\rho})  \simeq \Shv_{\mathcal{G}^G}(\Fl_G^{\omega^\rho})^\vee,\\
      \mathcal{F}_0\mapsto (\mathcal{F}\mapsto \langle \mathcal{F}, \mathcal{F}_0\rangle).
\end{split}
\end{equation} induces an equivalence of full subcategories \begin{equation}\label{dual invcoinv}
    \Whit_q(\Fl_G^{\omega^\rho})_{co}^\vee\simeq \Whit_{q^{-1}}(\Fl_G^{\omega^\rho}).
    \end{equation}.

\begin{definition}
The Verdier duality functor for Whittaker sheaves is defined as the composition of functors
\begin{equation}
  \mathbb{D}^{\verdier}:   \Whit_q(\Fl_G^{\omega^\rho})^\vee\underset{(\Av_*^{ren})^\vee}{\overset{\sim}{\longrightarrow}} \Whit_q(\Fl_G^{\omega^\rho})^\vee_{co} \underset{(\ref{dual invcoinv})}{\overset{\sim}{\longrightarrow}} \Whit_{q^{-1}}(\Fl_G^{\omega^\rho}).
\end{equation}
In particular, it defines an equivalence of subcategories generated by compact objects.
\begin{equation}
  \mathbb{D}^{\verdier}:   (\Whit_q(\Fl_G^{\omega^\rho})^c)^{op} {\overset{\sim}{\longrightarrow}} \Whit_{q^{-1}}(\Fl_G^{\omega^\rho})^c.
\end{equation}
\end{definition}



The dual of the standard object $\Delta_\lambda$ can be described by the dual BMW sheaf ${{J}}_\lambda^{\mathbb{D}}$.

\begin{prop}\label{LDK}
For any $\lambda\in \Lambda$, $\Delta_\lambda\in \Whit_q(\Fl_G^{\omega^\rho})$, we have
\begin{equation}
    \mathbb{D}^{\verdier}(\Delta_\lambda)\simeq \Av_*^{ren}({{J}}_\lambda^{\mathbb{D}}).
\end{equation}
\end{prop}

\begin{proof}
The object $\Delta_\lambda$ corresponds to the functor 
\begin{equation}
    \begin{split}
        \Whit_q(\Fl_G^{\omega^\rho})&\to \Vect\\
        \mathcal{F}&\mapsto \mathcal{H}om_{\Whit_q(\Fl_G^{\omega^\rho})}(\Delta_\lambda, \mathcal{F}).
    \end{split}
\end{equation}
in $\Whit_q(\Fl_G^{\omega^\rho})^\vee$. By definition, its image under $(\Av_*^{ren})^\vee$ is the functor
\begin{equation}
    \begin{split}
        \Whit_q(\Fl_G^{\omega^\rho})_{co}&\to \Vect\\
        \mathcal{F}'&\mapsto \mathcal{H}om_{\Whit_q(\Fl_G^{\omega^\rho})}(\Delta_\lambda, \Av_*^{ren}(\mathcal{F}')).
    \end{split}
\end{equation}
Using the adjointness, there is
\begin{align*}
    \mathcal{H}om_{\Whit_q(\Fl_G^{\omega^\rho})}(\Delta_\lambda, \Av_*^{ren}(\mathcal{F}'))\simeq& \mathcal{H}om_{\Shv_{\mathcal{G}^G}(\Fl_G^{\omega^\rho})}({{J}}_\lambda, \Av_*^{ren}(\mathcal{F}')).
\end{align*}
Note that ${{J}}_\lambda$ goes to ${{J}}_\lambda^{\mathbb{D}}$ under the equivalence $(\Shv_{\mathcal{G}^G}(\Fl_G^{\omega^\rho}))^\vee\simeq \Shv_{(\mathcal{G}^G)^{-1}}(\Fl_G^{\omega^\rho})$, we have $\mathcal{H}om_{\Shv_{\mathcal{G}^G}(\Fl_G^{\omega^\rho})}({{J}}_\lambda, \Av_*^{ren}(\mathcal{F}'))  \simeq\langle \Av_*^{ren}(\mathcal{F}'), {{J}}_{\lambda}^{\mathbb{D}}\rangle$.

So, by regarding $\Whit_q(\Fl_G^{\omega^\rho})_{co}^\vee$ as a subcategory of $\Shv_{\mathcal{G}^G}(\Fl_G^{\omega^\rho})^\vee$, the image of $\Delta_\lambda$ in $\Shv_{\mathcal{G}^G}(\Fl_G^{\omega^\rho})^\vee$ can be realized as
\begin{equation}
    \begin{split}
        \Shv_{\mathcal{G}^G}(\Fl_G^{\omega^\rho})&\to \Vect\\
        \mathcal{F}&\mapsto \langle \Av_*^{ren}(\mathcal{F}), {{J}}_{\lambda}^{\mathbb{D}}\rangle.
    \end{split}
\end{equation}
Since there is an isomorphism
\[\langle \Av_*^{ren}(\mathcal{F}), {{J}}_{\lambda}^{\mathbb{D}}\rangle\simeq \langle \mathcal{F}, \Av_*^{ren}({{J}}_{\lambda}^{\mathbb{D}})\rangle,\]
$(\Av_*^{ren})^\vee(\Delta_\lambda)$ is the functor
\begin{equation}
    \begin{split}
        \Shv_{\mathcal{G}^G}(\Fl_G^{\omega^\rho})&\to \Vect\\
        \mathcal{F}&\mapsto \langle \mathcal{F}, \Av_*^{ren}({{J}}_{\lambda}^{\mathbb{D}})\rangle.
    \end{split}
\end{equation}
By the construction of (\ref{dual invcoinv}), the image of $(\Av_*^{ren})^\vee(\Delta_\lambda)$ under (\ref{dual invcoinv}) is $\Av_*^{ren}({{J}}_{\lambda}^{\mathbb{D}})$.

\end{proof}

\section{The functor to the category of factorization modules }\label{try proof 1}
In this section, we will construct the functor from the twisted Whittaker category to $\Omega_q^L-\Fact$ mimicking the constructions in \cite{[GL1]}. {That is to say, \cite{[GL1]} uses the Jacquet functor, which is the pullback-pushfoward functor along $\Gr_G\longleftarrow \Gr_{B^-}\longrightarrow \Gr_T$. \footnote{The Jacquet functor is also the functor needed for geometric Satake equivalence.} We will construct an \textit{Iwahori Jacquet} functor. 

In bref, our Iwahori Jacquet functor is an adaptation of the pullback-pushforward functor $\Whit(\Fl_G)\longrightarrow \Shv(\Gr_T)$ of the diagram
\[\xymatrix{
\Fl^1_{B^-}\ar[r]\ar[d]& \Fl_G\\
\Gr_T&.
}\]
Here, $\Fl^1_{B^-}$ is substack of $\Fl_{B^-}:=\Gr_{B^-}\underset{\pt/G}{\times}\pt/B$ where we require the $B^-$-bundle and the Iwahori structure to be transversal at $x$. The $\lambda$-component of this functor is $H(\Fl_G, j_*(\omega_{S_{\Fl}^{-,\lambda}})\overset{!}{\otimes}-)$. In fact, the $\lambda$-component of functor constructed in this section is $H(\Fl_G, j_!(\omega_{S_{\Fl}^{-,\lambda}})\overset{!}{\otimes}-)$.
}

The organization of this section is as follows.

In Section \ref{whittaker category on ran fl}, we will construct a closed sub-prestack $(\overline{S}^{w_0}_{\Fl, \Ran_x})_{\infty\cdot x}$ of the Beilinson-Drinfeld affine flags $\Fl^{\omega^\rho}_{G, \Ran_x}$. Lemma \ref{infth} ensures that we can regard a Whittaker sheaf on $\Fl_G^{\omega^\rho}$ as a Whittaker sheaf on $(\overline{S}^{w_0}_{\Fl, \Ran_x})_{\infty\cdot x}$ with factorization property.

In Section \ref{configuration gr and fl}, we will introduce the configuration version affine flags which is important for the construction of the functor. 

In Section \ref{subsection F}, we will construct the functor $F^L$ that appeared in Theorem \ref{main theorem 1}. We will also construct another closely related functor $F^{DK}$ in this section.

In Section \ref{! fiber}, we will describe the functors defined in the previous section by calculating the $!$-stalks of $F^L$ and $F^{DK}$ at $\lambda\cdot x\in \Conf_{\infty\cdot x}$.

In Section \ref{Equivalence of categories}, we will prove the main theorem of this paper modulo Proposition \ref{imstan}.

\subsection{Whittaker category on {$\Fl^{\omega^\rho}_{G, \Ran_x}$}{}}\label{whittaker category on ran fl}

 \subsubsection{}
Recall that we defined the Beilinson-Drinfeld affine flags $\Fl_{G, \Ran_x}^{\omega^\rho}$ in Section \ref{Ran version of affine flags and affine Grassmannian} and \ref{2.5.3}. The idea of the construction of the functor $F^L$ is to regard a twisted Whittaker sheaf on $\Fl^{\omega^\rho}_G$ as a twisted Whittaker sheaf on $\Fl_{G, \Ran_x}^{\omega^\rho}$ (and $\Fl_{G, \Conf_{\infty\cdot x}}^{\omega^\rho}$), and then pushforward along the projection to $\Conf_{\infty\cdot x}$. Let us start by explaining the definition of Whittaker sheaves on $\Fl^{\omega^\rho}_{G, \Ran_x}$. 

Let $N(\mathcal{K})^{\omega^\rho}_{\Ran}$ (resp. $N(\mathcal{O})^{\omega^\rho}_{\Ran}$) be Ran-ified loop group of $N$. It is the prestack classifying the data $(\mathcal{I}, \alpha)$, where $\mathcal{I}\in \Ran(S)$ and $\alpha$ is an automorphism of $\omega^\rho\overset{T}{\times} B$ on $\overset{\circ}{\mathcal{D}}_{\mathcal{I}}$ (resp. ${\mathcal{D}}_{\mathcal{I}}$), which is compatible with the identification of $\omega^\rho$. Similarly, one can define $N(\mathcal{K})^{\omega^\rho}_{\Ran_x}$ and $N(\mathcal{O})^{\omega^\rho}_{\Ran_x}$.

We define a character 
\begin{equation}\label{cha chi}
    \begin{split}
        \chi_{\Ran_x}: N(\mathcal{K})^{\omega^\rho}_{\Ran_x}\to N(\mathcal{K})^{\omega^\rho}_{\Ran_x}/[N(\mathcal{K})^{\omega^\rho}_{\Ran_x}, N(\mathcal{K})^{\omega^\rho}_{\Ran_x}]\to \omega^r_{\Ran_x}(\mathcal{K})\to \\
\overset{\textnormal{sum}}{\longrightarrow} \omega_{\Ran_x}(\mathcal{K})\overset{\textnormal{residue}}{\longrightarrow} \mathbb{G}_a
    \end{split}
\end{equation}
of $N(\mathcal{K})^{\omega^\rho}_{\Ran_x}$.
Similarly, we can define a character 
\begin{equation}
    \chi_{\Ran}: N(\mathcal{K})_{\Ran}^{\omega^\rho}\to \mathbb{G}_a
\end{equation}
of $N(\mathcal{K})_{\Ran}^{\omega^\rho}$.

For $\bar{x}=\{x_1, x_2, \cdots, x_n\}\in \Ran$, we denote by $\chi_{\bar{x}}$ the restriction of $\chi_{\Ran}$ to $N(\mathcal{K})^{\omega^\rho}_{\bar{x}}:= N(\mathcal{K})^{\omega^\rho}_{\Ran}\underset{\Ran}{\times}\{\bar{x}\}$. Note that the character $\chi$ of $N(\mathcal{K})^{\omega^\rho}_{{x}}$ in \eqref{naive chi} equals $\chi_x$ here.

\subsubsection{}Left multiplication gives an action of $N(\mathcal{K})^{\omega^\rho}_{\Ran}$ on $\Gr^{\omega^\rho}_{G,\Ran}$, and an action of $N(\mathcal{K})^{\omega^\rho}_{\Ran_x}$ on $\Fl^{\omega^\rho}_{G,\Ran_x}$. Following \cite[Proposition 7.2.5]{[GL2]}, the pullback of $\mathcal{G}^G$ on $\Fl^{\omega^\rho}_{G, \Ran_x}$ (see Section \ref{gerbe used}) to $N(\mathcal{K})^{\omega^\rho}_{\Ran_x}$ is a multiplicative gerbe, in particular, the gerbe $\mathcal{G}^G$  on $\Fl^{\omega^\rho}_{G, \Ran_x}$ is equivariant with respect to $N(\mathcal{K})^{\omega^\rho}_{\Ran_x}$ against $\mathcal{G}^G$.  Since $N$ is unipotent, $N(\mathcal{K})^{\omega^\rho}_{\Ran_x}$ is an ind-pro-affine space over $\Ran_x$. There is a canonical trivialization of $\mathcal{G}^G$ on $N(\mathcal{K})^{\omega^\rho}_{\Ran_x}$. Hence, the gerbe $\mathcal{G}^G$ on $\Fl^{\omega^\rho}_{G, \Ran_x}$ is equivariant with respect to the action of $N(\mathcal{K})^{\omega^\rho}_{\Ran_x}$.  In particular, we may consider the category of $(N(\mathcal{K})^{\omega^\rho}_{\Ran_x},\chi_{\Ran_x})$-equivariant sheaves. 

\begin{definition}
We define
$$\Whit_q(\Fl^{\omega^\rho}_{G, \Ran_x}):= \Shv_{\mathcal{G}^G}(\Fl^{\omega^\rho}_{G, \Ran_x})^{N(\mathcal{K})^{\omega^\rho}_{\Ran_x},\chi_{\Ran_x}}.$$
\end{definition}
\subsubsection{}
Now we define a closed $N(\mathcal{K})_{\Ran_{x}}^{\omega^\rho}$-invariant subspace  $(\overline{S}^{w_0}_{\Fl, \Ran_x})_{\infty \cdot x}\subset \Fl_{G, \Ran_x}^{\omega^\rho}$.

\begin{definition}\label{S}
A point $(\mathcal{I},\mathcal{P}_G,\alpha,\epsilon)\in \Fl_{G,\Ran_{x}}^{\omega^\rho}$ belongs to $ (\overline{S}^{w_0}_{\Fl, \Ran_x})_{\infty \cdot x}$ if and only if for any dominant weight $\check{\lambda}\in \check{\Lambda}^+$, the composite meromorphic map
\begin{equation}\label{eq 6.1}
    \kappa^{\check{\lambda}}: (\omega^{\frac{1}{2}})^{\langle\check{\lambda},2\rho\rangle}\to \mathcal{V}_{\mathcal{P}_G^{\omega}}^{\check{\lambda}}\to \mathcal{V}_{\mathcal{P}_G}^{\check{\lambda}}(\infty\cdot x)
\end{equation}
is regular on $X\setminus x$. Here $\mathcal{V}_{\mathcal{P}_G}^{\check{\lambda}}$ (resp. $\mathcal{V}_{\mathcal{P}_G^\omega}^{\check{\lambda}}$) is defined as the vector bundle associated with $\mathcal{P}_G$ (resp. $\mathcal{P}_G^{\omega}:= \omega^\rho\overset{T}{\times}G$) with fiber the Weyl module $\mathcal{V}_G^{\check{\lambda}}$. The first map is the map mapping to the highest weight vector, and the second map is induced by $\alpha$.
\end{definition}

\begin{definition}\label{def 6.1.6}
A point $(\mathcal{I},\mathcal{P}_G, \alpha)\in \Gr^{\omega^\rho}_{G,\Ran}$ belongs to $\overline{S}_{\Gr, \Ran}^0$, if for any dominant weight $\check{\lambda}\in \check{\Lambda}^+$, the composite meromorphic map
\begin{equation}\label{6.1.3}
    \kappa^{\check{\lambda}}: (\omega^{\frac{1}{2}})^{\langle\check{\lambda},2\rho\rangle}\to \mathcal{V}_{\mathcal{P}_G^{\omega}}^{\check{\lambda}}\to \mathcal{V}_{\mathcal{P}_G}^{\check{\lambda}}
\end{equation} is regular on $X$.

If we require that $\kappa^{\check{\lambda}}$ in \eqref{6.1.3} is injective on $X$ for any $\check{\lambda}\in \check{\Lambda}^+$, the resulting prestack $S^0_{\Gr, \Ran}$ is the unique open dense $N(\mathcal{K})_{\Ran}^{\omega^\rho}$-orbit in $\overline{S}^0_{\Gr, \Ran}$.
\end{definition}

\subsubsection{Factorization property} Following the argument in the proof of \cite[Proposition 2.4]{[BFGM]}, we can prove that the prestacks defined above satisfy the following factorization property.
\begin{lem}\label{basic fact property}\
\begin{enumerate}[label={\upshape(\arabic*)}]
    \item $\overline{S}^0_{\Gr, \Ran}$ is a factorization prestack, i.e., there is a canonical isomorphism of prestacks

\begin{equation}\label{6.1.5}
    \overline{S}^0_{\Gr, \Ran}\mathop{\times}\limits_{\Ran} (\Ran\times \Ran)_{disj}\simeq \overline{S}^0_{\Gr, \Ran}\times \overline{S}^0_{ \Gr, \Ran}\mathop{\times}\limits_{\Ran\times \Ran} (\Ran\times \Ran)_{disj}.
\end{equation}
    \item $(\overline{S}^{w_0}_{\Fl, \Ran_x})_{\infty \cdot x}$ factorizes with respect to $\overline{S}^0_{\Gr, \Ran}$, i.e., there is a canonical isomorphism of prestacks

\begin{equation}\label{6.1.6}
\begin{split}
   (\overline{S}^{w_0}_{\Fl, \Ran_x})_{\infty \cdot x}&\mathop{\times}\limits_{\Ran_x} (\Ran\mathop{\times}\Ran_{x})_{disj}\\ \simeq&\\ \overline{S}^0_{\Gr, \Ran}\times (\overline{S}^{w_0}_{\Fl, \Ran_x})_{\infty \cdot x}&\mathop{\times}\limits_{\Ran\times \Ran_x} (\Ran\times \Ran_{ x})_{disj}.
\end{split}
\end{equation}
\end{enumerate}
\end{lem}

\subsubsection{}\label{6.1.9}
Let $\Fl_{G,x}^{\omega^\rho}$ (resp. $\Gr_{G,x_i}^{\omega^\rho}$) denote the fiber of $\Fl_{G, \Ran_x}^{\omega^\rho}$ (resp. $\Gr_{G, \Ran}^{\omega^\rho}$) over $x$ (resp. $x_i$), it is isomorphic to $\Fl_G^{\omega^\rho}$ (resp. $\Gr_G^{\omega^\rho}$) by choosing a uniformizer. Denote by $\overline{S}^0_{\Gr, x_i}$ the closure of the $N(\mathcal{K})_{x_i}^{\omega^\rho}$-orbit of $t^0\in \Gr^{\omega^\rho}_{G,x_i}$.

By definition, the fiber of  {$(\overline{S}^{w_0}_{\Fl,\Ran_{x}})_{\infty\cdot x}$} over the point ${\cI}=\{x,x_1,x_2,\cdots,x_k\}\in \Ran_{x}$ is isomorphic to the product $\Fl^{\omega^\rho}_{G,x}\times\prod_{i=1}^k \overline{S}_{\Gr, x_i}^{0}$, and the fiber of $\overline{S}^0_{\Ran,\Gr}$ over the point ${\cI}=\{x_1,x_2,\cdots,x_k\}\in \Ran$ is isomorphic to $\prod_{i=1}^k \overline{S}^0_{\Gr, x_i}$. 

\subsubsection{Relation with $\Whit_q(\Fl_G^{\omega^\rho})$}
Consider the product space $\Ran_{x}\times \Fl^{\omega^\rho}_{G,x}$. The $N(\mathcal{K})^{\omega^\rho}_x$ action  on the second factor gives a $N(\mathcal{K})^{\omega^\rho}_x$-action on $\Ran_x\times \Fl^{\omega^\rho}_{G,x}$. The pullback of the gerbe $\mathcal{G}^G$ on $\Fl^{\omega^\rho}_{G,x}$ to $\Ran_x\times \Fl^{\omega^\rho}_{G,x}$ is still $N(\mathcal{K})^{\omega^\rho}_x$-equivariant, hence, we can consider the Whittaker category on  $\Ran_x\times \Fl^{\omega^\rho}_{G,x}$,
\begin{equation}
    \Whit_q(\Ran_x\times \Fl^{\omega^\rho}_{G,x}):= \Shv_{\mathcal{G}^G}(\Ran_x\times \Fl^{\omega^\rho}_{G, \Ran_x})^{N(\mathcal{K})_x^{\omega^\rho}, \chi}.
\end{equation}



There is a closed embedding
   \begin{equation}\label{unit}
 \unit: \Ran_x\times \Fl^{\omega^\rho}_{G,x} \to (\overline{S}^{w_0}_{\Fl, \Ran_{x}})_{\infty \cdot x},
   \end{equation}
which sends $\mathcal{I}\in \Ran_x, (\mathcal{P}_G, \alpha, \epsilon)\in \Fl_{G,x}^{\omega^\rho}$ to $(\mathcal{I}, \mathcal{P}_G, \alpha, \epsilon)\in (\overline{S}^{w_0}_{\Fl, \Ran_{x}})_{\infty \cdot x}$.
   Similarly, we have
     \begin{equation}\label{unit Gr}
\unit_{\Gr}: \Ran\times\overline{S}^{0}_{\Gr} \to \overline{S}_{\Gr, \Ran}^{0}.
   \end{equation}
   

 The $!$-pullback along the projection $$\pr_{\Ran_x}: \Ran_x\times \Fl_{G,x}^{\omega^\rho}\to \Fl_{G,x}^{\omega^\rho}$$ gives rise to a functor 
  \begin{equation}
      \pr_{\Ran_x}^!:\Shv_{\mathcal{G}^G}(\Fl_{G,x}^{\omega^\rho})\to \Shv_{\mathcal{G}^G}(\Ran_x\times \Fl_{G,x}^{\omega^\rho}).
  \end{equation}
 
By definition, $\pr_{\Ran_x}$ commutes with $N(\mathcal{K})_x^{\omega^\rho}$-actions, so $\pr_{\Ran_x}^!$ induces a functor between the corresponding Whittaker categories
\begin{equation}
    \Whit_q(\Fl_{G,x}^{\omega^\rho})\to \Whit_q(\Ran_x\times \Fl_{G,x}^{\omega^\rho}).
\end{equation}
\subsubsection{}
Consider the pullback functor along \eqref{unit} $$\unit^!:\Shv_{\mathcal{G}^G}((\overline{S}^{w_0}_{\Fl, \Ran_{x}})_{\infty \cdot x})\to \Shv_{\mathcal{G}^G}(\Ran_x\times \Fl_{G,x}^{\omega^\rho}). $$  We claim that this map induces a functor between the corresponding Whittaker categories
\begin{equation}
    \unit^!: \Whit_q((\overline{S}^{w_0}_{\Fl, \Ran_{x}})_{\infty \cdot x})\to  \Whit_q(\Ran_x\times \Fl^{\omega^\rho}_{G,x}).
\end{equation}

Indeed, consider the closed subgroup $N'$ in $N(\mathcal{K})_{\Ran_x}^{\omega^\rho}$ whose fiber over a point $\{x, x_1,x_2,\cdots,x_k\} \linebreak \in \Ran_x$ is given by $N(\mathcal{K})_x^{\omega^\rho}\times\prod_{i=1}^k N(\mathcal{O})_{x_i}^{\omega^\rho}$. Restriction to $x$ gives a projection \begin{equation}\label{projection of N}
    N'\to N(\mathcal{K})^{\omega^\rho}_x.
\end{equation}
The map $\eqref{unit}$ is compatible with $N'$-action, where the action of $N'$ on $\Ran_x\times \Fl_{G,x}^{\omega^\rho}$ is given by the projection \eqref{projection of N} and the action of $N(\mathcal{K})^{\omega^\rho}_x$ on $\Fl_{G,x}^{\omega^\rho}$. Since the kernel of the projection (\ref{projection of N}) is pro-unipotent, the forgetful functor \[\Shv_{\mathcal{G}^G}(\Ran_x\times \Fl_{G,x}^{\omega^\rho})^{N',\chi}{\to} \Shv_{\mathcal{G}^G}(\Ran_x\times \Fl_{G,x}^{\omega^\rho})^{N(\mathcal{K})_x^{\omega^\rho},\chi}\] is an equivalence. Hence, $\unit^!$ induces a functor
\begin{equation}
    \begin{split}
        \Shv_{\mathcal{G}^G}((\overline{S}^{w_0}_{\Fl, \Ran_{x}})_{\infty \cdot x})^{N(\mathcal{K})^{\omega^\rho}_{\Ran_x},\chi_{\Ran_x}}\to \Shv_{\mathcal{G}^G}(\Ran_x\times \Fl_{G,x}^{\omega^\rho})^{N',\chi}\simeq\\ \simeq \Shv_{\mathcal{G}^G}(\Ran_x\times \Fl_{G,x}^{\omega^\rho})^{N(\mathcal{K})_x^{\omega^\rho},\chi},
    \end{split}
\end{equation}
i.e.,
\begin{equation} \label{inf}
    \unit^!:\ \Whit_q((\overline{S}^{w_0}_{\Fl, \Ran_{x}})_{\infty \cdot x})\to \Whit_q(\Ran_x\times \Fl_{G,x}^{\omega^\rho}).
\end{equation}
Similarly, we can define
\begin{equation}\label{inf Gr}
    \unit^!_{\Gr}: \Whit_q(\overline{S}^0_{\Gr, \Ran})\to \Whit_q(\Ran\times \overline{S}^0_{\Gr}).
\end{equation}

{According to \cite[Theorem 6.2.5]{[Ga5]}, the functor (\ref{inf Gr}) is an equivalence.} By an argument similar to \cite[Section {6.2-}6.6]{[Ga5]}, we can prove the following lemma. 
\begin{lem}\label{infth}
{The functor} (\ref{inf}) is an equivalence.
\end{lem} 

{
\begin{proof}
  We only sketch the proof. 
  
Given a finite set $\fI$ with a distinguished point, let $X^{\fI}_x$ be the subspace of $X^{\fI}$ such that the coordinate indexed by the distinguished point is $x$. Note that $\Ran_x=\colim X^{\fI}_x$, and we can define $\Whit_q((\overline{S}^{w_0}_{\Fl, X^{\fI}_{x}})_{\infty \cdot x})$ and $ \Whit_q(X^{\fI}_x\times \Fl_{G,x}^{\omega^\rho})$ similarly.

It is sufficient to show that $\unit_{\fI}^!: \Whit_q((\overline{S}^{w_0}_{\Fl, X^{\fI}_{x}})_{\infty \cdot x})\to \Whit_q(X^{\fI}_x\times \Fl_{G,x}^{\omega^\rho})$ is an equivalence for any finite set $\fI$ with a distinguished point, then the desired property follows by taking limit. For any such $\fI$, one can give $X^{\fI}_x$ a stratification $\{X^{\fB}_x\}$ according to the collision of points. Note that there is $(\overline{S}^{w_0}_{\Fl, X^{\fB}_{x}})_{\infty \cdot x}\simeq (((X-x)^k-\text{Diag})\underset{\Ran}{\times} \overline{S}^0_{\Gr,\Ran})\times \Fl_{G,x}^{\omega^\rho}$, where $k$ is the number of different elements in $\fB$ without the element containing the distinguished point. So, we have 
\begin{equation}
    \begin{split}
        \Whit_q((\overline{S}^{w_0}_{\Fl, X^{\fB}_{x}})_{\infty \cdot x})\simeq\Shv_{\cG^G}((((X-x)^k-\text{Diag})\underset{\Ran}{\times} \overline{S}^0_{\Gr,\Ran})\times \Fl_{G,x}^{\omega^\rho})^{N(\cK)^{\omega^\rho}_{X^{\fB}_x},\chi_{{\fB}}}\\ \simeq\Shv_{\cG^G}((((X-x)^k-\text{Diag})\underset{\Ran}{\times} {S}^0_{\Gr,\Ran})\times \Fl_{G,x}^{\omega^\rho})^{N(\cK)^{\omega^\rho}_{X^{\fB}_x},\chi_{{\fB}}}\\ \simeq \Shv_{\cG^G}((((X-x)^k-\text{Diag}))\times \Fl_{G,x}^{\omega^\rho})^{N(\cK)_x^{\omega^\rho},\chi_x} \simeq \Whit_q(X^{\fB}_{x}\times \Fl_{G,x}^{\omega^\rho}).
    \end{split}
\end{equation}
Here, the second equivalence follows from the fact that $\overline{S}^0_{\Gr,\Ran}\backslash{S}^0_{\Gr,\Ran}$ does not carry non-zero Whittaker sheaf, and the third equivalence follows from $N(\cK)^{\omega^\rho}_{X^{\fB}_x}\simeq (((X-x)^k-\text{Diag})\underset{\Ran}{\times} N(\cK)^{\omega^\rho}_{\Ran})\times N(\cK)_x^{\omega^\rho}$.
It implies that the restriction $\unit_{\fB}^!:\Whit_q((\overline{S}^{w_0}_{\Fl, X^{\fB}_{x}})_{\infty \cdot x})\to \Whit_q(X^{\fB}_{x}\times \Fl_{G,x}^{\omega^\rho})$ is an equivalence. 

In particular, $\unit^!_{\fI}$ is conservative, and we only need to construct the left adjoint functor $(\unit^!_{\fI})^L$ such that $\Id\to \unit^!_{\fI}\circ(\unit^!_{\fI})^L$ is an isomorphism.

For $n\geq 0$, let $I_{n}:= \Ad_{-n\rho}(G(\cO)^{\omega^\rho}\underset{G(\cO/t^n\cO)^{\omega^\rho}}{\times}N(\cO/t^n\cO)^{\omega^\rho})$. There is a canonical way to extend the character $\chi|_{I_n\cap N(\cK)^{\omega^\rho}}$ to $I_n$ such that it is trivial on the negative part (i.e., $B^-(\cO)\cap I_n$), we still use the same notation $\chi$.

    A quite non-trivial result of \cite[Theorem 2.7.1]{[Ras3]} says that the left adjoint $\Av_!^{N(\cK)^{\omega^\rho},\chi_x}$ is well-defined for any $(I_n,\chi_x)$-equivariant sheaf. With the same proof, one can show that $\Av_!^{N(\cK)_{{\fI}}^{\omega^\rho},\chi_{\fI}}$ is well-defined for any $(I_n',\chi_x)$-equivariant sheaf on $\Fl_{G,X^{\fI}_x}^{\omega^\rho}$. Here, $I_n'\subset G(\cK)^{\omega^\rho}_{\fI}$ is the subgroup whose fiber over $\{x,x_1,x_2,\cdots,x_n\}$ is $I_n\times \prod_{i=1}^n G(\cO)^{\omega^\rho}_{x_i}$, and $\chi_{\fI}$ is the character given by the map $I_n'\to I_n\overset{\chi}{\to}\BG_a$.

    To construct the left adjoint functor $(\unit_{\fI}^!)^L$ of $\unit_{\fI}^!$, we should prove that $\Av_!^{N(\cK)_{{\fI}}^{\omega^\rho},\chi_{\fI}}$ is well-defined on the image of the composition
    \[\Whit_q(X^{\fI}_{x}\times \Fl_G^{\omega^\rho})\overset{\oblv^{N(\cK)^{\omega^\rho},\chi_x}}{\longrightarrow}\Shv_{\cG^G}(X^{\fI}_{x}\times \Fl_G^{\omega^\rho})\overset{\unit_!}{\longrightarrow} \Shv_{\cG^G}((\overline{S}^{w_0}_{\Fl, X^{\fI}_{x}})_{\infty \cdot x}).\]
Then, $(\unit_{\fI}^!)^L$ is given by the composition of the above functor and $\Av_!^{N(\cK)_{{\fI}}^{\omega^\rho},\chi_{\fI}}$.

Note that the category $\Whit_q(X^{\fI}_{x}\times \Fl_G^{\omega^\rho})$ is generated by applying $\Av_!^{N(\cK)^{\omega^\rho},\chi_x}$ to $(I_n,\chi_x)$-equivariant objects on $X^{\fI}_{x}\times \Fl_G^{\omega^\rho}$. So, we only need to show that $\Av_!^{N(\cK)_{{\fI}}^{\omega^\rho},\chi_{\fI}}$ is well-defined on the image of 
\begin{equation}\label{6.1.17}
    \Shv_{\cG^G}(X^{\fI}_{x}\times \Fl_G^{\omega^\rho})^{I_n,\chi_x}\overset{\oblv^{I_n,\chi_x}}{\longrightarrow}\Shv_{\cG^G}(X^{\fI}_{x}\times \Fl_G^{\omega^\rho})\overset{\unit_!}{\longrightarrow} \Shv_{\cG^G}((\overline{S}^{w_0}_{\Fl, X^{\fI}_{x}})_{\infty \cdot x}).
\end{equation}

Furthermore, note that $\unit_{\fI}:X^{\fI}_x\times \Fl_{G,x}^{\omega^\rho}\to (\overline{S}^{w_0}_{\Fl, X^{\fI}_{x}})_{\infty \cdot x}$ is $I_n'$-invariant, where $I_n'$ acts on $X^{\fI}_{x}\times \Fl_G^{\omega^\rho}$ via $I_n'\to I_n$. In particular, the image of \eqref{6.1.17} is $({I_n',\chi_x})$-equivariant, and $\Av_!^{N(\cK)_{{\fI}}^{\omega^\rho},\chi_{\fI}}$ is well-defined for those equivariant sheaves.

Finally, by repeating the same argument as in \cite[Section 6.5]{[Ga5]} we can prove that the functor $(\unit_{\fI}^!)^L$ satisfies the desired property, i.e., $\Id\to \unit_{\fI}^!\circ (\unit_{\fI}^!)^L$ is an isomorphism.
\end{proof}

}

\begin{definition}
For $\mathcal{F}\in \Whit_q(\Fl_G^{\omega^\rho})$, we denote by $\sprd_{\Fl, \Ran_x}(\mathcal{F})$ the Whittaker sheaf on $(\overline{S}^{w_0}_{\Fl, \Ran_x})_{\infty\cdot x}$ corresponding to $\pr_{\Ran_x}^!(\mathcal{F})$ under the equivalence  (\ref{inf}).
\end{definition}
Using the fiber description of $(\overline{S}^0_{\Fl,\Ran_{x}})_{\infty\cdot x}$ in Section \ref{6.1.9}, we can describe $\sprd_{\Fl, \Ran_x}(\mathcal{F})$ more explicitly. Namely, the restriction of $\sprd_{\Fl, \Ran_x}(\mathcal{F})$ to $\Fl_{G,x}^{\omega^\rho}$ is $\mathcal{F}$, and its restriction to $\overline{S}^0_{\Gr, x_i}$ is the generator of $\Whit_q(\overline{S}^0_{\Gr, x_i})$.

\begin{definition}Let  $\Vac $ denote the Whittaker sheaf on $\overline{S}^0_{\Gr, \Ran}$ which is uniquely characterized by the property that its !-pullback to $\Ran$ via the canonical section 
\begin{equation}
    \begin{split}
    s_{\Ran}: \Ran\to \overline{S}^0_{\Gr, \Ran}\\
    \mathcal{I}\mapsto (\mathcal{I}, \mathcal{P}_G^{\omega}, \id)
    \end{split}
\end{equation}
is the dualizing sheaf on $\Ran$, i.e., $s_{\Ran}^!(\Vac)\simeq \omega_{\Ran}$. 
\end{definition}
It is important that $\Vac$ and $\sprd_{\Fl, \Ran_x}(\mathcal{F})$ satisfy the following factorization properties.
\begin{cor}\label{fact bas}\
\begin{enumerate}[label={\upshape(\arabic*)}]
    \item $\Vac$ is a factorization algebra on $\overline{S}^0_{\Gr, \Ran}$, i.e., there is a canonical isomorphism
\[\Vac\boxtimes \Vac|_{(\overline{S}^0_{\Gr, \Ran}\times \overline{S}^0_{\Gr, \Ran})_{disj}}\simeq \Vac|_{\overline{S}^0_{\Gr, \Ran}\underset{\Ran}{\times} (\Ran\times\Ran)_{disj}},\]
which is compatible with \eqref{6.1.5}.
    \item $\sprd_{\Fl, \Ran_x}(\mathcal{F})\in \Whit_q((\overline{S}^{w_0}_{\Fl, \Ran_{x}})_{\infty \cdot x})$ is a factorization module over $\Vac$, i.e., there is a canonical isomorphism
    \[ \Vac\boxtimes \sprd_{\Fl, \Ran_x}(\mathcal{F})|_{(\overline{S}^0_{\Gr, \Ran}\times (\overline{S}^{w_0}_{\Fl, \Ran_{x}})_{\infty \cdot x})_{disj}}\]\[\simeq\] \[\sprd_{\Fl, \Ran_x}(\mathcal{F})|_{(\overline{S}^{w_0}_{\Fl, \Ran_{x}})_{\infty \cdot x})\underset{\Ran_{ x}}{\times}(\Ran\times \Ran_{ x})_{disj}},\]
which is compatible with \eqref{6.1.6}.
\end{enumerate}
\end{cor} 

\begin{proof}
\textup{(1)} is \cite[Theorem 8.4.6 (a)]{[GL1]}.
We only show \textup{(2)}.

It is known that for any $x\in \overline{S}^0_{\Gr, \Ran}\setminus S^0_{\Gr, \Ran}$, we have \[\Stab_{N(\mathcal{K})^{\omega^\rho}_{\Ran}}(x)\nsubseteq \Ker(\chi_{\Ran}).\] 

It implies $\Whit_q(\overline{S}^0_{\Gr, \Ran}\setminus S^0_{\Gr, \Ran})=0$. In particular,  \[\Whit_q(\overline{S}^0_{\Gr, \Ran})\simeq \Whit_q(S^0_{\Gr, \Ran}).\]

Since $N(\mathcal{K})^{\omega^\rho}_{\Ran}$ acts transitively on $S^0_{\Gr, \Ran}$ over $\Ran$ and $N(\mathcal{K})^{\omega^\rho}_{\Ran}$ is ind-pro-unipotent, taking $!$-stalks along $s_{\Ran}$ induces an equivalence of categories \[\Whit_q(S^0_{\Gr, \Ran})\simeq \Shv(\Ran).\] 

Consider the following commutative diagram:
$$\xymatrix{
\Ran_x\times \Fl_{G,x}^{\omega^\rho}\ar[r]^{\unit} & (\overline{S}^{w_0}_{\Fl, \Ran_x})_{\infty\cdot x}\\
(\Ran\times \Ran_x\times \Fl_{G,x}^{\omega^\rho})_{disj}\ar[u]^{\cup_x\times \id}\ar[r]^{s_{\Ran}\times \unit}&(\overline{S}^0_{\Gr, \Ran}\times (\overline{S}^{w_0}_{\Fl, \Ran_x})_{\infty\cdot x})_{disj}\ar[u]}.$$
By Lemma \ref{infth}, we need to prove $$(\cup_x\times \id)^!\circ \unit^!(\sprd_{\Fl, \Ran_x}(\mathcal{F}))\simeq \omega_{\Ran}\boxtimes \unit^!(\sprd_{\Fl, \Ran_x}(\mathcal{F}))|_{disj},$$
which follows from the facts that $\unit^!(\sprd_{\Fl, \Ran_x}(\mathcal{F}))\simeq \omega_{\Ran_x}\boxtimes \mathcal{F}$, and the (twisted) sheaf $\omega_{\Ran_x}$ on $\Ran_x$ factorizes with respect to $\omega_{\Ran}$, i.e., we have \[\cup_x^!(\omega_{\Ran_x})|_{disj}\simeq \omega_{\Ran}\boxtimes\omega_{\Ran_x}|_{disj}.\]
\end{proof}

\subsection{Configuration version of {$\Gr_G^{\omega^\rho}$ and $\Fl_G^{\omega^\rho}$}{}}\label{configuration gr and fl}
The most important prestacks in this paper are constructed in this section. They are analogs of the constructions in \cite{[GL1]}. The target of the functor $F^L$ lives on $\Conf_{\infty\cdot x}$. Hence, it is convenient to consider prestacks over the configuration space. In this section, we will explain the configuration version $\Gr_G^{\omega^\rho}$ and $\Fl_G^{\omega^\rho}$ and related factorization prestacks. 

\begin{definition}\label{def conf}
Let $\Gr_{G, \Conf}^{\omega^\rho}$ (resp. $\Fl_{G, \Conf_{\infty\cdot x}}^{\omega^\rho}$) be the prestack over $\Conf$ (resp. $\Conf_{\infty\cdot x}$) which classifies the data $(D, \mathcal{P}_G, \alpha)$ (resp. $(D, \mathcal{P}_G, \alpha, \epsilon)$), here $D\in \Conf$ (resp. $\Conf_{\infty\cdot x}$ ), $\mathcal{P}_G\in \Bun_G,\ \alpha:\mathcal{P}_G|_{X\setminus \supp(D)}\simeq \mathcal{P}_G^{\omega}|_{X\setminus \supp(D)}$, and $\epsilon$ is a $B$-reduction of $\mathcal{P}_G$ at $x$.
 
 Similarly, we can define the configuration analog of $N(\mathcal{K})_{\Ran}^{\omega^\rho}$, $N(\mathcal{O})_{\Ran}^{\omega^\rho}$, $N(\mathcal{K})_{\Ran_{x}}^{\omega^\rho}$, $N(\mathcal{O})_{\Ran_{x}}^{\omega^\rho}$, etc. We denote the resulting prestacks by $N(\mathcal{K})_{\Conf}^{\omega^\rho}$, $N(\mathcal{O})_{\Conf}^{\omega^\rho}$, $N(\mathcal{K})_{\Conf_{\infty\cdot x}}^{\omega^\rho}$, and $N(\mathcal{O})_{\Conf_{\infty\cdot x}}^{\omega^\rho}$, respectively.
\end{definition}
\subsubsection{}
Note that $$\Gr_{G,\Ran}^{\omega^\rho}\mathop{\times}\limits_{\Ran} (\Gr^{\omega^\rho}_{T, \Ran})^{\textnormal{neg}}\simeq \Gr^{\omega^\rho}_{G, \Conf}\mathop{\times}\limits_{\Conf} (\Gr^{\omega^\rho}_{T, \Ran})^{\textnormal{neg}}.$$ As a result, the gerbe $\mathcal{G}^G$ on $\Gr^{\omega^\rho}_{G, \Ran}$ gives a gerbe on $\Gr^{\omega^\rho}_{G, \Conf}\mathop{\times}\limits_{\Conf} (\Gr^{\omega^\rho}_{T, \Ran})^{\textnormal{neg}}$. By Lemma \ref{shf eq conf Gr}, it descends to a gerbe on $\Gr^{\omega^\rho}_{G, \Conf}$. We still denote it by $\mathcal{G}^G$. Similarly, we can define gerbes on other prestacks in Definition \ref{def conf}.

To define the functor $F^L$, we need to define two sub-prestacks of $\Fl_{G, \Conf_{\infty\cdot x}}^{\omega^\rho}$: one carries Whittaker sheaves, and the other one carries the kernel. The former space is given by $(\overline{S}^{w_0}_{\Fl, \Conf_{\infty\cdot x}})_{\infty\cdot x}$, and the latter is $S^{-, \Conf_{\infty\cdot x}}_{\Fl, \Conf_{\infty\cdot x}}$.
\begin{definition}
Denote by $\overline{S}^{0}_{\Gr, \Conf}$ the closed sub-prestack of $\Gr_{G, \Conf}^{\omega^\rho}$ such that the maps $\kappa^{\check{\lambda}}$ in \eqref{6.1.3}
extend to regular maps on $X$  and satisfy the Pl\"{u}cker relations.

We denote by $(\overline{S}^{w_0}_{\Fl, \Conf_{\infty\cdot x}})_{\infty\cdot x}$ the closed sub-prestack of $\Fl_{G, \Conf_{\infty\cdot x}}^{\omega^\rho}$ such that the maps $\kappa^{\check{\lambda}}$ in \eqref{eq 6.1} are regular on $X\setminus x$  and satisfy Pl\"{u}cker relations.
\end{definition}
\begin{definition}
Let $S^{-, \Conf}_{\Gr, \Conf}$ (resp. $\overline{S}^{-, \Conf}_{\Gr, \Conf}$) denote the prestack classifying the data $(D, \mathcal{P}_G, \alpha)$, such that for any $\check{\lambda}$ dominant, the induced map

\begin{equation}
    \kappa^{-,\check{\lambda}}:\  '\mathcal{V}_{\mathcal{P}_G}^{\check{\lambda}}\to '\mathcal{V}_{\mathcal{P}_G^{\omega}}^{\check{\lambda}}\to (\omega^{\frac{1}{2}})^{\langle\check{\lambda}, 2\rho\rangle}(\langle \check{\lambda}, -D\rangle),
\end{equation}
which is a priori defined on $X\setminus \supp(D)$, extends to a surjective (resp. regular) map on the whole curve $X$ and satisfies the Pl\"{u}cker relations. Here $'\mathcal{V}_{\mathcal{P}_G}^{\check{\lambda}}$ (resp. $'\mathcal{V}_{\mathcal{P}_G^\omega}^{\check{\lambda}}$) is the vector bundle associated with $\mathcal{P}_G$ (resp. $\mathcal{P}_G^{\omega}= \omega^\rho\overset{T}{\times}G$) with fiber the dual Weyl module $'\mathcal{V}_G^{\check{\lambda}}$. The first map is induced by $\alpha$ and the second map is the map mapping to the highest weight vector.

Let $S^{-, \Conf_{\infty\cdot x}}_{\Fl, \Conf_{\infty\cdot x}}$ (resp. $\overline{S}^{-, \Conf_{\infty\cdot x}}_{\Fl, \Conf_{\infty\cdot x}}$) denote the prestack classifying the data from $S^{-, \Conf}_{\Gr, \Conf}$ (resp. $\overline{S}^{-, \Conf}_{\Gr, \Conf}$) plus a $B$-reduction of $\mathcal{P}_G$ at $x$. 

\end{definition}

The fiber of $\Fl_{G, \Conf_{\infty\cdot x}}^{\omega^\rho}$ (resp. $\Gr_{G, \Conf}^{\omega^\rho}$) over the point \eqref{Dx} (resp. \eqref{D}) is canonically isomorphic to
\begin{equation}\label{6.23}
    \Fl_{G,x}^{\omega^\rho}\times \prod_{1\leq i\leq k} \Gr_{G,x_i}^{\omega^\rho}.
\end{equation}
\begin{equation}
   \textnormal{(resp.}\ \prod_{1\leq i\leq k} \Gr_{G,x_i}^{\omega^\rho}\textnormal{)}
\end{equation}

\subsubsection{Factorization property}Similar to Lemma \ref{basic fact property}, $\overline{S}^0_{\Gr, \Conf}$ and $(\overline{S}^{w_0}_{\Fl, \Conf_{\infty\cdot x}})_{\infty \cdot x}$ are factorizable. That is to say,

\begin{equation}\label{conf 0 fact Gr}
\begin{split}
        \overline{S}^0_{\Gr, \Conf}&\mathop{\times}\limits_{\Conf} (\Conf\times \Conf)_{disj}\\ &\simeq\\ \overline{S}^0_{\Gr, \Conf}\times \overline{S}^0_{\Gr, \Conf}&\mathop{\times}\limits_{\Conf\times \Conf} (\Conf\times \Conf)_{disj},
\end{split}
\end{equation}
and
\begin{equation}\label{conf 0 fact}
\begin{split}
    (\overline{S}^{w_0}_{\Fl, \Conf_{\infty\cdot x}})_{\infty \cdot x}&\mathop{\times}\limits_{\Conf_{\infty\cdot x}} (\Conf\mathop{\times}\Conf_{x})_{disj}\\ &\simeq\\ \overline{S}^0_{\Gr, \Conf}\times (\overline{S}^{w_0}_{\Fl, \Conf_{\infty\cdot x}})_{\infty \cdot x}&\mathop{\times}\limits_{\Conf\times \Conf_{\infty\cdot x}} (\Conf\times \Conf_{ x})_{disj}.
\end{split}
\end{equation}

Furthermore, $S^{-, \Conf}_{\Gr, \Conf}$ is a factorization prestack, $S^{-, \Conf_{\infty\cdot x}}_{\Fl, \Conf_{\infty\cdot x}}$ is factorizable with respect to $S^{-, \Conf}_{\Gr, \Conf}$.

\subsubsection{Relative position}\label{relative p}
The prestack $S^{-, \Conf_{\infty\cdot x}}_{\Fl, \Conf_{\infty\cdot x}}$ admits a stratification given by the \textit{relative position} of the $B$-reduction given by $\epsilon$ and the $B^-$-reduction given by the morphisms $\{\kappa^{{-,\check{\lambda}}}\}$. To be more precise, the morphisms $\{\kappa^{-,\check{\lambda}}\}$ are surjective, so they induce a $B^-$-reduction of $\mathcal{P}_G$ at $x$, i.e., we have a map
\begin{equation}
    S^{-,\Conf_{\infty\cdot x}}_{\Fl, \Conf_{\infty\cdot x}}\to \Bun_G\mathop{\times}\limits_{\pt/G}\pt/B^-
\end{equation}
given by sending a point $(D, \mathcal{P}_G, \alpha, \epsilon)$ of $S^{-,\Conf_{\infty\cdot x}}_{\Fl, \Conf_{\infty\cdot x}}$ to $\mathcal{P}_G$ and its $B^-$-reduction at $x$ induced by $\{\kappa^{-,\check{\lambda}}\}$. In addition, $\epsilon$ also gives a map
\begin{equation}
    S^{-,\Conf_{\infty\cdot x}}_{\Fl, \Conf_{\infty\cdot x}}\to \Bun_G\mathop{\times}\limits_{\pt/G}\pt/B.
\end{equation}

Note that their compositions with the functors of inductions to $G$-bundles coincide, so we have a map of relative position
\begin{equation}
    \operatorname{rp}: S^{-,\Conf_{\infty\cdot x}}_{\Fl, \Conf_{\infty\cdot x}}\to \Bun_G\mathop{\times}\limits_{\pt/G}\pt/B\mathop{\times}\limits_{\pt/G}\pt/B^-\simeq \Bun_G\mathop{\times}\limits_{\pt/G} B^-\backslash G/B.
\end{equation}

The Bruhat decomposition gives a double coset decomposition of $(B^-,B)$ in $G$ and it induces a stratification of $B^-\backslash G/B$. We denote by $\Br^w\subset B^-\backslash G/B$ the Bruhat cell corresponding to $B^-wB$.

\begin{definition}
For $w\in W$, let us denote by $S^{-,w, \Conf_{\infty\cdot x}}_{\Fl, \Conf_{\infty\cdot x}}$ the preimage of $\Br^w$ in $S^{-,\Conf_{\infty\cdot x}}_{\Fl, \Conf_{\infty\cdot x}}$. In particular, $S^{-,1, \Conf_{\infty\cdot x}}_{\Fl, \Conf_{\infty\cdot x}}$ is open dense in $S^{-,\Conf_{\infty\cdot x}}_{\Fl, \Conf_{\infty\cdot x}}$.
\end{definition}

Similar to Lemma \ref{basic fact property}, the prestacks $S^{-,w, \Conf_{\infty\cdot x}}_{\Fl, \Conf_{\infty\cdot x}}$ and $S^{-,\Conf_{\infty\cdot x}}_{\Fl, \Conf_{\infty\cdot x}}$ (and their closures in $\Fl^{\omega^\rho}_{G,\Conf_{\infty\cdot x}}$) factorize with respect to $S^{-,\Conf}_{\Gr, \Conf}$ (resp. $\overline{S}^{-,\Conf}_{\Gr, \Conf}$). i.e.,

\begin{equation}\label{fact conf - w}
\begin{split}
        S^{-,w, \Conf_{\infty\cdot x}}_{\Fl, \Conf_{\infty\cdot x}}&\mathop{\times}\limits_{\Conf_{\infty\cdot x}} (\Conf\times \Conf_{\infty\cdot x})_{disj}\\ &\simeq\\ S^{-,\Conf}_{\Gr, \Conf}\times S^{-,w, \Conf_{\infty\cdot x}}_{\Fl, \Conf_{\infty\cdot x}}&\mathop{\times}\limits_{\Conf\times \Conf_{\infty\cdot x}} (\Conf\times \Conf_{\infty\cdot x})_{disj}
\end{split}
\end{equation}

\begin{equation}\label{fact conf - w '}
\begin{split}
      \textnormal{(resp.}\  \overline{S}^{-,w, \Conf_{\infty\cdot x}}_{\Fl, \Conf_{\infty\cdot x}}&\mathop{\times}\limits_{\Conf_{\infty\cdot x}} (\Conf\times \Conf_{\infty\cdot x})_{disj}\\&\simeq\\ \overline{S}^{-,\Conf}_{\Gr, \Conf}\times \overline{S}^{-,w, \Conf_{\infty\cdot x}}_{\Fl, \Conf_{\infty\cdot x}}&\mathop{\times}\limits_{\Conf\times \Conf_{\infty\cdot x}} (\Conf\times \Conf_{\infty\cdot x})_{disj}\textnormal{).}
\end{split}
\end{equation}

\subsubsection{Description of fibers}
The fiber of $S^{-,w,\Conf_{\infty\cdot x}}_{\Fl, \Conf_{\infty\cdot x}}$ over the point $D= \lambda_x\cdot x+ \mathop{\sum}\limits_i \lambda_i\cdot x_i\in \Conf_{\infty\cdot x}$ is canonically isomorphic to
\begin{equation}\label{fiber of S conf conf}
    S^{-,t^{\lambda_x} w}_{\Fl,x}\times \prod_{i} S^{-,\lambda_i}_{\Gr,x_i}.
\end{equation}

Here $S^{-,t^{\lambda_x} w}_{\Fl,x}\subset \Fl^{\omega^\rho}_{G,x}$ denotes the $N^-(\mathcal{K})^{\omega^\rho}$-orbit of $t^{\lambda_x} w\in \Fl_{G,x}^{\omega^\rho}$, and $S^{-,\lambda_i}_{\Gr,x_i}\subset \Gr^{\omega^\rho}_{G,x}$ denotes the $N^-(\mathcal{K})^{\omega^\rho}$-orbit of $t^{\lambda_i} \in \Gr_{G, x_i}^{\omega^\rho}$.

\begin{rem}
The above identification of the fiber is compatible with the one given in \eqref{6.23}.
\end{rem} 




{
\subsection{Semi-infinite sheaf on $S_{\Fl,\Conf_{\infty\cdot x}}^{-,\Conf_{\infty\cdot x}}$}\label{section 6.3}
In Appendix \ref{app a}, we review the theory of semi-infinite sheaves on affine flags. In this section, we use it to define the $!$-extension semi-infinite sheaf on $S_{\Fl,\Conf_{\infty\cdot x}}^{-,w,\Conf_{\infty\cdot x}}$.

Consider the sub-prestack $\overline{S}^{-, 1}_{\Fl,\Ran_x}$ of $\Fl_{G,\Ran_x}^{\omega^\rho}$, which classifies the data $(\cI, \mathcal{P}_G, \alpha, \epsilon)$, such that for any $\check{\lambda}$ dominant, the induced map
\begin{equation}
    \kappa^{-,\check{\lambda}}:\  '\mathcal{V}_{\mathcal{P}_G}^{\check{\lambda}}\to '\mathcal{V}_{\mathcal{P}_G^{\omega}}^{\check{\lambda}}\to (\omega^{\frac{1}{2}})^{\langle\check{\lambda}, 2\rho\rangle},
\end{equation}
which is a priori defined on $X-\cI$, extends to a  regular map on the whole curve $X$ and satisfies the Pl\"{u}cker relations. We let ${S}^{-,1}_{\Fl,\Ran_x}$ be the substack where we require that the extended map $\kappa^{-,\check{\lambda}}$ to be surjective and the induced $B^-$-bundle and $\epsilon$ to be transversal at $x$.

The restriction of $(\cG^G)^{-1}$ to ${S}^{-,1}_{\Fl,\Ran_x}$ is canonically trivialized, we denote by $\omega_{{S}^{-,1}_{\Fl,\Ran_x}}$ the $(\cG^G)^{-1}$-twisted dualizing sheaf on ${S}^{-,1}_{\Fl,\Ran_x}$. According to Proposition \ref{Prop A 1.5}, the $!$-extension is well-defined for $\omega_{{S}^{-,1}_{\Fl,\Ran_x}}$, and we denote it by $j_!(\omega_{{S}^{-,1}_{\Fl,\Ran_x}})$. 

\subsubsection{}
The twisted dualizing sheaf $\omega_{{S}^{-,1}_{\Fl,\Ran_x}}$ and its $!$-extension $j_!(\omega_{{S}^{-,1}_{\Fl,\Ran_x}})$ naturally acquire $T(\cO)_{\Ran_x}^{\omega^\rho}$-equivariant structures, and $(\cG^G)^{-1}$ is $T(\cO)_{\Ran_x}^{\omega^\rho}$-equivariant, we can think $j_!(\omega_{{S}^{-,1}_{\Fl,\Ran_x}})$ as a sheaf on $T(\cO)_{\Ran_x}^{\omega^\rho}\backslash \overline{S}_{\Fl,\Ran_x}^{-,1}$. In particular, the construction of $j_!(\omega_{S_{\Fl,\Ran_x}^{-,1}})$ admits a $T$-twisted construction.

That is to say, given a prestack $\cY$ with a map $\cY\longrightarrow T(\cO)^{\omega^\rho}_{\Ran_x}\backslash \Ran_x$, we consider the fiber product
\[{}_{\cY}\Fl:=\cY\underset{T(\cO)^{\omega}_{\Ran_x}\backslash \Ran_x}{\times} T(\cO)_{\Ran_x}^{\omega^\rho}\backslash \Fl_{G,\Ran_x}^{\omega^\rho}.\]

We let $({}_{\cY}\cG^G)^{-1}$ be the pullback of the descent gerbe on $T(\cO)_{\Ran_x}^{\omega^\rho}\backslash \Fl_{G,\Ran_x}^{\omega^\rho}$ and ${}_{\cY}j_!(\omega_{S_{\Fl,\Ran_x}^{-,1}})$ be the $!$-pullback of $j_!(\omega_{S_{\Fl,\Ran_x}^{-,1}})$ along the projection
\begin{equation}\label{map 6.3.3}
    {}_{\cY}\Fl\longrightarrow T(\cO)_{\Ran_x}^{\omega^\rho}\backslash \Fl_{G,\Ran_x}^{\omega^\rho}.
\end{equation}

\subsubsection{}
Let $\cY=(\Gr_{T,\Ran_x}^{\omega^\rho})_{\infty\cdot x}^{\text{neg}}$, we have the following identification
\begin{equation}\label{A 3.1}
    {}_{\cY}\Fl= \Fl_{G,\Ran_x}^{\omega^\rho}\mathop{\times}\limits_{\Ran_x} (\Gr^{\omega^\rho}_{T, \Ran_x})_{\infty\cdot x}^{\textnormal{neg}}\simeq \Fl_{G,\Conf_{\infty\cdot x}}^{\omega^\rho}\mathop{\times}\limits_{\Conf_{\infty\cdot x}} (\Gr^{\omega^\rho}_{T, \Ran_x})_{\infty \cdot x}^{\textnormal{neg}}.
\end{equation}

Under the above identification, the preimage of $T(\cO)_{\Ran_x}^{\omega^\rho}\backslash {S}_{\Fl,\Ran_x}^{-,1}$ under \eqref{map 6.3.3} is identified with the product ${S}^{-,1, \Conf_{\infty\cdot x}}_{\Fl, \Conf_{\infty\cdot x}}\mathop{\times}\limits_{\Conf_{\infty\cdot x}} (\Gr^{\omega^\rho}_{T, \Ran_x})_{\infty \cdot x}^{\textnormal{neg}}$. Additionally, the gerbe $({}_{\cY}\cG^G)^{-1}$ on $\Fl_{G,\Conf_{\infty\cdot x}}^{\omega^\rho}\mathop{\times}\limits_{\Conf_{\infty\cdot x}} (\Gr^{\omega^\rho}_{T, \Ran_x})_{\infty \cdot x}^{\textnormal{neg}}$ is identified with the ratio gerbe
\begin{equation}
\cG^{G,T,\ratio}=(\cG^G)^{-1}\otimes \cG^\Lambda.
\end{equation}

Since $(\Gr^{\omega^\rho}_{T, \Ran_x})_{\infty \cdot x}^{\textnormal{neg}}\longrightarrow \Conf_{\infty\cdot x}$ is an isomorphism in $h$-topology, the gerbes and the corresponding categories of twisted sheaves on ${}_{\cY}\Fl$ and $\Fl_{G,\Conf_{\infty\cdot x}}^{\omega^\rho}$ are the same. 

\begin{definition}
     We denote the sheaf corresponding to ${}_{\cY}j_!(\omega_{S_{\Fl,\Ran_x}^{-,1}})$ by $j_!({\omega_{S^{-,1, \Conf_{\infty\cdot x}}_{\Fl, \Conf_{\infty\cdot x}}}})\in \Shv_{\cG^{G,T,\ratio}}(\Fl_{G,\Conf_{\infty\cdot x}}^{\omega^\rho})$. Its restriction to $S^{-,1, \Conf_{\infty\cdot x}}_{\Fl, \Conf_{\infty\cdot x}}$ is the dualizing sheaf under the canonical trivialization of $\cG^{G,T,\ratio}|_{S^{-,1, \Conf_{\infty\cdot x}}_{\Fl, \Conf_{\infty\cdot x}}}$.
\end{definition}


\subsubsection{}
The above constructions also work for affine Grassmannian. 
To be more precise, let $j_!(\omega_{S_{\Gr,\Ran}^{-,0}})\in \Shv_{(\cG^G)^{-1}}(\overline{S}_{\Gr,\Ran}^{-,0})$ be the $!$-extension of the dualizing sheaf on ${S_{\Gr,\Ran}^{-,0}}$ (which is defined similarly as ${S_{\Fl,\Ran_x}^{-,1}}$, but without the Iwahori structure). Given $\cY:=(\Gr_{T,\Ran}^{\omega^\rho})^{\text{neg}}\longrightarrow T(\cO)_{\Ran}^{\omega^\rho}\backslash \Gr_{G,\Ran}^{\omega^\rho}$, we can also define a twisted sheaf $j_!({\omega_{S^{-,\Conf}_{\Gr, \Conf}}})\in \Shv_{\cG^{G,T,\ratio}}(\Gr_{G,\Conf}^{\omega^\rho})$ which corresponds to ${}_{\cY}j_!(\omega_{S^{-,0}_{\Gr,\Ran}})$. 

Using the fact that $j_!(\omega_{S_{\Fl,\Ran_x}^{-,1}})$ is factorizable with respect to $j_!(\omega_{S_{\Gr,\Ran}^{-,0}})$, and pulling-back along ${}_{\cY}\Fl\longrightarrow T(\cO)_{\Ran_x}^{\omega^\rho}\backslash \Fl_{G,\Ran_x}^{\omega^\rho}$ and ${}_{\cY}\Gr\longrightarrow T(\cO)_{\Ran}^{\omega^\rho}\backslash \Gr_{G,\Ran}^{\omega^\rho}$ preserve factorization structures, we conclude that $j_!({\omega_{S^{-,\Conf}_{\Gr, \Conf}}})$ and $j_!({\omega_{S^{-,1, \Conf_{\infty\cdot x}}_{\Fl, \Conf_{\infty\cdot x}}}})$ satisfy the factorization properties.

}


{That is to say, $j_!(S^{-, \Conf}_{\Gr, \Conf})$ is a factorization algebra,
\begin{equation}\label{factalg-}
\begin{split}
    &j_!(\omega_{S^{-, \Conf}_{\Gr, \Conf}})|_{\overline{S}^{-, \Conf}_{\Gr, \Conf}\mathop{\times}\limits_{\Conf} (\Conf\times \Conf)_{disj}} \\
\simeq &j_!(\omega_{S^{-, \Conf}_{\Gr, \Conf}})\boxtimes j_!(\omega_{S^{-, \Conf}_{\Gr, \Conf}})|_{\overline{S}^{-,\Conf}_{\Gr, \Conf}\times \overline{S}^{-, \Conf}_{\Gr, \Conf}\mathop{\times}\limits_{\Conf\times \Conf} (\Conf\times \Conf)_{disj}},
\end{split}
\end{equation}
and $j_!(\omega_{S^{-,1,\Conf_{\infty\cdot x}}_{\Fl, \Conf_{\infty\cdot x}}})$ factorizes with respect to $j_!(\omega_{S^{-, \Conf}_{\Gr, \Conf}})$, i.e.,
\begin{equation}\label{fact-}
\begin{split}
    &j_!(\omega_{S^{-,1,\Conf_{\infty\cdot x}}_{\Fl, \Conf_{\infty\cdot x}}})|_{\overline{S}^{-,1, \Conf_{\infty\cdot x}}_{\Fl, \Conf_{\infty\cdot x}}\mathop{\times}\limits_{\Conf_{\infty\cdot x}} (\Conf\times \Conf_{\infty\cdot x})_{disj}} \\
\simeq &j_!(\omega_{S^{-,\Conf}_{\Gr, \Conf}})\boxtimes j_!(\omega_{S^{-,1,\Conf_{\infty\cdot x}}_{\Fl, \Conf_{\infty\cdot x}}})|_{\overline{S}^{-,\Conf}_{\Gr, \Conf}\times \overline{S}^{-,1, \Conf_{\infty\cdot x}}_{\Fl, \Conf_{\infty\cdot x}}\mathop{\times}\limits_{\Conf\times \Conf_{\infty\cdot x}} (\Conf\times \Conf_{\infty\cdot x})_{disj}}.
\end{split}
\end{equation}}

\subsection{Constructions of functors}\label{subsection F}
In this section, we will define the functor $F^L: \Whit_q(\Fl^{\omega^\rho}_G)\to \Shv_{\mathcal{G}^\Lambda}(\Conf_{\infty\cdot x}).$ which is used in Theorem \ref{main theorem 1}.

To start with, let us summarize the prestacks defined in previous sections of this paper in the following diagram:
\begin{center}
\begin{equation}\label{diagram 15.2}
  \xymatrix{
&\Fl_{G, \Conf_{\infty\cdot x}}^{\omega^\rho}&\\
&(\overline{S}^{w_0}_{\Fl, \Conf_{\infty\cdot x}})_{\infty\cdot x}\ar[u]\ar[d]&\overline{S}^{-, \Conf_{\infty\cdot x}}_{\Fl, \Conf_{\infty\cdot x}}\ar[lu]\\
\Ran_x\times \Fl_{G,x}^{\omega^\rho} \ar[d]^{\pr_{\Ran_x}}\ar[r]^{\unit}&(\overline{S}^{w_0}_{\Fl, \Ran_x})_{\infty\cdot x}&(\overline{S}^{w_0}_{\Fl, \Conf_{\infty\cdot x}})_{\infty\cdot x}\cap \overline{S}^{-, \Conf_{\infty\cdot x}}_{\Fl, \Conf_{\infty\cdot x}}\ar[dd]^{v_{\Conf_{\infty\cdot x}}}\ar[lu]\ar[u]\\
\Fl_{G,x}^{\omega^\rho}&&&\\
&&\Conf_{\infty\cdot x}.\\
&&}
\end{equation}
\end{center}

The morphism $\unit: \Ran_x\times \Fl_{G,x}^{\omega^\rho}\to (\overline{S}^{w_0}_{\Fl, \Ran_x})_{\infty\cdot x}$ is given by (\ref{unit}).

\subsubsection{{{Construction of}} $F^L$} $F^L$ can be constructed via the following steps:
\begin{enumerate}[label=(\arabic*)]
    \item Given a twisted Whittaker sheaf $\mathcal{F}\in \Whit_q(\Fl_{G,x}^{\omega^\rho})$, first of all, we !-pullback it to $\Ran_x\times \Fl_{G,x}^{\omega^\rho}$ along the morphism $\pr_{\Ran_x}$. By Lemma \ref{infth}, it gives rise to a twisted Whittaker sheaf $\sprd_{\Fl, \Ran_x}(\mathcal{F})$ on  $(\overline{S}^{w_0}_{\Fl, \Ran_x})_{\infty\cdot x}$.
    \item Consider the image of $\sprd_{\Fl, \Ran_x}(\mathcal{F})$ under the following functor
\begin{align*}
    \Shv_{\mathcal{G}^G}((\overline{S}^{w_0}_{\Fl, \Ran_x})_{\infty \cdot x})\to \Shv_{\mathcal{G}^G}((\overline{S}^{w_0}_{\Fl, \Ran_x})_{\infty \cdot x}\mathop{\times}\limits_{\Ran_x} (\Gr^{\omega^\rho}_{T, \Ran_x})^{\textnormal{neg}}_{\infty \cdot x})\\ \simeq \Shv_{\mathcal{G}^G}((\overline{S}^{w_0}_{\Fl, \Conf_{\infty\cdot x}})_{\infty \cdot x}\mathop{\times}\limits_{\Conf_{\infty\cdot x}} (\Gr^{\omega^\rho}_{T, \Ran_x})^{\textnormal{neg}}_{\infty \cdot x}) \simeq \Shv_{\mathcal{G}^G}((\overline{S}^{w_0}_{\Fl, \Conf_{\infty\cdot x}})_{\infty \cdot x}).
\end{align*}
The first functor above is given by $!$-pullback. The second one follows from the isomorphism
\[(\overline{S}^{w_0}_{\Fl, \Ran_x})_{\infty \cdot x}\mathop{\times}\limits_{\Ran_x} (\Gr^{\omega^\rho}_{T, \Ran_x})_{\infty\cdot x}^{\textnormal{neg}}\simeq (\overline{S}^{w_0}_{\Fl, \Conf_{\infty\cdot x}})_{\infty\cdot x}\mathop{\times}\limits_{\Conf_{\infty\cdot x}} (\Gr^{\omega^\rho}_{T, \Ran_x})_{\infty \cdot x}^{\textnormal{neg}}.\] The third one is given by Lemma \ref{lemma 5.2}. 
We denote the resulting sheaf by $\sprd_{\Fl}(\mathcal{F})$.
\begin{equation}\label{def of sprd}
    \sprd_{\Fl}: \Whit_q(\Fl_G^{\omega^\rho})\to \Shv_{\mathcal{G}^G}((\overline{S}^{w_0}_{\Fl, \Conf_{\infty\cdot x}})_{\infty \cdot x}).
\end{equation}
\item Take $!$-tensor product of $\sprd_{\Fl}(\mathcal{F})$ with the semi-infinite $!$-extension sheaf $j_!(\omega_{S_{\Fl, \Conf_{\infty\cdot x}}^{-, 1,\Conf_{\infty\cdot x}}})$ defined in Section \ref{configuration gr and fl}.
\item Then take $!$ (or equivalently, take $*$)-pushforward along the projection $v_{\Conf_{\infty\cdot x}}$ with cohomology shift $\langle \lambda, 2\check{\rho} \rangle$ on the connected component $\Conf_{\infty\cdot x}^\lambda$ of $\Conf_{\infty\cdot x}$. 
\end{enumerate}
\begin{definition}
To summarize, the functor $$F^L: \Whit_q(\Fl_{G,x}^{\omega^\rho})\to \Shv_{\mathcal{G}^\Lambda}(\Conf_{\infty\cdot x})$$ is defined as
\begin{equation}\label{6.35}
    \mathcal{F}\mapsto v_{\Conf_{\infty\cdot x}, *}(\sprd_{\Fl}(\mathcal{F})\overset{!}{\otimes} j_!({\omega_{S^{-,{1}, \Conf_{\infty\cdot x}}_{\Fl, \Conf_{\infty\cdot x}}}})|_{(\overline{S}^{w_0}_{\Fl, \Conf_{\infty\cdot x}})_{\infty \cdot x}\cap\overline{S}^{-, \Conf_{\infty\cdot x}}_{\Fl, \Conf_{\infty\cdot x}}})[\deg],
\end{equation}
where the shift $[\deg]$ equals $\langle \lambda, 2\check{\rho} \rangle$ on the connected component $\Conf_{\infty\cdot x}^\lambda$.
\end{definition}
\begin{rem}
The resulting sheaf $F^L(\mathcal{F})$ is $\mathcal{G}^\Lambda$-twisted. Indeed, $\sprd_{\Fl}(\mathcal{F})\in \Shv_{\mathcal{G}^G}((\overline{S}^{w_0}_{\Fl, \Ran_x})_{\infty \cdot x})$ and $j_!({\omega_{S^{-,{1,} \Conf_{\infty\cdot x}}_{\Fl, \Conf_{\infty\cdot x}}}})\in \Shv_{\mathcal{G}^{G, T, \ratio}}(\overline{S}^{-, \Conf_{\infty\cdot x}}_{\Fl, \Conf_{\infty\cdot x}})$. Note that $\mathcal{G}^{G, T, \ratio}$ is the quotient of $\mathcal{G}^G$ by $\mathcal{G}^\Lambda$. Hence, the tensor product of the sheaf $\sprd_{\Fl}(\mathcal{F})\in \Shv_{\mathcal{G}^G}((\overline{S}^{w_0}_{\Fl, \Ran_x})_{\infty \cdot x})$ and the sheaf $j_!({\omega_{S^{-,{1,} \Conf_{\infty\cdot x}}_{\Fl, \Conf_{\infty\cdot x}}}})\in \Shv_{\mathcal{G}^{G, T, \ratio}}(\overline{S}^{-, \Conf_{\infty\cdot x}}_{\Fl, \Conf_{\infty\cdot x}})$ is ${\mathcal{G}^\Lambda}$-twisted.
\end{rem}
\subsubsection{{Construction of} $F^L_{\Gr}$}Similarly, we consider the following diagram:
\begin{center}
\begin{equation}\label{diagram 15.5}
\xymatrix{ &\Gr_{G, \Conf}^{\omega^\rho}&\\
&\overline{S}^{0}_{\Gr, \Conf}\ar[u]\ar[d]&\overline{S}^{-, \Conf}_{\Gr, \Conf}\ar[lu]\\
\Ran\times \overline{S}_{\Gr}^{0} \ar[d]^{\pr_{\Ran}}\ar[r]^{\unit_{\Gr}}&(\overline{S}^{0}_{\Gr, \Ran})&(\overline{S}^{0}_{\Gr, \Conf})\cap \overline{S}^{-, \Conf}_{\Gr, \Conf}\ar[dd]^{v_{\Conf}}\ar[lu]\ar[u]\\
\overline{S}_{\Gr}^{0}&&\\
&&\Conf.\\
&&}
\end{equation}
\end{center}

By applying the same steps (1), (2) as in the above construction (with a tiny modification: replace affine flags by the affine Grassmannian), we get a functor
\begin{equation}
    \sprd_{\Gr}: \Whit_q(\overline{S}^0_{\Gr})\to \Shv_{\mathcal{G}^G}(\overline{S}^{0}_{\Gr, \Conf}).
\end{equation}

\begin{definition}
The functor $$F^L_{\Gr}: \Whit_q(\overline{S}^0_{\Gr})\to \Shv_{\mathcal{G}^\Lambda}(\Conf)$$
is defined as
\begin{equation}\label{15.7}
    \mathcal{F}\mapsto v_{\Gr,*}(\sprd_{\Gr}(\mathcal{F})|_{(\overline{S}^{0}_{\Gr, \Conf})}\overset{!}{\otimes} j_!({\omega_{S^{-, \Conf}_{\Gr, \Conf}}})|_{(\overline{S}^{0}_{\Gr, \Conf})\cap\overline{S}^{-, \Conf}_{\Gr, \Conf}})[\deg].
\end{equation}
\end{definition}
Recall that $\Whit_q(\overline{S}^0_{\Gr})\simeq \Whit_q(S^0_{\Gr})\simeq \Vect$, hence, there exists a unique irreducible Whittaker sheaf on $\overline{S}^0_{\Gr}$. We denote it by $\mathcal{F}_0$. Set $\Omega_q^{L,'}:= F^L_{\Gr}(\mathcal{F}_0)$. The following lemma is proved in \cite[Theorem 6.2.5]{[Ga6]}.

\begin{lem}\label{15.1}
In the setting of D-modules, when $q$ avoids small torsion, there is an isomorphism of factorization algebras $\Omega_q^{L,'}\simeq \Omega_q^L$.
\end{lem}

\begin{prop}\label{prop 15.2}
Given any $\mathcal{F}\in \Whit_q(\Fl_G^{\omega^\rho})$, $F^L(\mathcal{F})$ has a naturally defined $\Omega_q^{L,'}$-factorization module structure.
\end{prop}

\begin{proof}
By the factorization property of $\overline{S}^{0}_{\Gr, \Conf}$, $\overline{S}^{-, \Conf}_{\Gr, \Conf}$, $(\overline{S}^{w_0}_{\Fl, \Conf_{\infty\cdot x}})_{\infty\cdot x}$, and $ \overline{S}^{-, \Conf_{\infty\cdot x}}_{\Fl, \Conf_{\infty\cdot x}}$ (see \eqref{conf 0 fact Gr}, \eqref{conf 0 fact}, \eqref{fact conf - w}, and \eqref{fact conf - w '}), we obtain that the prestack $\overline{S}^{0}_{\Gr, \Conf}\cap \overline{S}^{-, \Conf}_{\Gr, \Conf}$ is factorizable, and the prestack $(\overline{S}^{w_0}_{\Fl, \Conf_{\infty\cdot x}})_{\infty\cdot x}\cap \overline{S}^{-, \Conf_{\infty\cdot x}}_{\Fl, \Conf_{\infty\cdot x}}$ is a factorization module space with respect to $\overline{S}^{0}_{\Gr, \Conf}\cap \overline{S}^{-, \Conf}_{\Gr, \Conf}$. Note that $v_{\Conf}$ and $v_{\Conf_{\infty\cdot x}}$ are compatible with the factorization structures on $\overline{S}^{0}_{\Gr, \Conf}\cap \overline{S}^{-, \Conf}_{\Gr, \Conf}$ and $(\overline{S}^{w_0}_{\Fl, \Conf_{\infty\cdot x}})_{\infty\cdot x}\cap \overline{S}^{-, \Conf_{\infty\cdot x}}_{\Fl, \Conf_{\infty\cdot x}}$, hence, it suffices to show that  $$\sprd_{\Fl}(\mathcal{F})|_{((\overline{S}^{w_0}_{\Fl, \Conf_{\infty\cdot x}})_{\infty \cdot x})}\overset{!}{\otimes} j_!({\omega_{S^{-,1, \Conf_{\infty\cdot x}}_{\Fl, \Conf_{\infty\cdot x}}}})$$ factorizes with respect to $\sprd_{\Gr}(\mathcal{F}_0)|_{(\overline{S}^{0}_{\Gr, \Conf})}\overset{!}{\otimes} j_!({\omega_{S^{-, \Conf}_{\Gr, \Conf}}})$.

According to Corollary \ref{fact bas}, the Whittaker sheaf $\sprd_{\Fl, \Ran_x}(\mathcal{F})$ on $(\overline{S}^{w_0}_{\Fl, \Ran_x})_{\infty\cdot x}$ factorizes with respect to the factorization algebra $\Vac$. Since the $!$-pullback from $(\overline{S}^{w_0}_{\Fl, \Ran_x})_{\infty\cdot x}$ to $(\overline{S}^{w_0}_{\Fl, \Conf_{\infty\cdot x}})_{\infty\cdot x}$ is compatible with the factorization structure and $\sprd_{\Gr}(\mathcal{F}_0)$ is exactly the pullback of $\Vac$, we obtain that $\sprd_{\Fl}(\mathcal{F})$ is a factorization module over $\sprd_{\Gr}(\mathcal{F}_0)$. 

By (\ref{fact-}), $j_!(\omega_{S^{-,1, \Conf_{\infty\cdot x}}_{\Fl, \Conf_{\infty\cdot x}}})$ factorizes with respect to $j_!(\omega_{S^{-, \Conf}_{\Gr, \Conf}})$.

Now Proposition \ref{prop 15.2} follows from the fact that the tensor product of factorization modules is a factorization module over the tensor product of the corresponding factorization algebras.
\end{proof}

The functor $F^L$ defined in \eqref{6.35} factors through $\Omega_q^{L,'}-\Fact$. We will also denote by $F^L$ the resulting functor
\begin{equation}\label{6.42}
    F^L: \Whit_q(\Fl_G^{\omega^\rho})\to \Omega^{L,'}_{q}-\Fact.
\end{equation}

Of course, Theorem \ref{main theorem 1} can be deduced from the following stronger statement.

\begin{thm}\label{strong thm}
For any $q$, the functor $F^L$ in \eqref{6.42} is a t-exact equivalence that preserves standards and costandards. 
\end{thm}

\subsubsection{}
By constructions similar to $F^L$ and $F^L_{\Gr}$, if we replace the semi-infinite $!$-extension sheaves $j_!({\omega_{S^{-, \Conf}_{\Gr, \Conf}}})$ and $j_!({\omega_{S^{-,1, \Conf_{\infty\cdot x}}_{\Fl, \Conf_{\infty\cdot x}}}})$ by the semi-infinite $*$-extension sheaves, then we can define the following functors.

\begin{definition}
\begin{equation}
\begin{split}
        F^{DK}: \Whit_q(\Fl_{G,x}^{\omega^\rho})&\to \Shv_{\mathcal{G}^\Lambda}(\Conf_{\infty\cdot x})\\
        \mathcal{F}&\to v_{\Conf_{\infty\cdot x}, *}(\sprd_{\Fl}(\mathcal{F})|_{(\overline{S}^{w_0}_{\Fl, \Ran_x})_{\infty \cdot x}}\overset{!}{\otimes} j_*({\omega_{S^{-,1, \Conf_{\infty\cdot x}}_{\Fl, \Conf_{\infty\cdot x}}}}))[\deg],
\end{split}
\end{equation}
\begin{equation}\label{15.9}
    \begin{split}
        F^{DK}_{\Gr}: \Whit_q(\overline{S}^0_{\Gr})&\to \Shv_{\mathcal{G}^\Lambda}(\Conf)\\
        \mathcal{F}&\to v_{\Gr,*}(\sprd_{\Gr}(\mathcal{F})|_{(\overline{S}^{0}_{\Gr, \Conf})}\overset{!}{\otimes} j_*({\omega_{S^{-, \Conf}_{\Gr, \Conf}}}))[\deg].
    \end{split}
\end{equation}
\end{definition}

Similarly, we define $\Omega_q^{DK,'}:= F^{DK}_{\Gr}(\mathcal{F}_0)$. When $q$ avoids small torsion, we have $\Omega_q^{DK,'}\simeq \Omega_q^{DK}$ (\cite[Theorem 3.6.2]{[Ga6]}). By the same proof as that of Proposition \ref{prop 15.2}, we have

\begin{prop}\label{prop 15.3}
$F^{DK}$ factors through $\Omega_q^{DK,'}-\Fact$, i.e., it gives rise to a functor
\begin{equation}
   F^{DK}: \Whit_q(\Fl_{G,x}^{\omega^\rho})\to \Omega_q^{DK,'}-\Fact.
\end{equation}
\end{prop}

\subsection{{Calculation of the $!$-stalks of $F^L$ and $F^{DK}$}{}}\label{! fiber}
By Lemma \ref{4.1.1}, in the category of $\Omega-\Fact$, the standard object $\Delta_{\lambda, \Omega}$ is always uniquely characterized by the requirement that its $*$-stalks at $\mu\cdot x, \mu\in \Lambda$ is $\mathsf{e}$ if $\lambda=\mu$ and $0$ otherwise, and the costandard object $\nabla_{\lambda, \Omega}$ is uniquely characterized by the requirement that its $!$-stalks at $\mu\cdot x, \mu\in \Lambda$ is $\mathsf{e}$ if $\lambda=\mu$ and $0$ otherwise. Hence, in order to prove that $F^L$ sends standard objects to standard objects, costandard objects to costandard objects, we only need to find an explicit expression of the $!$-stalks and $*$-stalks of the image of $F^L$. 

The theory of sheaves on prestack is friendly with taking $!$-stalks. There are two reasons for this: the first one is that the $!$-pullback functor is always well-defined, the second one is that we have a base change theorem for $!$-pullback (see \cite[Corollary 3.1.4]{[GR2]}), hence, the calculation will be much easier than the calculation of $*$-stalks.

In this section, we will give an explicit formula (Proposition \ref{15.4 prop}) for the $!$-stalks of $F^L$ and $F^{DK}$ at $\lambda\cdot x$.








\subsubsection{}Consider the following Cartesian diagram:

\begin{center}
\begin{equation}
    \xymatrix{
\overline{S}^{-, \lambda}_{\Fl,x}\ar[r]\ar[d]&\overline{S}_{\Fl, \Conf_{\infty\cdot x}}^{-, \Conf_{\infty\cdot x}}\cap (\overline{S}^{w_0}_{\Fl, \Conf_{\infty\cdot x}})_{\infty\cdot x}\ar[d]^{v_{\Conf_{\infty\cdot x}}}\\
\lambda\cdot x\ar[r]^{i_\lambda}& \Conf_{\infty\cdot x}.
}
\end{equation}
\end{center}

Choosing a trivialization of the fiber of $\mathcal{G}^\Lambda$ at $\lambda\cdot x$. By the base change theorem (ref. \cite[Corollary 3.1.4]{[GR2]}), we have
 \begin{equation}\label{15.13}
 \begin{split}
     i_{\lambda}^!(F^L(\mathcal{F}))\simeq & i_{\lambda}^{!}\circ v_{\Conf_{\infty\cdot x},*}(\sprd_{\Fl}(\mathcal{F})|_{((\overline{S}^{w_0}_{\Fl, \Conf_{\infty\cdot x}})_{\infty \cdot x})}\overset{!}{\otimes} j_!({\omega_{S^{-,1, \Conf_{\infty\cdot x}}_{\Fl, \Conf_{\infty\cdot x}}}}))[\langle \lambda, 2\check{\rho}\rangle]\\
    \simeq & H(\Fl_{G,x}^{\omega^\rho}, i_{\lambda, \Conf_{\infty\cdot x}}^{!}(\sprd_{\Fl}(\mathcal{F}))\mathop{\otimes}\limits^!  i_{\lambda, \Conf_{\infty\cdot x}}^{-, \Conf_{\infty\cdot x},!}(j_!({\omega_{S^{-,1, \Conf_{\infty\cdot x}}_{\Fl, \Conf_{\infty\cdot x}}}})) [\langle \lambda, 2\check{\rho}\rangle]).
\end{split}
 \end{equation}

In the above formula,
\begin{itemize}
    \item $i_{\lambda, \Conf_{\infty\cdot x}}^{-, \Conf_{\infty\cdot x}}$ denotes the embedding of $\overline{S}^{-, \lambda}_{\Fl,x}$ into $\overline{S}_{\Fl, \Conf_{\infty\cdot x}}^{-, \Conf_{\infty\cdot x}}$,
   \item $i_{\lambda, \Conf_{\infty\cdot x}}$ denotes the embedding of $\Fl_{G,x}^{\omega^\rho}$ into $\overline{S}_{\Fl, \Conf_{\infty\cdot x}}^{-, \Conf_{\infty\cdot x}}$.
\end{itemize}

First, by the construction (\ref{def of sprd}), there is
\[ i_{\lambda, \Conf_{\infty\cdot x}}^{!}(\sprd_{\Fl}(\mathcal{F}))\simeq \mathcal{F}\qquad  \forall \lambda\in \Lambda\ \textnormal{and}\ \mathcal{F}\in \Whit_q(\Fl_{G,x}^{\omega^\rho}).\]

Second, {by Corollary \ref{cor A}}, we have a base change theorem for $!$-pushforward and $!$-pullback for semi-infinite sheaves. Namely, there is an isomorphism
\[i_{\lambda, \Conf_{\infty\cdot x}}^{-, \Conf_{\infty\cdot x},!}(j_!({\omega_{S^{-,1, \Conf_{\infty\cdot x}}_{\Fl, \Conf_{\infty\cdot x}}}}))\simeq j_!(\omega_{S_{\Fl,x}^{-,\lambda}})\qquad \forall \lambda\in \Lambda.\]

Here, $\omega_{S_{\Fl,x}^{-,\lambda}}$ denotes the {$(\cG^G)^{-1}\otimes\cG^\Lambda|_{\lambda\cdot x}$-} twisted dualizing sheaf on $S_{\Fl,x}^{-,\lambda}$ and $j_!(\omega_{S_{\Fl,x}^{-,\lambda}})$ denotes its $!$-pushforward to $\overline{S}_{\Fl,x}^{-,\lambda}$. {By choosing a trivialization of the fiber $\cG^\Lambda|_{\lambda\cdot x}$ (which is equivalent to choosing a trivialization of $\cG^G|_{t^\lambda\in \Fl}$), we can think $j_!(\omega_{S_{\Fl,x}^{-,\lambda}})$ as the $!$-extension of the {$(\cG^G)^{-1}$-} twisted dualizing sheaf on $S_{\Fl,x}^{-,\lambda}$ under the unique $N^-(\cK)^{\omega^\rho}_x$-equivariant trivialization.


}

From the above observations, we deduce the following proposition.

\begin{prop}\label{15.4 prop}
Choosing a trivialization of $\mathcal{G}^\Lambda|_{\lambda\cdot x}$, there exists an isomorphism
\begin{equation}
    i_{\lambda}^!(F^L(\mathcal{F}))\simeq H(\Fl_{G,x}^{\omega^\rho}, \mathcal{F}\mathop{\otimes}\limits^! j_!(\omega_{S_{\Fl,x}^{-,\lambda}})[\langle \lambda, 2\check{\rho}\rangle] ).
\end{equation}
\end{prop}

Similarly, we have the following proposition.

\begin{prop}\label{corep*H}
Choosing a trivialization of $\mathcal{G}^\Lambda|_{\lambda\cdot x}$, there exists an isomorphism
\begin{equation}
    i_{\lambda}^!(F^{DK}(\mathcal{F}))\simeq H(\Fl_{G,x}^{\omega^\rho}, \mathcal{F}\mathop{\otimes}\limits^! j_*(\omega_{S_{\Fl,x}^{-,\lambda}})[\langle \lambda, 2\check{\rho}\rangle] ).
\end{equation}
\end{prop}

\subsubsection{}
The following corollary relates the functor $i_\lambda^!(F^L)$ with the standard objects that we constructed in Definition \ref{stan def}.

\begin{cor}\label{6.4.3}
Given $\lambda\in \Lambda$ and a trivialization of $\mathcal{G}^\Lambda|_{\lambda\cdot x}$, there exists an isomorphism
\begin{equation}\label{corep}
    i_{\lambda}^!(F^L(\mathcal{F}))\simeq \mathcal{H}om_{\Whit_q(\Fl_G^{\omega^\rho})}(\Delta_\lambda, \mathcal{F}).
\end{equation}
\end{cor}
\begin{proof}
According to Proposition \ref{15.4 prop}, we have to prove
$$H(\Fl_{G,x}^{\omega^\rho}, \mathcal{F}\mathop{\otimes}\limits^! j_!(\omega_{S_{\Fl,x}^{-,\lambda}})[\langle \lambda, 2\check{\rho}\rangle] )\simeq \mathcal{H}om_{\Whit_q(\Fl_G^{\omega^\rho})}(\Delta_\lambda, \mathcal{F}).$$


By the assumption, $\mathcal{F}$ is $(N(\mathcal{K})_x^{\omega^\rho}, \chi)$-equivariant. In particular, it is $(N(\mathcal{O})_x^{\omega^\rho}, \chi)$-equivariant. 

Note that $\chi |_{N(\mathcal{O})_x^{\omega^\rho}}$ is trivial, we have
\begin{align*}
    H(\Fl_{G,x}^{\omega^\rho}, \mathcal{F}\mathop{\otimes}\limits^! j_!(\omega_{S_{\Fl,x}^{-,\lambda}})[\langle \lambda, 2\check{\rho}\rangle] )=& H(\Fl_{G,x}^{\omega^\rho}, \mathcal{F}\mathop{\otimes}\limits^!\Av_*^{N(\mathcal{O})^{\omega^\rho}}( j_!(\omega_{S_{\Fl,x}^{-,\lambda}})[\langle \lambda, 2\check{\rho}\rangle] )).
    \end{align*}   
  By Proposition \ref{! corr iwahori} below, we have
    \begin{align*}
    H(\Fl_{G,x}^{\omega^\rho}, \mathcal{F}\mathop{\otimes}\limits^! \Av_*^{N(\mathcal{O})^{\omega^\rho}}(j_!(\omega_{S_{\Fl,x}^{-,\lambda}})[\langle \lambda, 2\check{\rho}\rangle] ))\simeq& H(\Fl_{G,x}^{\omega^\rho}, \mathcal{F}\mathop{\otimes}\limits^! {{J}}_{\lambda}^{\mathbb{D}})
    \end{align*}
According to the construction of the convolution product in Section \ref{convolution product}, it is isomorphic to the $!$-stalks of the convolution product $\mathcal{F}\star ({{J}}_{-\lambda})_{\lambda}$ at $t^0\in \widetilde{\Fl}$. Furthermore, there exist isomorphisms
    \begin{align*}& \mathcal{H}om_{\Shv_{\mathcal{G}^G}(\widetilde{\Fl})_{\lambda}}((\delta_0)_{\lambda}, \mathcal{F}\star ({{J}}_{-\lambda})_{\lambda})\\
    \simeq & \mathcal{H}om_{\Whit_q(\widetilde{\Fl})_\lambda}(\Av_!^{N(\mathcal{K}), \chi}((\delta_0)_{\lambda}), \mathcal{F}\star ({{J}}_{-\lambda})_\lambda)\\
    \simeq & \mathcal{H}om_{\Whit_q(\widetilde{\Fl})_\lambda}(\Av_!^{N(\mathcal{K}), \chi}(({{J}}_0)_\lambda), \mathcal{F}\star ({{J}}_{-\lambda})_\lambda)\\
    \simeq &\mathcal{H}om_{\Whit_q({\Fl_G^{\omega^\rho}})}(\Av_!^{N(\mathcal{K}), \chi}(({{J}}_0)_\lambda)\star {{J}}_\lambda, \mathcal{F})\\
    \simeq &\mathcal{H}om_{\Whit_q({\Fl_G^{\omega^\rho}})}(\Av_!^{N(\mathcal{K}), \chi}({{J}}_\lambda), \mathcal{F})\\
    \simeq &\mathcal{H}om_{\Whit_q(\Fl_G^{\omega^\rho})}(\Delta_\lambda, \mathcal{F}).
\end{align*}

\end{proof}

    \begin{prop}\label{! corr iwahori}
Given $\lambda \in \Lambda$, there is an isomorphism
\begin{equation}
    \Av_*^{N(\mathcal{O})^{\omega^\rho}}(j_!(\omega_{S_{\Fl,x}^{-,\lambda}})[\langle \lambda, 2\check{\rho}\rangle])\simeq {{J}}^{\mathbb{D}}_{\lambda}.
\end{equation}
\end{prop}
\begin{proof} 
{We regard $j_!(\omega_{S_{\Fl,x}^{-,\lambda}})$ as a $(\cG^G)^{-1}$-twisted sheaf. Thus, it is the $!$-extension of the $(\cG^G)^{-1}$-twisted dualizing sheaf on $S_{\Fl,x}^{-,\lambda}$ under the unique (up to a choice of trivialization of $(\cG^G)^{-1}|_{t^\lambda\in \Fl}$) $N^-(\cK)^{\omega^\rho}_x$-equivariant trivialization.

Assuming $\alpha$ to be very dominant, we have $t^\alpha I^{\omega^\rho} t^{-\alpha+\lambda} I^{\omega^\rho}/I^{\omega^\rho}\subset N^-(\cK)^{\omega^\rho} t^{\lambda} I^{\omega^\rho}/I^{\omega^\rho}$, and $t^\alpha I^{\omega^\rho} t^{-\alpha+\lambda} I^{\omega^\rho}/I^{\omega^\rho}=t^\alpha N^-(t\cO)^{\omega^\rho} t^{-\alpha+\lambda} I^{\omega^\rho}/I^{\omega^\rho}$. Here, $N^-(t\cO)^{\omega^\rho}$ denotes the negative part of $I^{\omega^\rho}$.

Since $I^0$ is pro-unipotent, the twisting $(\cG^G)^{-1}$ on $I^{\omega^\rho} t^{-\alpha+\lambda} I^{\omega^\rho}/I^{\omega^\rho}=I^{0} t^{-\alpha+\lambda} I^{\omega^\rho}/I^{\omega^\rho}$ has a unique (up to a choice of trivialization of $(\cG^G)^{-1}|_{t^{-\alpha+\lambda}\in \Fl}$) $I^0$-equivariant trivialization. By definition, $J_{-\alpha+\lambda,!}$ is the $!$-extension of the $(\cG^G)^{-1}$-twisted constant perverse sheaf on $I^{\omega^\rho} t^{-\alpha+\lambda} I^{\omega^\rho}/I^{\omega^\rho}$ with respect to this trivialization.

Given an element $g\in G(\cK)^{\omega^\rho}$ and a trivialization of $(\cG^G)^{-1}|_g$, we can define the left transition functor $g\cdot-:\Shv_{(\cG^G)^{-1}}(\Fl_G^{\omega^\rho})\to \Shv_{(\cG^G)^{-1}}(\Fl_G^{\omega^\rho})$. Now, we choose a trivialization of $(\cG^G)^{-1}$ at $t^\alpha\in T^{\omega^\rho}(\cK)$ such that it matches the chosen trivializations $(\cG^G)^{-1}|_{t^{-\alpha+\lambda}\in \Fl}$ and $(\cG^G)^{-1}|_{t^\lambda\in \Fl}$ under the isomorphism $(\cG^G)^{-1}|_{t^\alpha\in G(\cK)}\otimes (\cG^G)^{-1}|_{t^{-\alpha+\lambda}\in \Fl}\simeq (\cG^G)^{-1}|_{t^\lambda\in \Fl}$. 

Left-multipling with $t^\alpha$, the $I^0$-equivariant trivialization on $I^{\omega^\rho} t^{-\alpha+\lambda} I^{\omega^\rho}/I^{\omega^\rho}$ becomes the unique $\Ad_{\alpha}I^0$-equivariant trivialization on $t^\alpha I^{\omega^\rho} t^{-\alpha+\lambda} I^{\omega^\rho}/I^{\omega^\rho}$.\footnote{However, conjugating the canonical trivialization of $\cG^G$ on $T(\cO)^{\omega^\rho}$ by $t^\alpha$ will change the trivialization by the character sheaf $b_\alpha$. So, left-transition will change the $T(\cO)^{\omega^\rho}$-equivariant structure.} It coincides with the restriction of the unique $N^-(\cK)^{\omega^\rho}$-equivariant trivialization on $N^-(\cK)^{\omega^\rho} t^{\lambda} I^{\omega^\rho}/I^{\omega^\rho}$ (since different trivializations differ by a tame local system on the affine space $t^\alpha I^{\omega^\rho} t^{-\alpha+\lambda} I^{\omega^\rho}/I^{\omega^\rho}$, which has to be trivial). In particular, $t^\alpha {{J}}_{-\alpha+\lambda,!}$ exactly coincides with the $!$-extension of the $!$-restriction of $j_!(\omega_{S_{\Fl,x}^{-,\lambda}})$ on  $t^\alpha I^{\omega^\rho} t^{-\alpha+\lambda} I^{\omega^\rho}/I^{\omega^\rho}$ up to a shift by $\langle \alpha-\lambda, 2\check{\rho}\rangle$.  The adjointness of $!$-pushforward and $!$-pullback gives rise to transition maps between $t^\alpha {{J}}_{-\alpha+\lambda,!}[\langle \alpha-\lambda, 2\check{\rho}\rangle]$.
}

Now we write $j_!(\omega_{S_{\Fl,x}^{-,\lambda}})[\langle \lambda, 2\check{\rho}\rangle]$ as $\mathop{\colim}\limits_{{\alpha,}\alpha-\lambda\in \Lambda^+} t^\alpha {{J}}_{-\alpha+\lambda,!}[\langle \alpha, 2\check{\rho}\rangle]$. Note  that ${{J}}_{-\alpha+\lambda,!}$ is $I^0$-equivariant and $T$-equivariant with respect to a character {$b_{-\lambda+\alpha}$.} {F}or such a sheaf $\mathcal{F}$, we have \[\Av_*^{N(\mathcal{O})^{\omega^\rho}}(t^\alpha\cdot \mathcal{F})[\langle \alpha, 2\check{\rho}\rangle]\simeq ({{J}}_{\alpha,*})_{-\lambda+\alpha}{\star}\mathcal{F}.\]
{
Indeed, for any such a $(\cG^G)^{-1}$-twisted sheaf $\cF$, taking $*$-averaging of $t^\alpha\cdot \cF$ with respect to $N(\mathcal{O})^{\omega^\rho}$ is given by taking the convolution of $\cF$ with the $*$-extension of the twisted constant sheaf on $N(\mathcal{O})^{\omega^\rho} t^\alpha$. Since $\cF$ is $(I^{\omega^\rho}, b_{-\lambda+\alpha})$-equivariant, we can first take right $(I^{\omega^\rho}, b_{-\lambda+\alpha})$-averaging of the constant sheaf on $N(\mathcal{O})^{\omega^\rho} t^\alpha$ and then take the convolution with $\cF$ after descending along $G(\cK)^{\omega^\rho}\times \Fl_G^{\omega^\rho}\longrightarrow G(\cK)^{\omega^\rho}\overset{I^{\omega^\rho}}{\times} \Fl_G^{\omega^\rho}$. Up to a shift, the right $(I^{\omega^\rho}, b_{-\lambda+\alpha})$-averaging of the constant sheaf on $N(\mathcal{O})^{\omega^\rho} t^\alpha$ is isomorphic to the pullback of $(J_{\alpha,*})_{-\lambda+\alpha}$.
}

So there is
\begin{equation}
    \begin{split}
        \Av_*^{N(\mathcal{O})^{\omega^\rho}}(\mathop{\colim}\limits_{{\alpha,} \alpha-\lambda \in \Lambda^+} t^\alpha {{J}}_{-\alpha+\lambda,!}[\langle \alpha, 2\check{\rho}\rangle])\simeq& \mathop{\colim}\limits_{{\alpha,}\alpha-\lambda\in \Lambda^+} ({{J}}_{\alpha,*})_{-\lambda+\alpha}\star {{J}}_{-\alpha+\lambda,!}
        \simeq {{J}}^{\mathbb{D}}_{\lambda}.
    \end{split}
\end{equation}
Here we use the fact that $\Av^{N(\mathcal{O})^{\omega^\rho}}_*$ commutes with colimits.
\end{proof}

Similarly, we can calculate $N(\mathcal{O})^{\omega^\rho}$-averaging of $j_*(\omega_{S_{\Fl,x}^{-,\lambda}})$.

\begin{prop}\label{*eq}
For $\lambda \in \Lambda$, we have $$\Av_*^{N(\mathcal{O})^{\omega^\rho}}(j_*(\omega_{S_{\Fl,x}^{-,\lambda}})[\langle \lambda, 2\check{\rho}\rangle])=\mathop{\colim}\limits_{\alpha{,\alpha-\lambda}\in \Lambda^+} ({{J}}_{\alpha,*})_{{\lambda}-\alpha}\star {{J}}_{-\alpha+\lambda,*}.$$ 
\end{prop}

\subsection{Proof of Theorem \ref{thm 2.1} modulo Proposition \ref{imstan}}\label{Equivalence of categories}
This section will be devoted to the proof of Theorem~\ref{main theorem 1} using Proposition \ref{imstan}.

First of all, let us check the compatibility of costandard objects under $F^L$.

\begin{prop}\label{im costand}
For any $\lambda\in \Lambda$, there is an isomorphism
\begin{equation}
    F^L(\nabla_\lambda)\simeq \nabla_{\lambda, \Omega_{q}^{L,'}}.
\end{equation}
\end{prop}
\begin{proof}
By the isomorphism (\ref{corep}), there is
\begin{equation*}
        i_\mu^!(F^L(\nabla_\lambda))\simeq \mathcal{H}om_{\Whit_q(\Fl_G^{\omega^\rho})}(\Delta_\mu, \nabla_\lambda)
\end{equation*}

Now the claim follows from Proposition \ref{stan}.
\end{proof}

Then we prove Theorem \ref{strong thm} with the help of the following proposition. In particular, by Lemma \ref{15.1}, we actually prove Theorem \ref{main theorem 1} (=Theorem \ref{thm 2.1}).

\begin{prop}\label{imstan}
For $\lambda\in \Lambda$, the functor $$F^L: \Whit_q(\Fl_{G,x}^{\omega^\rho})\to \Omega_{q}^{L,'}-\Fact$$ sends standards to standards, i.e.,
$$F^L(\Delta_\lambda)=\Delta_{\lambda,\Omega^{L,'}_q}.$$
\end{prop}

\begin{proof} (of Theorem \ref{strong thm}).


To prove the fully faithfullness of $F^L$, we need to prove that the following map is an isomorphism
$$\mathcal{H}om_{\Whit_q(\Fl_{G,x}^{\omega^\rho})}(\mathcal{F}_1,\mathcal{F}_2){\overset{F^L}{\longrightarrow}} \mathcal{H}om_{\Omega_{q}^{L,'}-\Fact}(F^L(\mathcal{F}_1),F^L(\mathcal{F}_2)),$$
for any $\mathcal{F}_1, \mathcal{F}_2\in \Whit_q(\Fl_{G,x}^{\omega^\rho})$.

Since the standards $\{\Delta_\lambda, \lambda\in \Lambda\}$ generate the category $\Whit_q(\Fl_{G,x}^{\omega^\rho})$ by cohomology shifts, colimits and extensions, it is sufficient to prove
{
\begin{equation}\label{eq 6.6.2}
\mathcal{H}om_{\Whit_q(\Fl_{G,x}^{\omega^\rho})}(\Delta_{\lambda},\mathcal{F}_2)\overset{F^L}{\simeq} \mathcal{H}om_{\Omega_{q}^{L,'}-\Fact}(F^L(\Delta_{\lambda}),F^L(\mathcal{F}_2)),
\end{equation}
for any $\mathcal{F}_2\in \Whit_q(\Fl_{G,x}^{\omega^\rho})$
}

{Fix a $\Delta_{\lambda}$. Since both $\Delta_\lambda$ and $F^L(\Delta_\lambda)=\Delta_{\lambda,\Omega_q^{L,'}}$ are compact, we only need to construct a collection of compact generators $\{\Delta'_{\lambda, \widetilde{w}}, \widetilde{w} \text{\ relevant}\}$ of $\Whit_q(\Fl_{G,x}^{\omega^\rho})$ and prove that $F^L$ induces an isomorphism \eqref{eq 6.6.2} if $\cF_2= \Delta'_{\lambda, \widetilde{w}}$. 

Without loss of generality, we can assume that $\lambda$ is domininat, the construction of $\Delta'_{\lambda, \widetilde{w}}$ for general $\lambda$ follows by taking the convolution with twisted BMW sheaf. 

In this case, we let $\Delta'_{\lambda, \widetilde{w}}=\Delta_\lambda$ if $\widetilde{w}=t^\lambda$, and $\Delta'_{\lambda, \widetilde{w}}=\nabla^{\ver}_{\widetilde{w}}$ if $\widetilde{w}\neq t^\lambda$. It is not hard to see that they form a collection of compact generators. Indeed, according to Proposition \ref{naive standard generate}, $\{\nabla^{\ver}_{\widetilde{w}},  \widetilde{w} \text{\ relevant}\}$ is a collection of compact generators. Furthermore, $\nabla^{\ver}_{\lambda}$ is a finite extension of $\Delta_{\lambda,\lambda}'=\Delta_{\lambda}$ and $\nabla^{\ver}_{\widetilde{w}}$, for $\widetilde{w}\neq t^\lambda$.

If $\widetilde{w}\neq t^\lambda$, by Corollary \ref{6.4.3} and Proposition \ref{imstan}, we have $ \mathcal{H}om_{\Whit_q(\Fl_{G,x}^{\omega^\rho})}(\Delta_{\lambda},\mathcal{F}_2)=i_\lambda^!(F^L(\mathcal{F}_2))=\mathcal{H}om_{\Omega_{q}^{L,'}-\Fact}(F^L(\Delta_{\lambda}),F^L(\mathcal{F}_2))=0$. The map \eqref{eq 6.6.2} has to be $0$.

If $\widetilde{w}= t^\lambda$, by Proposition \ref{imstan}, both sides of \eqref{eq 6.6.2} are $\mathsf{e}$, the map \eqref{eq 6.6.2} is an isomorphism since $F^L$ sends $\id$ to $\id$.
}

{Then, we note that standards $\{\Delta_{\lambda, \Omega_q^{L,'}}\}$ generate the category $\Omega_{q}^{L,'}-\Fact$ under cohomology shifts, extensions, and colimits, and the functor $F^L$ is compatible with cohomology shift, extensions, and colimits. Hence, $F^L$ is essentially surjective by Proposition \ref{imstan}.}

To prove the t-exactness of $F^L$, note that according to \eqref{3.32} and \eqref{factt}, t-structures on both sides are defined by the 'Hom' with standard objects. According to Proposition \ref{imstan}, the functor $F^L$ preserves standards. Hence, $F^L$ is t-exact.
\end{proof}

\begin{cor} \label{heart}
The objects $\Delta_\lambda$ and $\nabla_\lambda$ are in the heart of $\Whit_q(\Fl_G^{\omega^\rho})$.
\end{cor}
\begin{proof}
Note that $F^L$ preserves standards and costandards, and $F^L$ is t-exact, we only need to prove that $\Delta_{\lambda, \Omega_q^{L,'}}$ and $\nabla_{\lambda, \Omega_q^{L,'}}$ are in the heart of the t-structure on $\Omega_q^{L,'}-\Fact$. The later claim follows from the fact that $\Omega_q^L$ is perverse, and all standard objects and costandards in $\Omega-\Fact$ are perverse if $\Omega$ is perverse.
\end{proof}

The proof of Proposition \ref{imstan} is hard. It will occupy the rest of the paper and finally be given in Section~\ref{proof of imstan}.

\section{Global Whittaker category}\label{Whittaker category: global definition}
From now on, we will focus on the proof of Proposition \ref{imstan}. As we noted before, it is very hard to calculate $*$-stalks. Luckily, we can use duality functor to transfer the calculation of $*$-stalks to a calculation of $!$-stalks. To make the calculation possible, we introduce the global counterparts of the category $\Whit_q(\Fl_G^{\omega^\rho})$ and the functor $F^L$. In this section, our aim is to transfer Proposition \ref{imstan} to Proposition \ref{imstan glob} by the local-global comparison. 

\subsection{Drinfeld compactifications}\label{Drinfeld compactification sec}
Fix $x\in X$. The Drinfeld compactification is introduced in \cite[Section 1]{[BG]}. In this section, we define the Whittaker category on the (Iwahori version) Drinfeld compactification. 
\begin{definition}\label{Drinfeld compact ' inf N}

Let $(\overline{\Bun_N^{\omega^\rho}})'_{\infty\cdot x}$ be the stack classifying the triples $(\mathcal{P}_G,\{\kappa^{\check{\lambda}}, \check{\lambda}\in \Lambda^+\}, \epsilon)$,  where $\mathcal{P}_G\in \Bun_G$, $\{\kappa^{\check{\lambda}}, \check{\lambda}\in \Lambda^+\}$ is a family of morphisms of coherent sheaves \[\kappa^{\check{\lambda}}:(\omega^{\frac{1}{2}})^{\langle\check{\lambda}, 2\rho\rangle}\to \mathcal{V}_{\mathcal{P}_G}^{\check{\lambda}}(\infty\cdot x)\qquad \forall \check{\lambda}\in \Lambda^+\] 
which satisfy the Pl\"{u}cker relations, such that it is regular over $X\setminus x$, and $\epsilon$ is a $B$-reduction of $\mathcal{P}_G$ at $x$.

If we omit the Iwahori structure (i.e., $\epsilon$) at $x$ and ask $\kappa^{\check{\lambda}}$ to be defined and regular on the whole curve $X$ for any dominant weight $\check{\lambda}$, we will denote the resulting algebraic stack by $\overline{\Bun_N^{\omega^\rho}}$.
\end{definition}

\subsubsection{}
Note that there is a projection map from $\Fl_{G,x}^{\omega^\rho}$ to $(\overline{\Bun_N^{\omega^\rho}})_{\infty\cdot x}'$,
\begin{equation}\label{12.17}
    \pi_{\Fl, x}: \Fl_{G,x}^{\omega^\rho}\to (\overline{\Bun_N^{\omega^\rho}})_{\infty\cdot x}'.
\end{equation}

This morphism sends $(\mathcal{P}_G, \alpha,\epsilon)\in \Fl_{G,x}^{\omega^\rho}$ to $(\mathcal{P}_G, \{\kappa^{\check{\lambda}}, \check{\lambda}\in \check{\Lambda}^+\}, \epsilon)$. Here $\kappa^{\check{\lambda}}$ is induced from $\alpha$, i.e., for any dominant weight $\check{\lambda}$, $$\kappa^{\check{\lambda}}:(\omega^{\frac{1}{2}})^{\langle \check{\lambda}, 2\check{\rho}\rangle}\to \mathcal{V}_{\mathcal{P}^\omega}^{\check{\lambda}}\overset{\alpha^{-1}}{\to} \mathcal{V}_{\mathcal{P}_G}^{\check{\lambda}}.$$

Similarly, by omitting $\epsilon$, we have a projection map from $\overline{S}^0_{\Gr}$ to $\overline{\Bun_N^{\omega^\rho}}$,

\begin{equation}\label{12.18}
    \pi_{\Gr, x}: \overline{S}^0_{\Gr}\to \overline{\Bun_N^{\omega^\rho}}.
\end{equation}

\subsubsection{} $(\overline{\Bun_N^{\omega^\rho}})'_{\infty\cdot x}$ and $\overline{\Bun_N^{\omega^\rho}}$ project to $\Bun_G$. By taking the ratio of pullback of $\mathcal{G}^G$ on $\Bun_G$ and the fiber $\mathcal{G}^G|_{\mathcal{P}_G^{\omega^\rho}\in \Bun_G}$, we get gerbes on $(\overline{\Bun_N^{\omega^\rho}})'_{\infty\cdot x}$ and $\overline{\Bun_N^{\omega^\rho}}$. We denote the resulting the gerbes by $\mathcal{G}^G$. By constructions in Section \ref{gerbe used},  their pullbacks along the projections \eqref{12.17} and \eqref{12.18} are isomorphic to the same-named gerbes on $\Fl_{G}^{\omega^\rho}$ and $\Gr_G^{\omega^\rho}$.

\subsection{Global Whittaker category}\label{global whittaker category}
In \cite{[FGV]}, the authors defined the category $\Whit_{q}(\overline{\Bun_N^{\omega^\rho}})$. We can define the twisted Whittaker category on $(\overline{\Bun_N^{\omega^\rho}})'_{\infty\cdot x}$  similarly. 

Given a point $\bar{y}=\{y_1, y_2,\cdots, y_n\}$ in $\Ran$, which is disjoint from $x$, i.e., $x\neq y_i$ for any $i$.
\begin{definition}
We define $((\overline{\Bun_N^{\omega^\rho}})'_{\infty\cdot x})_{good\ at\ \bar{y}}$ as the open substack of $(\overline{\Bun_N^{\omega^\rho}})'_{\infty\cdot x}$, such that for any dominant weight $\check{\lambda}$, the map $\kappa^{\check{\lambda}}$ is injective on the fiber over any point $y_i\in \bar{y}$.
\end{definition}

 Since $\{\kappa^{\check{\lambda}}, \lambda\in \Lambda^+\}$ are injective bundle maps near $\bar{y}$, they give rise to a $N^{\omega^\rho}$-reduction of $\mathcal{P}_G$ near $\bar{y}$, which means there exists a $B$-bundle $\mathcal{P}_B$ on the disk $\mathcal{D}_{\bar{y}}$, such that $$\mathcal{P}_G|_{\mathcal{D}_{\bar{y}}}\simeq \mathcal{P}_B|_{\mathcal{D}_{\bar{y}}}\mathop{\times}\limits^B G,$$ and $\beta_{\bar{y}}^T: \mathcal{P}_B\mathop{\times}\limits^B T\simeq \omega^\rho$.
 
Similar to \cite[Section 2.3]{[Ga2]}, we construct a $N(\mathcal{O})_{\bar{y}}^{\omega^\rho}$-principal bundle $_{\bar{y}}((\overline{\Bun_N^{\omega^\rho}})'_{\infty\cdot x})_{good\ at\ \bar{y}}$ over the stack $((\overline{\Bun_N^{\omega^\rho}})'_{\infty\cdot x})_{good\ at\ \bar{y}}$. This bundle classifies a data from $((\overline{\Bun_N^{\omega^\rho}})'_{\infty\cdot x})_{good\ at\ \bar{y}}$ plus a choice of identification of the $B$-bundle $\mathcal{P}_B|_{\mathcal{D}_{\bar{y}}}$ with the $B$-bundle induced from $\omega^\rho$, such that it is compatible with $\beta_{\bar{y}}^T$. 

By a standard gluing procedure (see \cite[Lemma 3.2.7]{[FGV]}), we can extend the $N(\mathcal{O})_{\bar{y}}^{\omega^\rho}$-action on $_{\bar{y}}((\overline{\Bun_N^{\omega^\rho}})'_{\infty\cdot x})_{good\ at\ \bar{y}}$ to an action of  $N(\mathcal{K})_{\bar{y}}^{\omega^\rho}$. 

\begin{definition}\label{globdef}
A twisted Whittaker sheaf on $(\overline{\Bun_N^{\omega^\rho}})'_{\infty\cdot x}$ is a twisted sheaf on $(\overline{\Bun_N^{\omega^\rho}})'_{\infty\cdot x}$ such that its pullback to ${}_{\bar{y}}((\overline{\Bun_N^{\omega^\rho}})'_{\infty\cdot x})_{good\ at\ \bar{y}}$ is $(N(\mathcal{K})^{\omega^\rho}_{\bar{y}}, {-}\chi_{\bar{y}})$-equivariant for any $\bar{y}$ disjoint with $x$. We denote the category of $\mathcal{G}^G$-twisted Whittaker sheaf on $(\overline{\Bun_N^{\omega^\rho}})'_{\infty\cdot x}$ by $\Whit_{q}((\overline{\Bun_N^{\omega^\rho}})_{\infty\cdot x}')$.
\end{definition}

Applying the method of the proof of \cite[Theorem 5.2.2]{[Ga5]}, we have
\begin{lem}\label{thm local glob}
$$\pi_{\Fl, x}^!: \Whit_{q}((\overline{\Bun_N^{\omega^\rho}})_{\infty\cdot x}') \to \Whit_q(\Fl_{G,x}^{\omega^\rho})$$ is an equivalence of categories.
 $$\pi_{\Gr, x}^!: \Whit_{q}(\overline{\Bun_N^{\omega^\rho}}) \to \Whit_q(\overline{S}^0_{\Gr})$$ is an equivalence of categories.
\end{lem}                             

We denote by $\Delta_{\glob}^{\lambda}$ the twisted sheaf $\pi_{\Fl,x}^!(\Delta_\lambda)[d_g]$, where 
$d_g:= \dim(\Bun_N^{\omega^\rho})$.
\begin{rem}
Although the local Whittaker categories are equivalent to the global Whittaker categories, we have to use both of them in this paper: we use the local Whittaker category to show the factorization property, and we use the global Whittaker category to show prove Proposition \ref{imstan}.  
\end{rem}

{
\subsubsection{Proof of Lemma \ref{thm local glob}}\label{section 7.2.5}
The second claim of Lemma \ref{thm local glob} is the statement of \cite[Theorem 5.2.2]{[Ga5]}. To be self-contained, we prove the first claim.

First, we give the algebraic ind-stack $(\overline{\Bun_N^{\omega^\rho}})_{\infty\cdot x}'$ a stratification. By definition, the algebraic ind-stack $(\overline{\Bun_N^{\omega^\rho}})_{\infty\cdot x}'$ classifies the data of points $(\cP_G, \sigma, \epsilon)$, where $\cP_G$ is a $G$-bundle on $X$, $\sigma$ is a section of $X-x$ in $T\backslash 
\overline{N\backslash G}^{\aff}\overset{G}{\times} \cP_G$ which generically lies in $
{T\backslash (N\backslash G)}\overset{G}{\times} \cP_G$, such that the induced map to $T\backslash \pt$ is given by $\omega^{\rho}$, and $\epsilon$ is a $B$-reduction of $\cP_G$ at $x$.

According to \cite[Lemma 1.3.7]{[Zhu]}, we can trivialize $\cP_G$ on the formal disc $\cD_x$. Furthermore, we can choose an isomorphism $\cP_G|_{\cD_x}\overset{\phi}{\simeq} \cP_G^\omega$, such that the Iwahori structure $\epsilon$ goes to $\omega^\rho\overset{T}{\times} B|_x$ under this isomorphism.

For any geometric point in $(\overline{\Bun_N^{\omega^\rho}})_{\infty\cdot x}'$, taking the restriction of $\sigma$ to $\overset{\circ}{\cD}_x$, we obtain a map $\sigma|_{\overset{\circ}{\cD}_x}:\overset{\circ}{\cD}_x \to T\backslash (N\backslash G)\overset{G}{\times}\cP_G\simeq T\backslash(N\backslash G)\overset{G}{\times}\cP^{\omega}_G$, such that the induced map to $T\backslash \pt$ is given by $\omega^\rho$. It gives rise to a point in $(N\backslash G)^{\omega^\rho}(\cK)$. Also, different identifications $\phi$ preserving the $B$-reduction structure at $x$ differ by a multiplication by Iwahori $I^{\omega^\rho}$. Hence, for any geometric point in $(\overline{\Bun_N^{\omega^\rho}})_{\infty\cdot x}'$, we can obtain a point in $|(N\backslash G)^{\omega^\rho}(\cK)/I^{\omega^\rho}|=|N(\cK)^{\omega^\rho}\backslash \Fl_G^{\omega^\rho}|=W^{\ext}$. 

For any $t^\lambda w\in W^{\ext}$, we denote the corresponding locally closed substack by $(\overline{\Bun_N^{\omega^\rho}})_{=\lambda\cdot x}^w$. It has an open substack $({\Bun_N^{\omega^\rho}})_{=\lambda\cdot x}^w$ where we require $\kappa^{\check{\lambda}}$ to be injective on $X-x$.

The projection $\pi_{\Fl, x}$ induces a map for each stratum:
\begin{equation}
    \pi_{\Fl, x}: 
S^{t^\lambda w}_{\Fl, x}\longrightarrow (\overline{\Bun_N^{\omega^\rho}})_{=\lambda\cdot x}^w. 
\end{equation}

We claim that $\pi_{\Fl,x}^!$ induces a strata-wise equivalence, i.e., 
\begin{equation}
  \pi_{\Fl, x}^!:  \Whit_q((\overline{\Bun_N^{\omega^\rho}})_{=\lambda\cdot x}^w)\overset{\sim}{\longrightarrow} \Whit_q(S^{t^\lambda w}_{\Fl, x}).
\end{equation}

First, using the same proof as \cite[Lemma 6.2.8]{[FGV]}, one can show that $(\overline{\Bun_N^{\omega^\rho}})_{=\lambda\cdot x}^w- ({\Bun_N^{\omega^\rho}})_{=\lambda\cdot x}^w$ does not carry non-zero Whittaker sheaf. 

Furthermore, there is an isomorphism of stacks
\begin{equation}\label{eq 7.2.3}
    ({\Bun_N^{\omega^\rho}})_{=\lambda\cdot x}^w\simeq N^{\omega^\rho}_{X-x}\backslash S^{t^\lambda w}_{\Fl,x},
\end{equation}
where $N^{\omega^\rho}_{X-x}$ is the mapping space $X-x\longrightarrow N^{\omega^\rho}$ and the right-hand side of \eqref{eq 7.2.3} is understood as the fpqc (equivalently, \'{e}tale) sheafification of the prestack quotient.

Also, for any geometric point $\bar{y}=\{y_1, y_2, \cdots, y_n\}\in \Ran_{X-x}$, we have
\begin{equation}
   ({\Bun_N^{\omega^\rho}})_{=\lambda\cdot x, good\ at\ \bar{y}}^w\simeq ({\Bun_N^{\omega^\rho}})_{=\lambda\cdot x}^w\simeq N^{\omega^\rho}_{X-\bar{y}-x}\backslash \prod_{i=1}^{n} S^0_{\Gr, y_i}\times S^{t^\lambda w}_{\Fl,x},
\end{equation}
where $N^{\omega^\rho}_{X-\bar{y}-x}$ acts on $\prod_{i=1}^{n} S^0_{\Gr, y_i}\times S^{t^\lambda w}_{\Fl,x}$ diagonally.

It follows immediately that
\begin{equation}
   {}_{\bar{y}}({\Bun_N^{\omega^\rho}})_{=\lambda\cdot x, good\ at\ \bar{y}}^w\simeq N^{\omega^\rho}_{X-\bar{y}-x}\backslash \prod_{i=1}^{n} N(\cK)_{y_i}^{\omega^\rho}\times S^{t^\lambda w}_{\Fl,x},
\end{equation}
and the $(N(\cK)^{\omega^\rho}_{\bar{y}},-\chi_{\bar{y}})$-equivariant twisted D-modules on ${}_{\bar{y}}({\Bun_N^{\omega^\rho}})_{=\lambda\cdot x, good\ at\ \bar{y}}^w$ are exactly those twisted D-modules on $S^{t^\lambda w}_{\Fl,x}$ which are $(N^{\omega^\rho}_{X-\bar{y}-x}, \chi_{x})$-equivariant. 

Now, the strata-wise equivalence follows from 
\begin{equation}
  \Whit_q(({\Bun_N^{\omega^\rho}})_{=\lambda\cdot x}^w) \simeq \bigcap \Shv_{\cG^G}^{N^{\omega^\rho}_{X-\bar{y}-x}, \chi_{x}} (S^{t^\lambda w}_{\Fl,x})\simeq \bigcap \Shv_{\cG^G}^{N(\cK)_x^{\omega^\rho}, \chi_{x}} (S^{t^\lambda w}_{\Fl,x})\simeq \Shv_{\cG^G}^{N(\cK)_x^{\omega^\rho}, \chi_{x}} (S^{t^\lambda w}_{\Fl,x}).
\end{equation}

To prove Lemma \ref{thm local glob}, it is sufficient to show that $\pi^!_{\Fl,x}$ induces a functor between the Whittaker categories and is fully faithful. (Then, the essentially surjective property follows from $\Whit_q(\Fl_{G,x}^{\omega^\rho})$ is generated by $*$-extensions of objects in $\Whit_q(S_{\Fl,x}^{t^\lambda w}$)).

The functor $\pi_{\Fl,x}^!$ sends any global Whittaker sheaf to local Whittaker sheaf: $\Whit_q((\overline{\Bun_N^{\omega^\rho}})_{\infty\cdot x}')$ is generated by $*$-extensions of objects in $\Whit_q({\Bun_N^{\omega^\rho}})_{=\lambda\cdot x}^w)$, and $\pi^!_{\Fl,x}$ sends any such sheaf to a local Whittaker sheaf, then using the fact that the local Whittaker category is a full cocomplete subcategory of $\Shv_{\cG^G}(\Fl_{G,x}^{\omega^\rho})$, we obtain that $\pi^!_{\Fl,x}$ sends any global Whittaker sheaf to a local Whittaker sheaf.

Consider the following commutative diagram
\[\xymatrix{
&\Ran_{X,x}\times \Fl^{\omega^\rho}_{G,x}\ar[rd]^{\unit}\ar[ld]&\\
\Fl^{\omega^\rho}_{G,x}\ar[rd]^{\pi_{\Fl,x}}&&(\overline{S}^{w_0}_{\Fl,\Ran_x})_{\infty\cdot x}\ar[ld]^{\pi_{\Fl, \Ran_{X,x}}}\\
&(\overline{\Bun_N^{\omega^\rho}})_{\infty\cdot x}'.&}
\]

Since $\Ran_x$ is contractible (ref. \cite[Theorem 1.6.5]{[Ga11]} and \cite[Proposition 4.3.3]{[BD]}), we only need to show that the $!$-pullback along the right-hand side induces a fully faithful embedding:
\begin{equation}
\Whit_q((\overline{\Bun_N^{\omega^\rho}})_{\infty\cdot x}')\longrightarrow \Whit_q((\overline{S}^{w_0}_{\Fl,\Ran_x})_{\infty\cdot x})\longrightarrow \Whit_q(\Ran_{X,x}\times \Fl^{\omega^\rho}_{G,x}).
\end{equation}

According to Lemma \ref{infth}, there is $\Whit_q((\overline{S}^{w_0}_{\Fl,\Ran_x})_{\infty\cdot x})\simeq \Whit_q(\Ran_{X,x}\times \Fl^{\omega^\rho}_{G,x})$. So, we only need to show $\pi_{\Fl, \Ran_{X,x}}: \Whit_q((\overline{\Bun_N^{\omega^\rho}})_{\infty\cdot x}')\to \Whit_q((\overline{S}^{w_0}_{\Fl,\Ran_x})_{\infty\cdot x})$ is fully faithful. Note that $\Whit_q((\overline{\Bun_N^{\omega^\rho}})_{\infty\cdot x}')$ and $ \Whit_q((\overline{S}^{w_0}_{\Fl,\Ran_x})_{\infty\cdot x})$ are full subcategories of $\Shv_{\cG^G}((\overline{\Bun_N^{\omega^\rho}})_{\infty\cdot x}')$ and $\Shv_{\cG^G}((\overline{S}^{w_0}_{\Fl,\Ran_x})_{\infty\cdot x})$, respectively. We only need to prove $\pi_{\Fl,\Ran_{X,x}}:\Shv_{\cG^G}((\overline{\Bun_N^{\omega^\rho}})_{\infty\cdot x}') \longrightarrow \Shv_{\cG^G}((\overline{S}^{w_0}_{\Fl,\Ran_x})_{\infty\cdot x})$ is fully faithful.

Denote by $\Bun_G^{N, \text{gen}}$ the stack which classifies principal $G$-bundles on $X$ with a generic $N^{\omega^\rho}$-reduction. Consider the following Cartesian diagram 
\begin{equation}\label{uhc dia}
    \xymatrix{
(\overline{S}^{w_0}_{\Fl,\Ran_x})_{\infty\cdot x}\ar[r]\ar[d]& \Gr_{G, \Ran}^{\omega^\rho}\ar[d]\\
(\overline{\Bun_N^{\omega^\rho}})_{\infty\cdot x}'\ar[r]& \Bun_G^{N, \text{gen}}.
}
\end{equation}
Now the desired fully faithfulness follows from the fact that $\Gr_{G, \Ran}^{\omega^\rho}\longrightarrow \Bun_G^{N, \text{gen}}$ is  universally homologically
contractible, ref \cite[Theorem A.1.10]{[Ga4]}. 

}

\subsection{Global semi-infinite !-extension sheaf}\label{global semi-infinite ! extension D-module}
Before we define global functors corresponding to the functors $F^L$ and $F^{DK}$, we should construct the global analog of the semi-infinite sheaves $j_!(\omega_{S^{-,\Conf}_{\Gr, \Conf}})$,  $j_!(\omega_{S^{-,w,\Conf_{\infty\cdot x}}_{\Fl, \Conf_{\infty\cdot x}}})$ defined in Section \ref{configuration gr and fl}.

\begin{definition}\label{define B-}
We denote by $\Bun_{B^-}'$ the algebraic stack classifying $B^-$-bundles on $X$ plus a $B$-reduction of the induced $G$-bundle at the point $x$. In other words, it is the fiber product of $\Bun_{B^-}$ with the classifying stack of $B$ over the classifying stack of $G$.

 We define $\overline{\Bun}{}_{B^-}'$ as the Drinfeld compactification of $\Bun_{B^-}'$. It classifies the quadruples $\{\mathcal{P}_T, \mathcal{P}_G, \{\kappa^{-, \check{\lambda}}, \check{\lambda}\in \Lambda^+\},\epsilon \}$, where $\mathcal{P}_T\in \Bun_T$ is a $T$-bundle on $X$, $\mathcal{P}_G$ is a $G$-bundle on $X$ and 
\begin{equation}
    \kappa^{-,\check{\lambda}}:\  '\mathcal{V}_{\mathcal{P}_G}^{\check{\lambda}}\to (\omega^{\frac{1}{2}})^{\langle\check{\lambda}, 2\rho\rangle}\qquad \forall\check{\lambda}\in \Lambda^+
\end{equation}
is a collection of morphisms which are regular  on $X$ and satisfy the Pl\"{u}cker relations.

By omitting the Iwahori structure $\epsilon$, we get the Drinfeld compactification $\overline{\Bun}{}_{B^-}$ of $\Bun_{B^-}$.
\end{definition}

\begin{rem}
If we require $\kappa^{-.\check{\lambda}}$ to be surjective in the definition of $\overline{\Bun}{}_{B^-}'$ (resp. $\overline{\Bun}{}_{B^-}$), the resulting stack is ${\Bun}_{B^-}'$ (resp. $\Bun_{B^-}$).
\end{rem} 

\begin{definition}
We  define the gerbe $\mathcal{G}_{\glob}^{G, T, \ratio}$ on $\overline{\Bun}{}_{B^-}$ (resp. $\overline{\Bun}{}_{B^-}'$) as $(\mathcal{G}^G)^{-1}\otimes (\mathcal{G}^T)$.
\end{definition}

$\Bun_{B^-}'$ has a relative position map
\begin{equation}
    \Bun_{B^-}'\to \Bun_G\mathop{\times}\limits_{\pt/G}\pt/B\mathop{\times}\limits_{\pt/G}\pt/B^-\simeq \Bun_G\mathop{\times}\limits_{\pt/G} B^-\backslash G/B.
\end{equation} We will denote the preimage of the Bruhat cell $\Br^w\subset B^-\backslash G/B$ in $\Bun_{B^-}'$ by $\Bun_{B^-}^w$, $w\in W$. For convenience, we denote by $\Bun_{B^-}^{''}$ the stack $\Bun_{B^-}^{1}$.

By the definitions of $\mathcal{G}^G$ and $\mathcal{G}^T$ in Section \ref{gerbe used}, we see that the gerbe $\mathcal{G}_{\glob}^{G, T, \ratio}$ is canonically trivial on $\Bun_{B^-}''\subset \overline{\Bun}{}_{B^-}'$ and $\Bun_{B^-}\subset \overline{\Bun}{}_{B^-}$. Hence, the categories of $\mathcal{G}_{\glob}^{G, T, \ratio}$-twisted sheaves on $\Bun_{B^-}''$ and $\Bun_{B^-}$ are equivalent to the categories of non-twisted sheaves on the corresponding stacks. In particular, we can consider the constant sheaf in the twisted case.

\begin{definition}
We denote by $j^-_{!, \glob, \Fl}$ (resp. $j^-_{!, \glob, \Gr}$) the $!$-extension of the $\mathcal{G}_{\glob}^{G, T, \ratio}$-twisted {perverse} constant sheaf on $\Bun_{B^-}''$ (resp. $\Bun_{B^-}$) to $\overline{\Bun}{}_{B^-}'$ (resp. $\overline{\Bun}{}_{B^-}$). 
\end{definition}

\subsection{Zastava spaces}\label{Zastava space}
 Zastava spaces are introduced in \cite{[FM]}. They play an important role in our global construction of the functor. Let us recall the definitions of the Zastava space and related stacks in this section.

\begin{definition}
We define the compactified Zastava space $\bar{Z}_{\Gr}$ and Zastava space ${Z}_{\Gr}$ as
\begin{equation*}
    \begin{split}
        \bar{Z}_{\Gr}:=\overline{\Bun_N^{\omega^\rho}}\mathop{\times'}\limits_{\Bun_G} \overline{\Bun}{}_{B^-},\\
         {Z}_{\Gr}:=\overline{\Bun_N^{\omega^\rho}}\mathop{\times'}\limits_{\Bun_G} {\Bun}_{B^-},
    \end{split}
\end{equation*}
where $\times'$ means the open substack of the fiber product such that the composition of $\kappa^{\check{\lambda}}$ and $\kappa^{-, \check{\lambda}}$ is non-zero for any dominant weight $\check{\lambda}$.
\end{definition}

\begin{definition}Similarly, we define the affine flags version of Zastava spaces as
\begin{equation}
    \begin{split}
        (\bar{Z}_{\Fl,x})_{\infty \cdot x}:=\overline{(\Bun_N^{\omega^\rho})}'_{\infty\cdot x}\mathop{\times'}\limits_{\Bun_G'} \overline{\Bun}{}_{B^-}',\\
        ({Z}_{\Fl,x})_{\infty \cdot x}:=\overline{(\Bun_N^{\omega^\rho})}'_{\infty\cdot x}\mathop{\times'}\limits_{\Bun_G'} {\Bun_{B^-}^{''}}.
    \end{split}
\end{equation}

\end{definition}


Since we assume that $[G,G]$ is simply-connected, taking zeros of the composition of $\kappa^{\check{\lambda}}$ and $\kappa^{-,\check{\lambda}}$ gives ind-proper maps

\begin{equation}\label{eva}
    v_{\Fl, \glob}:\ (\bar{Z}_{\Fl,x})_{\infty \cdot x}\to \Conf_{x},
    \end{equation}
and
\begin{equation}
    v_{\Gr, \glob}:\ \bar{Z}_{\Gr}\to \Conf.
\end{equation}



By \cite[Section 2.3]{[BFGM]}, the Zastava spaces satisfy factorization property.
\begin{lem}\label{factzas}
There exists isomorphisms
\begin{equation}
    \bar{Z}_{\Gr}\mathop{\times}\limits_{\Conf} (\Conf\times \Conf)_{disj}\simeq (\bar{Z}_{\Gr}\times \bar{Z}_{\Gr})\mathop{\times}\limits_{\Conf\times \Conf} (\Conf\times \Conf)_{disj},
\end{equation}
and
\begin{equation}
    \begin{split}
        (\bar{Z}_{\Fl,x})_{\infty \cdot x}&\mathop{\times}\limits_{\Conf_{\infty\cdot x}} (\Conf\times \Conf_{\infty\cdot x})_{disj}\\ &\simeq\\ (\bar{Z}_{\Gr}\times (\bar{Z}_{\Fl,x})_{\infty \cdot x})&\mathop{\times}\limits_{\Conf\times \Conf_{\infty\cdot x}} (\Conf\times \Conf_{\infty\cdot x})_{disj}.
    \end{split}
\end{equation}
\end{lem}

\subsection{Construction of global functors}\label{subsection F glob}
Let us consider the following diagram:

\begin{center}
\begin{equation}\label{diagram 16.1}
    \xymatrix{
&(\bar{Z}_{\Fl,x})_{\infty \cdot x}\ar[ld]^{\bar{\mathfrak{q}}_Z'}\ar[dd]^{v_{\Fl, \glob}}\ar[rd]^{\bar{\mathfrak{p}}_Z'}&\\
(\overline{\Bun_N^{\omega^\rho}})'_{\infty\cdot x}&&\overline{\Bun}{}_{B^-}'\\
&\Conf_{\infty\cdot x}&}
\end{equation}

\end{center}

\begin{definition}\label{functorglob}
We define \textit{global} functors $F_{\glob}^L$ and $F_{\glob}^{DK}$ as
\begin{equation}
\begin{split}
F^{L}_{\glob}: \Whit_{q}((\overline{\Bun_N^{\omega^\rho}})'_{\infty\cdot x})&\to \Shv_{\mathcal{G}^\Lambda}(\Conf_{\infty\cdot x})\\
        \mathcal{F}&\mapsto v_{\Fl, \glob, !}(\bar{\mathfrak{q}}_Z^{',!}(\mathcal{F})\mathop{\otimes}\limits^{!} \bar{\mathfrak{p}}_Z^{',!} (j^-_{!, \glob, \Fl} [\dim \Bun_G'])),
\end{split}
\end{equation}
and
\begin{equation}
\begin{split}
F^{DK}_{\glob}: \Whit_{q}((\overline{\Bun_N^{\omega^\rho}})'_{\infty\cdot x})&\to \Shv_{\mathcal{G}^\Lambda}(\Conf_{\infty\cdot x})\\
\mathcal{F}&\mapsto v_{\Fl, \glob, !}(\bar{\mathfrak{q}}_Z^{',!}(\mathcal{F})\mathop{\otimes}\limits^{!}\bar{\mathfrak{p}}_Z^{',!} (j^-_{*, \glob, \Fl}[\dim \Bun_G'])).
\end{split}
\end{equation}

\end{definition}

Similarly, we consider the diagram without the Iwahori structures:
\begin{center}
\begin{equation}\label{diagram 16.4}
  \xymatrix{
&\bar{Z}_{\Gr}\ar[ld]^{\bar{\mathfrak{q}}_Z}\ar[dd]^{v_{\Gr, \glob}}\ar[rd]^{\bar{\mathfrak{p}}_Z}&\\
\overline{\Bun_N^{\omega^\rho}}&&\overline{\Bun}{}_{B^-}\\
&\Conf&}
\end{equation}
\end{center}
\begin{definition}\label{functorglob gr}
We define functors $F^{L}_{\glob, \Gr}$ and $F^{L}_{\glob, \Gr}(\mathcal{F})$ as
\begin{equation}\label{16.5}
\begin{split}
F^{L}_{\glob, \Gr}: \Whit_{q}(\overline{\Bun_N^{\omega^\rho}})&\to \Shv_{\mathcal{G}^\Lambda}(\Conf)\\
      \mathcal{F}&\mapsto v_{\Gr,\glob, !}(\bar{\mathfrak{q}}_Z^{!}(\mathcal{F})\mathop{\otimes}\limits^{!} \bar{\mathfrak{p}}_Z^{!} (j^-_{!, \glob, \Gr}[\dim \Bun_G])),
\end{split}
\end{equation}
and
\begin{equation}\label{16.6}
\begin{split}
F^{DK}_{\glob, \Gr}: \Whit_{q}(\overline{\Bun_N^{\omega^\rho}})&\to \Shv_{\mathcal{G}^\Lambda}(\Conf)\\
     \mathcal{F}&\mapsto v_{\Gr,\glob, !}(\bar{\mathfrak{q}}_Z^{!}(\mathcal{F})\mathop{\otimes}\limits^{!}\bar{\mathfrak{p}}_Z^{!} (j^-_{*, \glob, \Gr}[\dim \Bun_G])).
\end{split}
\end{equation}
\end{definition}

{
\subsection{Comparison between local-global functors}

\subsubsection{}
Recall that in Section \ref{section 6.3}, we define a ${\cY}$-parametrized prestack ${}_{\cY}\Fl$, for any prestack $\cY$ with a map to $T(\cO)^{\omega^\rho}_{\Ran_x}\backslash \Ran_x$. Now, we consider two cases: $\cY=(\Gr_{T,\Ran_x}^{\omega^\rho})_{\infty\cdot x}^{\text{neg}}$ and $\cY=\Bun_T\times \Ran_x$. We denote the resulting prestacks by ${}_{\Gr_T}\Fl$ and ${}_{\Bun_T}\Fl$, respectively.

Note that the restriction of $\cG^G$ on $\Bun_{B^-}':=\Bun_{B^-}\underset{\pt/G}{\times}\pt/B$ is canonically identified with $\cG^T$, which implies that $\cG^{G,T,\ratio}$ on $\Bun_{B^-}'$ is canonically trivial. Let $j_{!,\glob,\Fl}^-\in \Shv_{((\cG)^{G,T,\ratio})^{-1}}(\overline{\Bun}_{B^-}')$ be the $!$-extension of the (twisted) constant perverse on $\Bun_{B^-}'$.

The algebraic stack $\overline{\Bun}_{B^-}'$ is an algebraic stack over $\Bun_T$ and can be regarded as a $T$-twisted construction of $(\overline{\Bun}_{N^-}^{\omega^\rho})^{1}$ (the negative analog of $(\overline{\Bun}_{N}^{\omega^\rho})^{w_0}$ in Section \ref{app a}). In other words, it is a stack over $\Bun_T$ and the fiber over $\cP_T$ is the $\cP_T$-twisted $(\overline{\Bun}_{N^-}^{\omega^\rho})^{1}$. We have a natural projection \[{}_{\Bun_T}\pi_{\Fl,\Ran_x}: {}_{\Bun_T}\overline{S}_{\Fl, \Ran_x}^{-,1}\longrightarrow \overline{\Bun}^{'}_{B^-}.\] 

A relative version of the proof of Lemma \ref{equation A.2.1} yields the following lemma
\begin{lem}
    \begin{equation}
        {}_{\Bun_T}\pi_{\Fl,\Ran_x}^!(j_{!,\glob,\Fl}^-)[\dim {\Bun}_{B^-}']= {}_{\Bun_T}j_!(\omega_{S_{\Fl,\Ran_x}^{-,1}}).
    \end{equation}
\end{lem}

Furthermore, by definition, the $!$-pullback of ${}_{\Bun_T}j_!(\omega_{S_{\Fl,\Ran_x}^{-,1}})$ along ${}_{\Gr_T}\Fl\longrightarrow {}_{\Bun_T}\Fl$ goes to $ {}_{\Gr_T}j_!(\omega_{S_{\Fl,\Ran_x}^{-,1}})$. This implies the following lemma 
\begin{lem}\label{Lema a 3.5}
\begin{equation}
     \pi^!_{S_{\Conf}\to \Bun_B} (j_{!,\glob,\Fl}^-)[\dim {\Bun}_{B^-}']\simeq j_!(\omega_{{S}^{-,1, \Conf_{\infty\cdot x}}_{\Fl, \Conf_{\infty\cdot x}}}),
\end{equation}
where $\pi_{S_{\Conf}\to \Bun_B}$ is the natural projection from $\overline{S}^{-,1, \Conf_{\infty\cdot x}}_{\Fl, \Conf_{\infty\cdot x}}$ to $\overline{\Bun}'_{B^-}$.
\end{lem}

}

 {With the preparations above, we prove the following proposition, which is the analog of \cite[Proposition 20.3.4]{[GL1]} in the affine flags case.}

\begin{lem}\label{thm2}
$$F^L_{\glob}\simeq F^L\circ \pi_{x, \Fl}^![d_g],$$ and
$$F^{DK}_{\glob}\simeq F^{DK}\circ \pi_{x, \Fl}^![d_g].$$
\end{lem}
\begin{proof}
The Zastava space $(\bar{Z}_{\Fl,x})_{\infty\cdot x}$ is isomorphic to $(\overline{S}_{\Fl, \Conf_{\infty\cdot x}}^{w_0})_{\infty\cdot x}\cap (\overline{S}_{\Fl, \Conf_{\infty\cdot x}}^{-,\Conf_{\infty\cdot x}})$ (\cite[Proposition 20.2.2]{[GL1]}). Under this identification, $v_{\Conf_{\infty\cdot x}}$ is identified with $v_{\Fl, \glob}$.

Now the lemma follows from the following two facts:
\begin{itemize}
    \item The $!$-pullback of $j_{!,\glob,\Fl}^-[\dim \Bun_G']$ along $(\bar{Z}_{\Fl,x})_{\infty\cdot x}\to \overline{\Bun}{}_{B^-}'$ is isomorphic to the sheaf $j_!(\omega_{S_{\Fl, \Conf_{\infty\cdot x}}^{-,\Conf_{\infty\cdot x}}})[\deg+d_g]|_{(\bar{Z}_{\Fl,x})_{\infty\cdot x}}$.
    \item The sheaf $\sprd_{\Fl}\circ \pi_{\Fl,x}^!(\mathcal{F})$ is isomorphic to the $!$-pullback of $\mathcal{F}$ along $(\overline{S}_{\Fl, \Conf_{\infty\cdot x}}^{w_0})_{\infty\cdot x}\to (\overline{\Bun}_N^{\omega^\rho})_{\infty\cdot x}$.
\end{itemize}

The first fact follows from {Lemma \ref{Lema a 3.5}, } and the second one follows from Lemma \ref{thm local glob}.
\end{proof}
Similarly, by a local-global comparison, one can prove that
 \[\Omega^{L,'}_{q}\simeq v_{\Gr,\glob, *}(\bar{\mathfrak{q}}_Z^{!}(\mathcal{F}_\emptyset)\mathop{\otimes}\limits^{!} \bar{\mathfrak{p}}_Z^{!} (j^-_{!, \glob, \Gr}[\dim \Bun_G])),\]
and 
\[\Omega^{DK'}_q\simeq v_{\Gr,\glob, *}(\bar{\mathfrak{q}}_Z^{!}(\mathcal{F}_\emptyset)\mathop{\otimes}\limits^{!} \bar{\mathfrak{p}}_Z^{!} (j^-_{*, \glob, \Gr}[\dim \Bun_G])).\]
Here $\mathcal{F}_{\emptyset}$ is the unique irreducible object in $\Whit_{q}(\overline{\Bun_N^{\omega^\rho}})$. 

\begin{rem}
\textit{A priori}, it is not easy at all to show that $F^L_{\glob}$ factors through a category of factorization modules and the image of $\mathcal{F}_\emptyset$ under $F^L_{\glob,\Gr}$ admits a factorization algebra structure. It is the reason why we have to start from the local Whittaker category and construct Ran-ified affine flags in preceding sections.
\end{rem}

By Lemma \ref{thm2}, to prove Proposition \ref{imstan}, we only need to prove the following proposition.

\begin{prop}\label{imstan glob}
For any $\lambda\in \Lambda$, there exists an isomorphism
\begin{equation}
F^L_{\glob}(\Delta_{\glob}^\lambda)\simeq \Delta_{\lambda, \Omega_q^{L,'}}.
\end{equation}
\end{prop}

\section{Proof of Proposition \ref{imstan glob}: Duality}\label{step 5}
In this section, we will study the relationship between $F_{\glob}^{L}$ and $F_{\glob}^{DK}$. We want to prove that $F^L_{\glob}$ and $F^{DK}_{\glob}$ are Verdier dual to each other
$$F^L_{\glob}\circ\mathbb{D}^{\verdier}(\mathcal{F})\simeq \mathbb{D}^{\verdier}\circ F^{DK}_{\glob}(\mathcal{F}): \Whit_{q}((\overline{\Bun_N^{\omega^\rho}})_{\infty \cdot x}')^{loc.c}\to \Omega_q^{L}-\Fact. $$
The method is given by reducing the above isomorphism to some stack where we can apply the \textit{universally locally acyclic} property (ULA).

\subsection{Universally locally acyclic}


In \cite[Section 5]{[BG]}, the authors introduced the notion of ULA\footnotemark.\footnotetext{It is expected to be equivalent to the notion of universally locally acyclic introduced in \cite{[D]}.} Roughly speaking, a sheaf is {ULA} with respect to a morphism if its singular support over each fiber is the same. 
The following lemma (it is straightforward from the definition in \cite[Section 5.1.2]{[BG]}) explains why the notion of ULA is important. 

\begin{lem}\label{17.2 lemma}
For algebraic stacks $\mathcal{X}, \mathcal{Y}$ and $\mathcal{W}$, consider the following Cartesian diagram of algebraic stacks:
$$\xymatrix{
\mathcal{X}\mathop{\times}\limits_{\mathcal{W}} \mathcal{Y}\ar[r]^{p_1}\ar[d]^{p_2}&\mathcal{X}\ar[d]^{q_2}\\
\mathcal{Y}\ar[r]^{q_1}&\mathcal{W}
}$$
Given $\mathcal{F}_1\in \Shv(\mathcal{X}),\mathcal{F}_2\in \Shv(\mathcal{Y})$. If we assume that $\mathcal{F}_1$ is ULA with respect to $q_2$ and $\mathcal{W}$ is smooth and of dimension $d$, then the following canonical map (see \cite[Section 5.1]{[BG]})

$$p_1^*(\mathcal{F}_1)\mathop{\otimes}\limits^* p_2^*(\mathcal{F}_2)[-d]\overset{\sim}{\longrightarrow} p_1^!(\mathcal{F}_1)\mathop{\otimes}\limits^! p_2^!(\mathcal{F}_2)[d]$$ is an isomorphism.
\end{lem}

\subsection{{Duality of $F_{\glob}^L$ and $F_{\glob}^{DK}$}{}}\label{Duality of }
We denote by  $\overline{\Bun}{}_{B^-}^{', \lambda}$ the substack of  $\overline{\Bun}{}_{B^-}'$ such that the degree of the $T$-bundle is $-\lambda+(2g-2)\rho$, and denote by $\overline{\Bun}_{B^-, \leq \mu}^{',\lambda}$ the open substack of $\overline{\Bun}{}_{B^-}^{', \lambda}$ such that the total order of
degeneracy of the generalized $B^-$-reductions is no more than $\mu$.

The following lemma is a tiny modification of \cite[Propsoition 4.1.1]{[Camp]} (see \cite{[Drinfeld Iwahori]} for a detailed proof).
\begin{lem}\label{ULA thm +}
There exists an integer $d$ which depends only on the genus of $X$, such that, for any $\mu\in \Lambda^{\textnormal{pos}}$, and $\lambda\in \Lambda$ satisfying the condition (X): for any $0\leq\mu'\leq \mu$, $$ \langle-\lambda-\mu', \check{\alpha}_i\rangle> d,$$
the restriction of $j^{-}_{!,\glob, \Fl}$ to $\overline{\Bun}_{B^-, \leq \mu}^{',\lambda}$ is ULA with respect to the natural projection
\begin{equation}\label{b- projection}
    \overline{\Bun}_{B^-, \leq \mu}^{',\lambda}\to \Bun_G'.
\end{equation}

\end{lem}

In this section, we will prove the following theorem using Lemma \ref{ULA thm +}.
\begin{thm}\label{dualfunctor}
$F_{\glob}^L$ and $F_{\glob}^{DK}$ are Verdier dual to each other, i.e.,
\begin{equation}\label{17.3}
    F_{\glob}^L\circ\mathbb{D}^{\verdier}\simeq \mathbb{D}^{\verdier}\circ F_{\glob}^{DK}.
\end{equation}
\end{thm}
{The proof of  Theorem \ref{dualfunctor} follows a standard factorization argument that has been used in various references such as \cite[Section 16.4]{[GN]}, \cite[Proof of Proposition 3.6.5]{[Ga3]}, \cite[Section 21.2]{[GL1]}, etc.}

\subsubsection{Step I}\label{step 1}


We want to prove that the natural transformation 
\begin{equation*}
    \begin{split}
        \mathbb{D}^{\verdier}(v_{\Fl,\glob,!}(\bar{\mathfrak{q}}_Z^{!}(\mathcal{F})\mathop{\otimes}\limits^{!}&\bar{\mathfrak{p}}_Z^{!} (j^{-}_{*, \glob, \Fl}[\dim \Bun_G'])))\\ \downarrow&\\
         v_{\Fl,\glob, !}(\bar{\mathfrak{q}}_Z^{!}(\mathbb{D}^{\verdier}(\mathcal{F}))\mathop{\otimes}\limits^{!}&\bar{\mathfrak{p}}_Z^{!} (j^{-}_{!, \glob, \Fl}[\dim \Bun_G'])).
    \end{split}
\end{equation*}
which is obtained in \cite[Section 5.1]{[BG]} is an isomorphism for any locally compact object $\mathcal{F}\in \Whit_q((\overline{\Bun_N^{\omega^\rho}})_{\infty \cdot x}')^{loc.c}$.

By definition, $v_{\Fl, \glob}$ is {ind-proper}, hence, it suffices to prove \[\mathbb{D}^{\verdier}(\bar{\mathfrak{q}}_Z^{!}(\mathcal{F})\mathop{\otimes}\limits^{!}\bar{\mathfrak{p}}_Z^{!} (j^{-}_{*, \glob, \Fl}[\dim \Bun_G'])))\simeq \bar{\mathfrak{q}}_Z^{!}(\mathbb{D}^{\verdier}(\mathcal{F}))\mathop{\otimes}\limits^{!}\bar{\mathfrak{p}}_Z^{!} (j^{-}_{!, \glob,  \Fl}[\dim \Bun_G']).\]

We only need to prove
\begin{equation*}
\begin{split}
        \bar{\mathfrak{q}}_Z^{*}(\mathcal{F})\mathop{\otimes}\limits^{*}\bar{\mathfrak{p}}_Z^{*} (\mathbb{D}^{\verdier}(j^{-}_{*, \glob, \Fl})))[-\dim \Bun_G']\simeq \bar{\mathfrak{q}}_Z^{!}(\mathcal{F})\mathop{\otimes}\limits^{!}\bar{\mathfrak{p}}_Z^{!} (j^{-}_{!, \glob,  \Fl})[\dim \Bun_G'],
\end{split}
\end{equation*}
for any $\mathcal{F}\in \Whit_{q^{-1}}((\overline{\Bun_N^{\omega^\rho}})_{\infty \cdot x}')^{loc.c}$. Since $j^{-}_{*, \glob,  \Fl}$ and $j^{-}_{!, \glob, \Fl}$ are dual to each other, we should prove
\begin{equation}\label{17.6}
\begin{split}
    \bar{\mathfrak{q}}_Z^{*}(\mathcal{F})\mathop{\otimes}\limits^{*}\bar{\mathfrak{p}}_Z^{*} (j^{-}_{!, \glob, \Fl}))[-\dim \Bun_G']\simeq \bar{\mathfrak{q}}_Z^{!}(\mathcal{F})\mathop{\otimes}\limits^{!}\bar{\mathfrak{p}}_Z^{!} (j^{-}_{!, \glob, \Fl})[\dim \Bun_G'].
\end{split}
\end{equation}

It indicates us to use Lemma \ref{17.2 lemma}. That is to say, if we can prove that $j^{-}_{!, \glob,  \Fl}$ is ULA with respect to the projection morphism (\ref{b- projection}), then (\ref{17.6}) follows from Lemma \ref{17.2 lemma}. But in fact, we do not need such a strong property. we can recover the isomorphism (\ref{17.6}) by factorization property from its restriction to an open subset.




\begin{definition}\label{Z lambda mu}
 Denote by $(\bar{Z}_{\Fl,x})_{\infty \cdot x}^{\lambda, \leq \mu}$ the preimage of $\overline{\Bun}_{B^-, \leq \mu}^{',\lambda}$ in $(\bar{Z}_{\Fl,x})_{\infty \cdot x}$ under the projection morphism $(\bar{Z}_{\Fl,x})_{\infty \cdot x}\to \overline{\Bun}{}_{B^-}'$. Note that $(\bar{Z}_{\Fl,x})_{\infty \cdot x}^{\lambda, \leq \mu}$ is open in $(\bar{Z}_{\Fl,x})_{\infty \cdot x}$ for any $\lambda\in \Lambda, \mu\in \Lambda^{\textnormal{pos}}$.
\end{definition}


Combine Lemma \ref{17.2 lemma} with Lemma \ref{ULA thm +}, we obtain the following corollary.
\begin{cor}\label{dual}
If $\lambda$, $\mu$ satisfy the condition (X), then the natural transformation (\ref{17.6}) is an isomorphism on $(\bar{Z}_{\Fl,x})_{\infty \cdot x}^{\lambda, \leq \mu}$,
for any twisted sheaf $\mathcal{F}\in \Shv_{\mathcal{G}^G}((\overline{\Bun_N^{\omega^\rho}})'_{\infty\cdot x})$.
\end{cor}
We set
\begin{equation}
    (\bar{Z}_{\Fl,x})_{\infty \cdot x}^s:= \underset{\lambda\in \Lambda, \mu\in \Lambda^{\textnormal{pos}},\ \textnormal{condition} (X)}{\bigcup} (\bar{Z}_{\Fl,x})_{\infty \cdot x}^{\lambda, \leq \mu},
\end{equation}

then by the corollary above, (\ref{17.6}) is an isomorphism on $(\bar{Z}_{\Fl,x})_{\infty \cdot x}^s$.
\subsubsection{Step II}

Now we want to extend this isomorphism to the whole affine flags Zastava space $(\bar{Z}_{\Fl,x})_{\infty \cdot x}$ via the factorization properties. 

We denote by $\bar{Z}_{\Gr}^{\lambda}$  the fiber product $$\bar{Z}_{\Gr}\mathop{\times}\limits_{\Conf}\Conf^{\lambda}.$$ Similarly, we denote by $(\bar{Z}_{\Fl,x})_{\infty \cdot x}^{\lambda}$ the fiber product $(\bar{Z}_{\Fl,x})_{\infty \cdot x}\mathop{\times}\limits_{\Conf_{\infty\cdot x}}\Conf_{\infty\cdot x}^{\lambda}$.

By Proposition \ref{factzas}, the affine flags Zastava space $(\bar{Z}_{\Fl,x})_{\infty \cdot x}$ is a factorization module space with respect to $\bar{Z}_{\Gr}$. Note that the factorization structure is compatible with degree, i.e.,
\begin{equation}\label{8.7}
    \begin{split}
        \bar{Z}_{\Gr}^{\lambda_1}\mathop{\times}(\bar{Z}_{\Fl,x})_{\infty \cdot x}^{\lambda_2}&\mathop{\times}\limits_{\Conf^{\lambda_1}\times \Conf_{\infty\cdot x}^{\lambda_2}} (\Conf^{\lambda_1}\times \Conf_{\infty\cdot x}^{\lambda_2})_{disj}\\&\simeq\\ (\bar{Z}_{\Fl,x})_{\infty \cdot x}^{\lambda_1+\lambda_2}&\mathop{\times}\limits_{\Conf^{\lambda_1+\lambda_2}_{\infty\cdot x}} (\Conf^{\lambda_1}\times \Conf_{\infty\cdot x}^{\lambda_2})_{disj}.
    \end{split}
\end{equation}
Denote by ${Z}_{\Gr}^{\circ}:= \Bun_N^{\omega^\rho}\times' \Bun_{B^-}$. Taking the restriction of \eqref{8.7}, we get the following map
\begin{equation}\label{17.8}
\begin{split}
     {Z}_{\Gr}^{\circ,\lambda_1}\mathop{\times}(\bar{Z}_{\Fl,x})_{\infty \cdot x}^{\lambda_2}&\mathop{\times}\limits_{\Conf^{\lambda_1}\times \Conf_{\infty\cdot x}^{\lambda_2}} (\Conf^{\lambda_1}\times \Conf_{\infty\cdot x}^{\lambda_2})_{disj}\\ &\downarrow\\ (\bar{Z}_{\Fl,x})_{\infty \cdot x}^{\lambda_1+\lambda_2}&\mathop{\times}\limits_{\Conf_{\infty\cdot x}^{\lambda_1+\lambda_2}} (\Conf^{\lambda_1}\times \Conf_{\infty\cdot x}^{\lambda_2})_{disj}.
\end{split}
\end{equation}


We note that for any point in $Z_{\Gr}^{\circ,\lambda_1}$, the $B^-$-structure is genuine (non-degenerate). As a result, given an arbitrary point $z_2\in (\bar{Z}_{\Fl,x})_{\infty \cdot x}^{\lambda_2}$ and arbitrary point $z_1$ in $Z_{\Gr}^{\circ, \lambda_1}$, the corresponding object on the right-hand side of (\ref{17.8}) has the same order of degeneracy of generalized $B^-$-bundle as $z_2$.
\subsubsection{}\label{8.2.7}
$Z_{\Gr}^{\circ,\lambda_1}\mathop{\times}(\bar{Z}_{\Fl,x})_{\infty \cdot x}^{\lambda_2}\mathop{\times}\limits_{\Conf^{\lambda_1}\times \Conf_{\infty\cdot x}^{\lambda_2}} (\Conf^{\lambda_1}\times \Conf_{\infty\cdot x}^{\lambda_2})_{disj}$ admits two smooth morphisms to $(\bar{Z}_{\Fl,x})_{\infty \cdot x}$:
\begin{itemize}
    \item one is given by the projection to $(\bar{Z}_{\Fl,x})^{\lambda_2}_{\infty \cdot x}$
\begin{equation}
   r_1^{\lambda_1}: Z_{\Gr}^{\circ,\lambda_1}\mathop{\times}(\bar{Z}_{\Fl,x})_{\infty \cdot x}^{\lambda_2}\mathop{\times}\limits_{\Conf^{\lambda_1}\times \Conf_{\infty\cdot x}^{\lambda_2}} (\Conf^{\lambda_1}\times \Conf_{\infty\cdot x}^{\lambda_2})_{disj}\to (\bar{Z}_{\Fl,x})_{\infty \cdot x}^{\lambda_2},
\end{equation}
    \item 
another one is given by the factorization map (\ref{17.8}) composed with the projection to $(\bar{Z}_{\Fl,x})_{\infty \cdot x}$
\begin{equation}
    \begin{split}
        r_2^{\lambda_1}: Z_{\Gr}^{\circ,\lambda_1}\mathop{\times}(\bar{Z}_{\Fl,x})_{\infty \cdot x}^{\lambda_2}\mathop{\times}\limits_{\Conf^{\lambda_1}\times \Conf_{\infty\cdot x}^{\lambda_2}} (\Conf^{\lambda_1}\times \Conf_{\infty\cdot x}^{\lambda_2})_{disj}\to\\ \overset{(\ref{17.8})}{\longrightarrow}(\bar{Z}_{\Fl,x})_{\infty \cdot x}^{\lambda_1+\lambda_2}\mathop{\times}\limits_{\Conf^{\lambda_1+\lambda_2}_{\infty\cdot x}} (\Conf^{\lambda_1}\times \Conf_{\infty\cdot x}^{\lambda_2})_{disj}\to \\
        \to (\bar{Z}_{\Fl,x})_{\infty \cdot x}^{\lambda_1+\lambda_2}.
    \end{split}
\end{equation}

\end{itemize}


The key observation for the proof of Theorem \ref{dualfunctor} is
\begin{itemize}
    \item for any $\mu\in \Lambda^{\textnormal{pos}}$, $\lambda_1\in \Lambda^{\textnormal{neg}}$, and $\lambda_2\in \Lambda$,  we can take an open subset  $(Z_{\Gr}^{\circ, \lambda_1}\times (\bar{Z}_{\Fl,x})_{\infty \cdot x}^{\lambda_2})_{\mu}$ of $Z_{\Gr}^{\circ, \lambda_1}\times (\bar{Z}_{\Fl,x})_{\infty \cdot x}^{\lambda_2}$ whose image under $r_2^{\lambda_1}$ lies in $(\bar{Z}_{\Fl,x})_{\infty \cdot x}^s$, and if we let $\lambda_1$ and $\mu$ vary, the collection of stacks $\{(Z_{\Gr}^{\circ, \lambda_1}\times (\bar{Z}_{\Fl,x})_{\infty \cdot x}^{\lambda_2})_{\mu}, \mu\in \Lambda^{\textnormal{pos}}, \lambda_1\in \Lambda^{\textnormal{neg}} \}$ gives a smooth cover of $(\bar{Z}_{\Fl,x})_{\infty \cdot x}^{\lambda_2}$ by the map $r_1^{\lambda_1}$.
\end{itemize}

Now let us explain the construction of $(Z_{\Gr}^{\circ, \lambda_1}\times (\bar{Z}_{\Fl,x})_{\infty \cdot x}^{\lambda_2})_{\mu}$.
\begin{definition}
 
Given $\mu\in \Lambda^{\textnormal{pos}}$, $\lambda_1\in \Lambda^{\textnormal{neg}}$, and $\lambda_2\in \Lambda$, a point of $(Z_{\Gr}^{\circ,\lambda_1}\times (\bar{Z}_{\Fl,x})_{\infty \cdot x}^{\lambda_2})_{disj}$ belongs to $(Z_{\Gr}^{\circ, \lambda_1}\times (\bar{Z}_{\Fl,x})_{\infty \cdot x}^{\lambda_2})_{\mu}$ if and only if
\begin{enumerate}[label=(\arabic*)]
    \item the order of degeneracy of the generalized $B^-$-structure  is no more than $\mu$,
    \item  $\lambda:=\lambda_1+\lambda_2$ and $\mu$ satisfy the condition (X).
\end{enumerate}

\end{definition}
If we allow $\lambda_1$ and $\mu$ vary, the collection of $(Z_{\Gr}^{\circ, \lambda_1}\times (\bar{Z}_{\Fl,x})_{\infty \cdot x}^{\lambda_2})_{\mu}$ forms a smooth cover of $(\bar{Z}_{\Fl,x})_{\infty \cdot x}^{\lambda_2}$ by $r_1^{\lambda_1}$. The claim \eqref{17.6} is local in smooth topology, so we only need to prove that the $!$-pullback of the morphism (\ref{17.6}) to $(Z_{\Gr}^{\circ, \lambda_1}\times (\bar{Z}_{\Fl,x})_{\infty \cdot x}^{\lambda_2})_{\mu}$ along $r_1^{\lambda_1}$ is an isomorphism. 

By the same argument as \cite[Section 3.9]{[Ga3]}, we can see that the pullbacks of the morphism (\ref{17.6}) along $r_1^{\lambda_1}$ and $r_2^{\lambda_1}$ differ by a $lisse$ local system. To be more precise, by factorization property, the $!$-pullback of the restriction of \eqref{17.6} on $(\bar{Z}_{\Fl,x})_{\infty \cdot x}^{\lambda_1+\lambda_2}$ along $r_2^{\lambda_1}$ is given by the restriction of the external product of $\bar{\mathfrak{q}}_Z^{!}(\mathcal{F}_\emptyset)\mathop{\otimes}\limits^{!} \bar{\mathfrak{p}}_Z^{!} (j^-_{!, \glob, \Gr}[\dim \Bun_G])$ and \eqref{17.6}. On the other hand, the $!$-pullback of \eqref{17.6} along $r_1^{\lambda_1}$ is given by the restriction of the external product of the dualizing sheaf on $Z^{\circ,\lambda_1}_{\Gr}$ and \eqref{17.6}.


 Hence, we only need to prove that the pullback of $\eqref{17.6}$ along $r_2^{\lambda_1}$ is an isomorphism when restricted to $(Z_{\Gr}^{\circ, \lambda_1}\times (\bar{Z}_{\Fl,x})_{\infty \cdot x}^{\lambda_2})_{\mu}$.

By Corollary \ref{dual}, we know that our claim is true on \[(\bar{Z}_{\Fl,x})_{\infty \cdot x}^{s,\lambda_1+\lambda_2}:=(\bar{Z}_{\Fl,x})_{\infty \cdot x}^{s}\cap (\bar{Z}_{\Fl,x})_{\infty \cdot x}^{\lambda_1+\lambda_2}.\] Hence, the pullback of the morphism (\ref{17.6}) to the open subset $(r_2^{\lambda_1})^{-1}((\bar{Z}_{\Fl,x})_{\infty \cdot x}^{s,\lambda_1+\lambda_2})$ in $Z_{\Gr}^{\circ, \lambda_1}\times (\bar{Z}_{\Fl,x})_{\infty \cdot x}^{\lambda_2}$ is still an isomorphism. Now the claim follows from the fact that $(Z_{\Gr}^{\circ, \lambda_1}\times (\bar{Z}_{\Fl,x})_{\infty \cdot x}^{\lambda_2})_{\mu}$ is contained in $(r_2^{\lambda_1})^{-1}((\bar{Z}_{\Fl,x})_{\infty \cdot x}^{s,\lambda_1+\lambda_2})$ by our choice of $\lambda_1,\lambda_2$ and $\mu$.

So, we proved Theorem \ref{dualfunctor}.
\subsection{Proof of Proposition \ref{imstan glob}}\label{proof of imstan}
We define $\widetilde{\nabla}_{\lambda, \glob}$ to be the Verdier dual of $\Delta_{\lambda, \glob}$.
\begin{prop}\label{8.3.1}
There is an isomorphism
$$F_{\glob}^{DK}(\widetilde{\nabla}_{\lambda, \glob})\simeq \nabla_{\lambda, \Omega_{q}^{DK, '}}.$$
\end{prop}
\begin{proof}
In order to simplify the notation, we omit the twisting notation here.

According to Lemma \ref{thm2}, we have $F^{DK}\simeq F^{DK}_{\glob}\circ \pi_{\Fl,x}[d_g]$. Furthermore, since the Verdier duality functor commutes with $\pi_{\Fl,x}[d_g]$, we only need to prove that the image of $\widetilde{\nabla}_\lambda:= \mathbb{D}^{\verdier}(\Delta_\lambda)$ under the functor $F^{DK}$ is isomorphic to $\nabla_{\lambda, \Omega_{q}^{DK, '}}$.

Proposition \ref{LDK} asserts that the twisted sheaf $\widetilde{\nabla}_\lambda\simeq \mathbb{D}^{\verdier}(\Delta_\lambda)$ is isomorphic to $\Av_*^{ren}({{J}}_\lambda^{\mathbb{D}})$. By Corollary \ref{corep*H} , in order to show the proposition, it suffices to show $$H(\Fl_{G,x}^{\omega^\rho}, \Av_*^{ren}({{J}}_\lambda^{\mathbb{D}})\mathop{\otimes}\limits^! j_{*}(\omega_{S^{-, \mu}_{\Fl,x}})[\langle\mu, 2\check{\rho}\rangle] )=0$$ if $\lambda\neq \mu$, and $$H(\Fl_{G,x}^{\omega^\rho}, \Av_*^{ren}({{J}}_\lambda^{\mathbb{D}})\mathop{\otimes}\limits^! j_{*}(\omega_{S^{-, \mu}_{\Fl,x}})[\langle\mu, 2\check{\rho}\rangle] )=\mathsf{e}$$ if $\lambda=\mu$.

{

Note that $\Av_*^{ren}({{J}}_\lambda^{\mathbb{D}})$ is compact, so it is supported on finitely many $N(\cK)^{\omega^\rho}$-orbits in $\Fl_{G,x}^{\omega^\rho}$. Furthermore, the intersection $S_{\Fl,x}^{t^\lambda w}\cap S_{\Fl,x}^{-,\mu}$ is of finite type for any $t^\lambda w$ and $\mu$ (since it is the central fiber of a finite type scheme $(\Bun_N^{\omega^\rho})_{=\lambda\cdot x}^{w}\underset{\Bun_G'}{\times'} \Bun_{B^-}^{'',\mu}$ over $\mu\cdot x\in \Conf^\mu_{\leq \lambda\cdot x}$). This implies that there exists a very dominant $\eta$ such that $\supp(\Av_*^{ren}({{J}}_\lambda^{\mathbb{D}}))\cap t^{\eta}I^{\omega^\rho} t^{-\eta+\mu}I^{\omega^\rho}/I^{\omega^\rho}= \supp(\Av_*^{ren}({{J}}_\lambda^{\mathbb{D}}))\cap S_{\Fl,x}^{-,\mu}.$ Let $t^\eta\Fl^{{-\eta+\mu}}:=t^{\eta}I^{\omega^\rho} t^{-\eta+\mu}I^{\omega^\rho}/I^{\omega^\rho}$, and denote $j_*(\omega_{t^\eta\Fl^{{-\eta+\mu}}})$ as the $*$-extension of the twisted dualizing sheaf on $t^\eta\Fl^{{-\eta+\mu}}$.
}

Since $\Av_*^{ren}({{J}}_\lambda^{\mathbb{D}})$ is $N(\mathcal{O})^{\omega^\rho}$-equivariant, we have

{
\begin{equation}
    \begin{split}
         &H(\Fl_{G,x}^{\omega^\rho}, \Av_*^{ren}({{J}}_\lambda^{\mathbb{D}})\mathop{\otimes}\limits^! j_{*}(\omega_{S^{-, \mu}_{\Fl,x}})[\langle\mu, 2\check{\rho}\rangle] )\\
          \simeq &H(\Fl_{G,x}^{\omega^\rho}, \Av_*^{ren}({{J}}_\lambda^{\mathbb{D}})\mathop{\otimes}\limits^! j_*(\omega_{t^\eta\Fl^{{-\eta+\mu}}})[\langle\mu, 2\check{\rho}\rangle] )\\
           \simeq& H(\Fl_{G,x}^{\omega^\rho}, \Av_*^{ren}({{J}}_\lambda^{\mathbb{D}})\mathop{\otimes}\limits^! \Av_*^{N(\mathcal{O})^{\omega^\rho}}(j_{*}(\omega_{t^\eta\Fl^{{-\eta+\mu}}})[\langle\mu, 2\check{\rho}\rangle]) )\\
            \simeq &H(\Fl_{G,x}^{\omega^\rho}, \Av_*^{ren}({{J}}_\lambda^{\mathbb{D}})\overset{!}{\otimes}  {{J}}_{\eta,*}\star {{J}}_{-\eta+\mu,*})\\
    \simeq &\mathcal{H}om_{\Shv_{\mathcal{G}^G}(\Fl_{G,x}^{\omega^\rho})}(\delta_0, \Av_*^{ren}({{J}}_\lambda^{\mathbb{D}})\star  {{J}}_{\eta-\mu,*}\star {{J}}_{-\eta,*})\\
    \simeq& \mathcal{H}om_{\Shv_{\mathcal{G}^G}(\Fl_{G,x}^{\omega^\rho})}({{J}}_{\eta,!}, \Av_*^{ren}({{J}}_\lambda^{\mathbb{D}})\star  {{J}}_{\eta-\mu,*})\\
    \simeq&\mathcal{H}om_{\Whit_q(\Fl_{G,x}^{\omega^\rho})}(\Av_!^{N(\mathcal{K})^{\omega^\rho},\chi}({{J}}_{\eta,!}), \Av_*^{ren}({{J}}_\lambda^{\mathbb{D}})\star  {{J}}_{\eta-\mu,*})\\
    \simeq&\mathcal{H}om_{\Whit_q(\Fl_{G,x}^{\omega^\rho})}(\Av_!^{N(\mathcal{K})^{\omega^\rho},\chi}({{J}}_{\eta,!}), \Av_*^{ren}({{J}}_{\lambda+\eta-\mu}^{\mathbb{D}}))\\
    \simeq& \mathcal{H}om_{\Whit_q(\Fl_{G,x}^{\omega^\rho})}(\Av_!^{N(\mathcal{K})^{\omega^\rho},\chi}(\delta_{\eta,!})[-\langle\eta,2\check{\rho}\rangle], \Av_*^{ren}(\delta_{\lambda+\eta-\mu})[\langle\lambda+\eta-\mu,2\check{\rho}\rangle]).
    \end{split}
\end{equation}

}

And the latter space is $0$ if $\lambda\neq\mu$ and is $\mathsf{e}$ if $\lambda=\mu$.

\end{proof}

Combining Proposition \ref{8.3.1} with \eqref{dualfact}, there is an isomorphism
$$\nabla_{\lambda, \Omega_{q^{-1}}^{DK, '}}\simeq \mathbb{D}^{\verdier}(\Delta_{\lambda, \Omega_q^{L}}).$$
Now Proposition \ref{imstan glob} follows directly.
\begin{proof}(of Proposition \ref{imstan glob})

According to Theorem  \ref{dualfunctor}, there is
\begin{align*}
    F_{\glob}^L(\Delta_{\lambda, \glob})\simeq&F_{\glob}^L(\mathbb{D}^{\verdier}(\widetilde{\nabla}_{\lambda, \glob}))\\
    \simeq& \mathbb{D}^{\verdier}F_{\glob}^{DK}(\widetilde{\nabla}_{\lambda, \glob})\\
    \simeq&\mathbb{D}^{\verdier}(\nabla_{\lambda, \Omega_{q^{-1}}^{DK,'}})\\
    \simeq& \Delta_{\lambda, \Omega_q^{L,'}}.
\end{align*}
\end{proof}

\appendix
    \section{Semi-infinite sheaves on affine flags}\label{app a}
{To be self-contained, we review the theory of semi-infinite sheaves developed in \cite{[Ga3]} and \cite{[Ga4]} and provide additional details. Our goal is to supply the necessary materials for the $!$-extension semi-infinite sheaves $j_!(\omega_{S^{-, \Conf}_{\Gr, \Conf}})$ and $j_!(\omega_{S^{-,w,\Conf_{\infty\cdot x}}_{\Fl, \Conf_{\infty\cdot x}}})$ in Section \ref{configuration gr and fl}.

To simplify the notation, we will consider the semi-infinite sheaves on $N(\cK)$-orbit, while in the main content, we consider semi-infinite sheaves on $N^-(\cK)$-orbit.
\subsection{Existence of $!$-extension semi-infinite sheaf}
Recall the prestack ${S}_{\Gr, \Ran}^0$ defined in Definition \ref{def 6.1.6}. Since $\kappa^{\check{\lambda}}$ is injective for any $\check{\lambda}$, the collection of maps \{$\kappa^{\check{\lambda}}$\} determines a $N^{\omega^\rho}$-reduction (in particular, a $B$-reduction) of $\cP_G$ at $x\in X$. The fiber product $S_{\Fl, \Ran_x}:={S}_{\Gr, \Ran}^0\underset{\pt/G}{\times} \pt/B$ admits a relative position map to $\pt/B\underset{\pt/G}{\times}{\pt/B}\simeq B\backslash G/B$. For any $w\in W$, let $S^w_{\Fl, \Ran_x}$ be the preimage of the Bruhat cell $B\backslash BwB/B$ under the above relative position map.

\subsubsection{}
Now, we define the $!$-extension of semi-infinite sheaves on $S^{w_0}_{\Fl,\Ran_x}$. Similar constructions work for $S^w_{\Fl,\Ran_x}$ and $S^0_{\Gr,\Ran}$ as well.

Denote $\overline{S}^{w_0}_{\Fl,\Ran_x}$ as the closure of ${S}^{w_0}_{\Fl,\Ran_x}$ in $\Fl^{\omega^\rho}_{G,\Ran_x}$, it is isomorphic to $\overline{S}^{0}_{\Gr,\Ran}\underset{\pt/G}{\times}\pt/B$. If $\lambda\in \Lambda^{\text{neg}}$, let $(\Conf^\lambda \times \Ran_x)^\subset \subset \Conf^\lambda \times \Ran_x$ be the subspace such that $(D,\mathcal{I})\in \Conf^\lambda \times \Ran_x$ belongs to $(\Conf^\lambda \times \Ran_x)^\subset$ if and only if $\supp(D)\subset \mathcal{I}$. 

Let $ \preccurlyeq$ be the semi-infinite Bruhat order. For any $t^\lambda w \preccurlyeq w_0$, we define 
\begin{equation}
    S^{t^{\lambda}w}_{\Fl, \Ran_x}\subset (\Conf^\lambda\times \Ran_x)^\subset \underset{\Ran_x}{\times} \Fl^{\omega^\rho}_{\Ran_x}
\end{equation}
as the sub-prestack such that the map $\kappa^{\check{\lambda}}$ induced by $\alpha$ extends to an injective map
\begin{equation}
    (\omega^{\frac{1}{2}})^{\langle \check{\lambda}, 2\rho\rangle}(-\langle\check{\lambda}, D\rangle)\longrightarrow \cV^{\check{\lambda}}_{\cP_G}
\end{equation} on $X$, and the relative position of the resulting $B$-reduction at $x$ and the Iwahori structure $\epsilon$ is $w$.

For any such $S^{t^{\lambda}w}_{\Fl, \Ran_x}$, the map
\begin{equation}\label{A.1.3}
  i^{t^\lambda w}:  S^{t^{\lambda}w}_{\Fl, \Ran_x}\longrightarrow (\Conf^\lambda\times \Ran_x)^\subset \underset{\Ran_x}{\times}\overline{S}^{w_0}_{\Fl, \Ran_x}\longrightarrow \overline{S}^{w_0}_{\Fl, \Ran_x}
\end{equation}
is a locally closed embedding, and $\{S^{t^{\lambda}w}_{\Fl, \Ran_x}, t^\lambda w\preccurlyeq w_0\}$ gives rise to a stratification of $\overline{S}^{w_0}_{\Fl, \Ran_x}$. 

The projection $p^{t^\lambda w}: S^{t^{\lambda}w}_{\Fl, \Ran_x}\to(\Conf^\lambda\times \Ran_x)^\subset $ has a section $s^{t^\lambda w}:(\Conf^\lambda\times \Ran_x)^\subset \to S^{t^{\lambda}w}_{\Fl, \Ran_x} $ which sends $(D, \cI)$ to $(D,\cI, \cP_G, \alpha, \epsilon)$, where $\cP_G= \omega^\rho(-D)\overset{T}{\times}G$, $\alpha$ is given by the identification of $\omega^\rho(-D)$ and $\omega^\rho$ on $X-\cI$, and $\epsilon$ is given by $\omega^\rho(-D)|_x\overset{T}{\times}wB$.

\begin{definition}
   We define the semi-infinite category $
        \SI^{\preccurlyeq w_0}_{q,\Fl,\Ran_x}:= \Shv_{\cG^G}(\overline{S}^{w_0}_{\Fl, \Ran_x})^{N(\cK)_{\Ran_x}^{\omega^\rho}}$ and $\SI^{=t^\lambda w}_{q,\Fl,\Ran_x}:= \Shv_{\cG^G}({S}^{t^\lambda w}_{\Fl, \Ran_x})^{N(\cK)_{\Ran_x}^{\omega^\rho}}.$
\end{definition}

One can check that the pullback $\cG^G$ along \eqref{A.1.3} is canonically identified with the pullback of $\cG^\Lambda$ along
\begin{equation}
    S^{t^{\lambda}w}_{\Fl, \Ran_x}\overset{p^{t^\lambda w}}{\longrightarrow} (\Conf^\lambda\times \Ran_x)^\subset \longrightarrow \Conf^\lambda.
\end{equation}
In particular, we define
\begin{equation}\label{A.1.5}
\begin{split}
     (s^{t^\lambda w})^!: \SI_{q,\Fl, \Ran_x}^{=t^\lambda w}\longrightarrow \Shv_{\cG^\Lambda}((\Conf^\lambda\times \Ran_x)^\subset),\\
     (p^{t^\lambda w})^!: \Shv_{\cG^\Lambda}((\Conf^\lambda\times \Ran_x)^\subset)\longrightarrow \SI_{q,\Fl, \Ran_x}^{=t^\lambda w}.
\end{split}
\end{equation}
Here, we use the observation that $(p^{t^\lambda w})^!:\Shv_{\cG^\Lambda}((\Conf^\lambda\times \Ran_x)^\subset)\longrightarrow \Shv_{\cG^G}({S}^{t^\lambda w}_{\Fl, \Ran_x})$ factors through the full subcategory $\SI_{q,\Fl, \Ran_x}^{=t^\lambda w}$.

Since $(\Conf^\lambda\times \Ran_x)^\subset \underset{\Ran_x}{\times}N(\cK)_{\Ran_x}^{\omega^\rho}$ acts transitively on $S^{t^{\lambda}w}_{\Fl, \Ran_x}$ and is (ind-pro-) unipotent, the functors in \eqref{A.1.5} are equivalences. Furthermore, we have 
\begin{lem}
    $(\SI_{q,\Fl, \Ran_x}^{=t^\lambda w})^{T(\cO)^{\omega^\rho}_{\Ran_x}}=0$ if $\lambda\notin \Lambda^\sharp$. Here, $\Lambda^\sharp$ denotes the kernel of the bilinear form $b$ (i.e., $b_\lambda$ is trivial).
\end{lem}
\begin{proof}
    By the factorization property, we only consider the point case $\Shv_{\cG^G}(S^{-,t^\lambda w}_{\Fl,x})^{N^-(\cK)^{\omega^\rho}T(\cO)^{\omega^\rho}}\simeq \Shv_{\cG^\Lambda|_{\lambda\cdot x}}(\pt)^{T(\cO)^{\omega^\rho}}$. According to \cite[Section 7.5]{[GL2]}, the $T(\cO)^{\omega^\rho}$-equivariance structure on the fiber $\cG^\Lambda|_{\lambda\cdot x}$ corresponds to the character $b_\lambda$. In particular, it is trivial only if $\lambda\in \Lambda^\sharp$. 
\end{proof}


\subsubsection{}
With the preparations above, we prove
\begin{prop}\label{Prop A 1.5}
    The left adjoint functor of $(i^{t^\lambda w})^!:\SI_{q,\Fl, \Ran_x}^{\preccurlyeq w_0}\longrightarrow \SI_{q,\Fl, \Ran_x}^{=t^\lambda w} $ is well-defined.
\end{prop}
\begin{proof}
By considering the dual category, it is equivalent to proving that $(i^{t^\lambda w})^*$ exists, which is further equivalent to the following
    \begin{enumerate}
        \item for any finite set $\fI$ with a distinguished point, the functor
        \begin{equation}
            (i^{t^\lambda w}_{\fI})^*: \SI_{q,\Fl, X^\fI_x}^{\preccurlyeq w_0}\longrightarrow \SI_{q,\Fl, X^\fI_x}^{=t^\lambda w} 
        \end{equation} exists, where $\SI_{q,\Fl, X^\fI_x}^{\preccurlyeq w_0}$ (resp. $\SI_{q,\Fl, X^\fI_x}^{=t^\lambda w}$) is the base change of $\SI_{q,\Fl, \Ran_x}^{\preccurlyeq w_0}$ (resp. $\SI_{q,\Fl, \Ran_x}^{=t^\lambda w}$) along $X^\fI_x\to \Ran_x$,\\
        \item for any surjection preserving the distinguished point $\phi: \fI\to \fJ$, denote $\Delta_\phi: X^\fJ_x\hookrightarrow X^\fI_x$ the corresponding diagonoal embedding. The natural transformation 
        \begin{equation}
            (i^{t^\lambda w}_{\fJ})^*\circ \Delta_\phi^!\longrightarrow \Delta_\phi^!\circ (i^{t^\lambda w}_{\fI})^*
        \end{equation} is an isomorphism.
    \end{enumerate}
    Once the above two points are proven, we obtain the desired functor by passing to the limit.

The category $\SI_{q,\Fl, X^\fI_x}^{\preccurlyeq w_0}$ admits a block decomposition according to different characters of $T$ indexed by $\Lambda/\Lambda^\sharp$. To show the existence of adjoint functor, it is sufficient to show in the block. We assume $\lambda\in \Lambda^\sharp$, and prove (i),(ii) for $T$-monodromic objects in $\SI_{q,\Fl, X^\fI_x}^{\preccurlyeq w_0}$. 

In this case, (i) and (ii) are corollaries of the Braden theorem in \cite[Theorem 3.1.6]{[DG]}.

Similar to the construction of ${S}^{t^\lambda w}_{\Fl, X^{\fI}_x}$, we can define ${S}^{-, t^\lambda w}_{\Fl, X^{\fI}_x}$. Let $s_{\fI}^{-,t^\lambda w}:(\Conf^\lambda\times X^{\fI}_x)^\subset \to S^{-,t^{\lambda}w}_{\Fl, X^{\fI}_x} , i_{\fI}^{-,t^\lambda w}: S^{-,t^{\lambda}w}_{\Fl, X^{\fI}_x}\cap \overline{S}^{w_0}_{\Fl, X^{\fI}_x} \to \overline{S}^{w_0}_{\Fl, X^{\fI}_x}, \text{and } p_{\fI}^{-,t^\lambda w}:S^{-,t^{\lambda}w}_{\Fl, X^{\fI}_x}\to (\Conf^\lambda\times X^{\fI}_x)^\subset$ denote the corresponding maps. 

Consider the $\BG_m$-action on the fiber of $\overline{S}^{w_0}_{\Fl, \Ran_x}$ via $\BG_m\overset{2\rho}{\longrightarrow} T\curvearrowright \overline{S}^{w_0}_{\Fl, \Ran_x}$. In our specific case, the Braden theorem says that the functors $(s_{\fI}^{t^\lambda w})^!\circ (i_{\fI}^{t^\lambda w})^*$ and $(s_{\fI}^{-,t^\lambda w})^*\circ (i_{\fI}^{-,t^\lambda w})^!$ are well-defined for $\BG_m$-monodromic D-modules on $\overline{S}^{w_0}_{\Fl, \Ran_x}$, and are canonically isomorphic. Additionally, $(s_{\fI}^{-,t^\lambda w})^*= (p_{\fI}^{-,t^\lambda w})_*$ for $\BG_m$-monodromic D-modules.

Since \eqref{A.1.5} are equivalences, we have $(i_{\fI}^{t^\lambda w})^*=(p_{\fI}^{t^\lambda w})^!\circ (s_{\fI}^{-,t^\lambda w})^*\circ (i_{\fI}^{-,t^\lambda w})^!=(p_{\fI}^{t^\lambda w})^!\circ (p_{\fI}^{-,t^\lambda w})_*\circ (i_{\fI}^{-,t^\lambda w})^!:  \SI_{q,\Fl, X^\fI_x}^{\preccurlyeq w_0}\longrightarrow \SI_{q,\Fl, X^\fI_x}^{=t^\lambda w}$. This implies (i) immediately. For (ii), we observe that the $!$-pullback and $*$-pushforward satisfy base-change, in particular $(p_{\fI}^{t^\lambda w})^!\circ (p_{\fI}^{-,t^\lambda w})_*\circ (i_{\fI}^{-,t^\lambda w})^!$ commutes with taking $!$-restriction to the diagonal.
\end{proof}

\subsection{Local-global comparison}
Recall the substack $(\Bun_N^{\omega^\rho})_{=\lambda\cdot x}^w$ of $\overline{(\Bun_N^{\omega^\rho})}'_{\infty\cdot x}$  in Section \ref{section 7.2.5}. In this section, we will focus on the $!$-extension of the constant D-module on  $(\Bun_N^{\omega^\rho})^{w_0}:=(\Bun_N^{\omega^\rho})_{=0\cdot x}^{w_0}$.

The restriction of $\cG^G$ to $(\Bun_N^{\omega^\rho})^{w_0}$ is canonically trivialized. Let $j_{!, glob, \Fl}^N$ be the $!$-extension of the twisted constant sheaf on $(\Bun_N^{\omega^\rho})^{w_0}$, it is well-defined since the constant sheaf is holonomic. In this section, we aim to prove
\begin{lem}\label{Lemma A.2.1}
Pulling-back along $\pi_{\Fl, \Ran_x}: \overline{S}_{\Fl,\Ran_x}^{w_0}\hookrightarrow (\overline{S}_{\Fl,\Ran_x}^{w_0})_{\infty\cdot x}\longrightarrow \overline{(\Bun_N^{\omega^\rho})}'_{\infty\cdot x}$ induces an isomorphism of semi-infinite sheaves
\begin{equation}\label{equation A.2.1}
     \pi_{\Fl,\Ran_x}^!(j_{!, glob, \Fl}^N)[d_g]\simeq j_!(\omega_{S^{w_0}_{\Fl, \Ran_x}}).
\end{equation}
\end{lem}
\subsubsection{}
Let $\Bun_B^w$ be the preimage of the Bruhat cell $B\backslash BwB/B$ under $\Bun_B':= \Bun_B\underset{\pt/G}{\times}\pt/B\longrightarrow B\backslash G/B$.

For any $t^\lambda w\preccurlyeq w_0$, we define 
\[(\Bun_N^{\omega^\rho})^{t^\lambda w}:= \Bun_B^w\underset{\Bun_T}{\times} \Conf^\lambda,\]
where the map $\Conf^\lambda\longrightarrow \Bun_T$ is given by $D\mapsto \omega^\rho(-D)$.

Let $i^{t^\lambda w}_{\glob}$ be the locally closed embedding
\[i^{t^\lambda w}_{\glob}:(\Bun_N^{\omega^\rho})^{t^\lambda w}\longrightarrow  \overline{(\Bun_N^{\omega^\rho})}'_{\infty\cdot x}\]
which sends $(\cP_B, \epsilon, D)$ to $(\cP_G, \{\kappa^{\check{\lambda}}\}, \epsilon)$. Here, $\cP_G= \cP_B\overset{B}{\times} G$, and $\kappa^{\check{\lambda}}: (\omega^{\frac{1}{2}})^{\langle \check{\lambda}, 2\rho\rangle}\longrightarrow \cV^{\check{\lambda}}_{\cP_G}$ is given by $(\omega^{\frac{1}{2}})^{\langle \check{\lambda}, 2\rho\rangle}\hookrightarrow (\omega^{\frac{1}{2}})^{\langle \check{\lambda}, 2\rho\rangle}(-\langle\check{\lambda}, D\rangle)\longrightarrow \cV^{\check{\lambda}}_{\cP_G}$. 

The map $i^{t^\lambda w}$ factors through $(\overline{\Bun}^{\omega^\rho}_N)^{w_0}:= \overline{\Bun}_N^{\omega^\rho}\underset{\pt/G}{\times}\pt/B$. The collection $\{(\Bun_N^{\omega^\rho})^{t^\lambda w}, t^\lambda w\preccurlyeq w_0\}$ gives rise to a stratification of $(\overline{\Bun}^{\omega^\rho}_N)^{w_0}$. 

Furthermore, the following lemmas follow from definitions. 
\begin{lem}\label{lem A 2.3}
    The diagram
   \[ \xymatrix{
    S^{t^\lambda w}_{\Fl, \Ran_x}\ar[r]^{i^{t^\lambda w}}\ar[d]^{\pi_{\Fl,\Ran_x}} &\overline{S}^{w_0}_{\Fl, \Ran_x}\ar[d]_{\pi_{\Fl, \Ran_x}}\\ (\Bun_N^{\omega^\rho})^{t^\lambda w}\ar[r]^{i^{t^\lambda w}_{\glob}}& (\overline{\Bun}^{\omega^\rho}_N)^{w_0}
    }\]
    is Cartesian.
\end{lem}
\begin{lem}\label{lem A 2.4}
    The morphism 
    \begin{equation}
        S^{t^\lambda w}_{\Fl, \Ran_x}\longrightarrow (\Conf^\lambda\times \Ran_x)^{\subset}\longrightarrow \Conf^\lambda
    \end{equation}
    is identified with
    \begin{equation}
        S^{t^\lambda w}_{\Fl, \Ran_x}\longrightarrow (\Bun_N^{\omega^\rho})^{t^\lambda w} \overset{p^{t^\lambda w}_{\glob}}{\longrightarrow} \Conf^\lambda.
    \end{equation}
    Here, $p^{t^\lambda w}_{\glob}: (\Bun_N^{\omega^\rho})^{t^\lambda w}=\Bun_B^w\underset{\Bun_T}{\times} \Conf^\lambda {\longrightarrow} \Conf^\lambda$ is the projection.
\end{lem}

Similar to Definition \ref{globdef}, if we erase the character $\chi$, we can define the global semi-infinite sheaf category $\SI_{q,\Fl, \glob}$ on $\overline{(\Bun_N^{\omega^\rho})}'_{\infty\cdot x}$. We denote by $\SI_{q,\Fl, \glob}^{\preccurlyeq w_0}$ and $\SI_{q, \Fl,\glob}^{=t^\lambda w}$ the corresponding categories on $(\overline{\Bun}^{\omega^\rho}_N)^{w_0}$ and $(\Bun_N^{\omega^\rho})^{t^\lambda w}$, respectively. 

Since the equivariance property is against a unipotent groupoid, the global semi-infinite sheaf category is a full subcategory of the category of D-modules. The pushforward and pullback functors for pl{a}in D-modules give rise to the corresponding functors for semi-infinite D-modules. That is to say, we have the following functors:
\begin{equation*}
    \begin{split}
        i^{t^\lambda w}_{\glob,!}: \SI_{q,\Fl, \glob}^{=t^\lambda w}\rightleftharpoons &\SI_{q,\Fl, \glob}^{\preccurlyeq w_0}: i^{t^\lambda w,!}_{\glob}\\
        i^{t^\lambda w,*}_{\glob}: \SI_{q,\Fl, \glob}^{\preccurlyeq w_0}\rightleftharpoons &\SI_{q,\Fl, \glob}^{=t^\lambda w}: i^{t^\lambda w}_{\glob,*}.
    \end{split}
\end{equation*}

One can check that the full subcategory $\SI_{q,\Fl, \glob}^{=t^\lambda w}\subset \Shv_{\cG^G}((\Bun_N^{\omega^\rho})^{t^\lambda w})$ coincides with the image of the fully faithful functor $p^{t^\lambda w,!}_{\glob}: \Shv_{\cG^\Lambda}(\Conf^\lambda)\longrightarrow \Shv_{\cG^G}((\Bun_N^{\omega^\rho})^{t^\lambda w})$. So, for any object $\cF\in \SI_{q,\Fl, \glob}^{\preccurlyeq w_0}$, its restriction to the strata $(\Bun_N^{\omega^\rho})^{t^\lambda w}$ is the $!$-pullback of a $\cG^\Lambda$-twisted sheaf on $\Conf^\Lambda$. Combined with Lemma \ref{lem A 2.3} and \ref{lem A 2.4}, $i^{t^\lambda w,!}\circ \pi^!_{\Fl, \Ran_x}(\cF)$ lies in the full subcategory $\Shv_{\cG^\Lambda}((\Conf^\lambda\times \Ran_x)^{\subset})\simeq \SI_{q, \Fl, \Ran_x}^{=t^\lambda w}$. Note that for any object $\cF$ in $\Shv_{\cG^G}(\overline{S}^{w_0}_{\Fl, \Ran_x})$, it belongs to $\SI^{\preccurlyeq w_0}_{q,\Fl,\Ran_x}$ if and only if $i^{t^\lambda w,!}(\cF)\in \SI_{q, \Fl, \Ran_x}^{=t^\lambda w}$ for any $t^\lambda w$. In particular, $\pi^!_{\Fl, \Ran_x}(\cF)\in \SI^{\preccurlyeq w_0}_{q,\Fl,\Ran_x}$. \footnote{However, the naive analogy of Lemma \ref{thm local glob} in the semi-infinite setting is not correct. }

In order to prove Lemma \ref{Lemma A.2.1}, we need to show for any $t^\lambda w\preccurlyeq w_0$, we have
\begin{equation}\label{eq A.2.4}
    (i^{t^\lambda w})^* \pi_{\Fl, \Ran_x}^! (j_{!, \glob, \Fl}^N)=0.
\end{equation}

It is sufficient to use the Braden theorem again. Recall that $(s^{t^\lambda w})^!$ in \eqref{A.1.5} is an equivalence, the equation \eqref{eq A.2.4} equals
\begin{equation}
    (s^{t^\lambda w})^!\circ (i^{t^\lambda w})^* \pi_{\Fl, \Ran_x}^! (j_{!, \glob, \Fl}^N)=0.
\end{equation}

Also, there is an action of $T$ on $(\overline{\Bun}^{\omega^\rho}_N)^{w_0}$ given by the adjoint action of $T$ on $N$, which is compatible with the $T$-action on the fiber of $\overline{S}_{\Fl,\Ran_x}^{w_0}$. In particular, since $j_{!,\glob,\Fl}^N$ is $T$-monodromic, the sheaf $\pi^!_{\Fl,\Ran_x}(j_{!,\glob,\Fl}^N)$ is $T$-monodromic. In particular, \eqref{eq A.2.4} is true if $\lambda\notin \Lambda^\sharp$.

If $\lambda\in \Lambda^\sharp$, using the Braden theorem, we have
\begin{equation}\label{eq A 2.6}
    \begin{split}
        (s^{t^\lambda w})^!\circ (i^{t^\lambda w})^*\circ \pi_{\Fl, \Ran_x}^! (j_{!, \glob, \Fl}^N)= (s^{-,t^\lambda w})^*\circ (i^{-,t^\lambda w})^!\circ \pi_{\Fl, \Ran_x}^! (j_{!, \glob, \Fl}^N)\\
        =(p^{-,t^\lambda w})_*\circ (i^{-,t^\lambda w})^! \circ\pi_{\Fl, \Ran_x}^! (j_{!, \glob, \Fl}^N).
    \end{split}
\end{equation}

Let $\Bun_{B^-}^{\lambda, w}$ be the algebraic substack of $\Bun_{B^-}'$ such that the degree of the induced $T$-bundle is $-\lambda+(2-2g)\rho$ and the relative position of the $B^-$-bundle and the Iwahori structure at $x$ is $w$. We denote by $(\overline{\Bun}_N^{\omega^\rho})^{ w_0}\underset{\Bun_G'}{\times'} \Bun_{B^-}^{\lambda, w}$ the sub-Zastava space of $(\bar{Z}_{\Fl,x})_{\infty\cdot x}$. It has a projection $v^{\lambda, w}_{\Fl,\glob}$ to $\Conf^\lambda$, and we denote by $s_{\glob}^{t^\lambda w}: \Conf^\lambda\to(\overline{\Bun}_N^{\omega^\rho})^{ w_0}\underset{\Bun_G'}{\times'} \Bun_{B^-}^{\lambda, w}$ its section.

\begin{lem}
   \[ \xymatrix{
    \overline{S}_{\Fl,\Ran_x}^{w_0}\cap S_{\Fl, \Ran_x}^{-, t^\lambda w}\ar[rd]^{p^{-, t^\lambda w}}\ar@{=}[d]&\\
    ((\overline{\Bun}_N^{\omega^\rho})^{ w_0}\underset{\Bun_G'}{\times'} \Bun_{B^-}^{\lambda, w})\underset{\Conf^\lambda}{\times} (\Conf^\lambda\times \Ran_x)^\subset \ar[r]_{pr^{\lambda}_{\Conf\times \Ran} }\ar[d]^{\Id\times pr^\lambda}& (\Conf^\lambda\times \Ran_x)^{\subset}\ar[d]^{pr^\lambda} \\
     (\overline{\Bun}_N^{\omega^\rho})^{ w_0}\underset{\Bun_G'}{\times'} \Bun_{B^-}^{\lambda, w}\ar[r]^{v^{\lambda, w}_{\Fl,\glob}}& \Conf^\lambda
    }\]
    The upper diagram is commutative and the lower diagram is Cartesian.

    Also, pulling-back the gerbe $\cG^G$ along 
    \[\overline{S}_{\Fl,\Ran_x}^{w_0}\cap S_{\Fl, \Ran_x}^{-, t^\lambda w}\longrightarrow  (\overline{\Bun}_N^{\omega^\rho})^{ w_0}\underset{\Bun_G'}{\times'} \Bun_{B^-}^{\lambda, w}\overset{\bar{\fq}_Z'}{\longrightarrow} (\overline{\Bun}_N^{\omega^\rho})^{ w_0}\] is
    canonically isomorphic to the pullback of the gerbe $\cG^\Lambda$ along
    \[\overline{S}_{\Fl,\Ran_x}^{w_0}\cap S_{\Fl, \Ran_x}^{-, t^\lambda w}\longrightarrow(\Conf^\lambda\times \Ran_x)^\subset\longrightarrow \Conf^\lambda.\]
\end{lem}

Using the above lemma, we obtain 
\begin{equation}
\begin{split}
     \eqref{eq A 2.6}=pr^{\lambda}_{\Conf\times \Ran,*}\circ (\Id\times pr^\lambda)^!\circ \bar{\fq}_Z'^!(j_{!,\glob,\Fl}^N)\\
     \underset{\text{Base change}}{=}
     (pr^\lambda)^!\circ v^{\lambda, w}_{\Fl,\glob,*}\circ \bar{\fq}_Z'^! (j_{!,\glob,\Fl}^N)\\
     \underset{\text{Braden Theorem}}{=}(pr^\lambda)^!\circ (s_{\glob}^{t^\lambda w})^* \circ  \bar{\fq}_Z'^! (j_{!,\glob,\Fl}^N).
\end{split}
\end{equation}

It is remained to show 
\begin{equation}\label{eq A 2.8}
     (s_{\glob}^{t^\lambda w})^* \circ  \bar{\fq}_Z'^! (j_{!,\glob,\Fl}^N)=0.
\end{equation}

In the affine Grassmannian case, it is well-known that the $!$-pullback from the Drinfeld compactification to the Zastava space sends $!$-extension (resp. IC) sheaf on $(\overline{\Bun}_N^{\omega^\rho})_{\infty\cdot x}$ to $!$-extension (resp. IC) sheaf on the Zastava space. The analogous result also holds in the affine flags case. That is to say, 
\begin{prop}\label{prop a 2.6}
    Up to a cohomological shift, there is
    \begin{equation}\label{eq A 2.9}
        \bar{\fq}_Z'^! (j_{!,\glob,\Fl}^N)\simeq j_!(c_{({\Bun}_N^{\omega^\rho})^{ w_0}\underset{\Bun_G'}{\times'} \Bun_{B^-}^{\lambda, w}}).
    \end{equation}
    Here, $j_!(c_{({\Bun}_N^{\omega^\rho})^{ w_0}\underset{\Bun_G'}{\times'} \Bun_{B^-}^{\lambda, w}})$ is the shifted $!$-extension of the (twisted) constant sheaf on the open Zastava space ${({\Bun}_N^{\omega^\rho})^{ w_0}\underset{\Bun_G'}{\times'} \Bun_{B^-}^{\lambda, w}}$.
\end{prop}
\begin{proof}
    The proof adapts a similar argument of Section \ref{step 1}-\ref{8.2.7}. Here, we sketch the proof.
    
  We fix a $\mu$, and let $\lambda$ be very anti-dominant. Taking projection defines a map \begin{equation}
      r_1: (({\Bun}_N^{\omega^\rho}\underset{\Bun_G}{\times'} \Bun_{B^-}^{\mu})\times((\overline{\Bun}_N^{\omega^\rho})^{ w_0}\underset{\Bun_G'}{\times'} \Bun_{B^-}^{\lambda, w}))_{disj}\longrightarrow (\overline{\Bun}_N^{\omega^\rho})^{ w_0}\times \Conf^\mu.
  \end{equation} The factorization structure gives another map, i.e., composing the factorization map
    \begin{equation}
    \begin{split}
         (({\Bun}_N^{\omega^\rho}\underset{\Bun_G}{\times'} \Bun_{B^-}^{\mu})\times((\overline{\Bun}_N^{\omega^\rho})^{ w_0}\underset{\Bun_G'}{\times'} \Bun_{B^-}^{\lambda, w}))_{disj}\\ \longrightarrow  (\overline{\Bun}_N^{\omega^\rho})^{ w_0}\underset{\Bun_G'}{\times'} \Bun_{B^-}^{\mu+\lambda, w}\underset{\Conf_{\infty\cdot x}^{\mu+\lambda}}{\times} (\Conf^\mu\times \Conf_{\infty\cdot x}^\lambda)_{disj}
    \end{split}
    \end{equation}
with the projection
\begin{equation}
   (\overline{\Bun}_N^{\omega^\rho})^{ w_0}\underset{\Bun_G'}{\times'} \Bun_{B^-}^{\mu+\lambda, w}\underset{\Conf_{\infty\cdot x}^{\mu+\lambda}}{\times} (\Conf^\mu\times \Conf_{\infty\cdot x}^\lambda)_{disj}\longrightarrow (\overline{\Bun}_N^{\omega^\rho})^{ w_0}\times \Conf^\mu,
\end{equation}
gives a map $r_2$. 

Furthermore, the images of both maps land in the open subspace $((\overline{\Bun}_N^{\omega^\rho})^{ w_0}\times \Conf^\mu)_{good}\subset (\overline{\Bun}_N^{\omega^\rho})^{ w_0}\times \Conf^\mu$, where we impose the condition that the generalized $N^{\omega^\rho}$-structure is genuine at the support of the point in $\Conf^\mu$.

Similar to Section \ref{global whittaker category}, one can define a $N(\cO)^{\omega^\rho}_{\Conf^\mu}$-bundle $((\overline{\Bun}_N^{\omega^\rho})^{ w_0}\times \Conf^\mu)^{level}_{good}$ on $((\overline{\Bun}_N^{\omega^\rho})^{ w_0}\times \Conf^\mu)_{good}$, and the action of $N(\cO)^{\omega^\rho}_{\Conf^\mu}$ extends to an action of $N(\cK)^{\omega^\rho}_{\Conf^\mu}$. One can check that further compositions of $r_1, r_2$ with the projection 
\begin{equation}\label{eq a 2.13}
 ((\overline{\Bun}_N^{\omega^\rho})^{ w_0}\times \Conf^\mu)_{good}\longrightarrow ((\overline{\Bun}_N^{\omega^\rho})^{ w_0}\times \Conf^\mu)^{level}_{good}/N'   
\end{equation}
are the same if $N'$ is a large enough sub pro-group of $N(\cK)^{\omega^\rho}_{\Conf^\mu}$ which contains $N(\cO)^{\omega^\rho}_{\Conf^\mu}$.

To prove \eqref{eq A 2.9}, since the composition of $r_1$ and the projection to $(\overline{\Bun}_N^{\omega^\rho})^{ w_0}\underset{\Bun_G'}{\times'} \Bun_{B^-}^{\lambda, w}$ is surjective, we need to prove that the $!$-pullback of the $!$-extension sheaf on $(\overline{\Bun}_N^{\omega^\rho})^{ w_0}$ along $(\overline{\Bun}_N^{\omega^\rho})^{ w_0}\times \Conf^\mu\longrightarrow (\overline{\Bun}_N^{\omega^\rho})^{ w_0}$ and $r_1$ is the $!$-extension sheaf on $(({\Bun}_N^{\omega^\rho}\underset{\Bun_G}{\times'} \Bun_{B^-}^{\mu})\times((\overline{\Bun}_N^{\omega^\rho})^{ w_0}\underset{\Bun_G'}{\times'} \Bun_{B^-}^{\lambda, w}))_{disj}$.

By the identification of composed maps $r_1, r_2$ with projection \eqref{eq a 2.13}, it is equivalent to proving that the $!$-pullback of the $!$-extension sheaf on $((\overline{\Bun}_N^{\omega^\rho})^{ w_0}\times \Conf^\mu)_{good}$ along $r_2$ is the $!$-extension sheaf on $(({\Bun}_N^{\omega^\rho}\underset{\Bun_G}{\times'} \Bun_{B^-}^{\mu})\times((\overline{\Bun}_N^{\omega^\rho})^{ w_0}\underset{\Bun_G'}{\times'} \Bun_{B^-}^{\lambda, w}))_{disj}$. 

Since the $!$-extension sheaf on $((\overline{\Bun}_N^{\omega^\rho})^{ w_0}\times \Conf^\mu)_{good}$ is the $!$-pullback of $j_{!,\glob,\Fl}^N$, we only need to show that the $!$-pullback of $j_{!,\glob,\Fl}^N$ along $((\overline{\Bun}_N^{\omega^\rho})^{ w_0}\times \Conf^\mu)_{good}\longrightarrow (\overline{\Bun}_N^{\omega^\rho})^{ w_0}$ and $r_2$ is the $!$-extension sheaf. Now, it follows that if $\lambda+\mu$ is anti-dominant enough, the composed map of $r_2$ and $((\overline{\Bun}_N^{\omega^\rho})^{ w_0}\times \Conf^\mu)_{good}\longrightarrow (\overline{\Bun}_N^{\omega^\rho})^{ w_0}$ is smooth. 
\end{proof}

Now, the desired isomorphism \eqref{eq A 2.8} follows immediately, since the image of $s^{t^\lambda w}_{\glob}$ lies in the complement of ${({\Bun}_N^{\omega^\rho})^{ w_0}\underset{\Bun_G'}{\times'} \Bun_{B^-}^{\lambda, w}}$.

\subsubsection{Comparison between $j_{!,\glob,\Fl}^N$ and $j_!(\omega_{S_{\Fl,x}^{w_0}}).$}
Let $\omega_{S^{w_0}_{\Fl}}$ be the twisted dualizing sheaf on $S^{w_0}_{\Fl}$ under the canonical trivialization of $\cG^G$ on  $S^{w_0}_{\Fl}$. In this section, we prove
\begin{lem}\label{lem A 2.8}
    $\pi_{\Fl,x}^!(j_{!,\glob,\Fl}^N)[d_g]\simeq j_!(\omega_{S^{w_0}_{\Fl}}).$
\end{lem}
\begin{proof}
Similar to the Ran case,  $\pi_{\Fl,x}^!(j_{!,\glob,\Fl}^N)$ is also $N(\cK)^{\omega^\rho}_{x}$-equivariant and $T$-equivariant. So, we can use the Braden theorem once again.

Recall the notation $j_{t^\lambda w,\Fl}: S_{\Fl}^{t^\lambda w}\longrightarrow \Fl_G^{\omega^\rho}$ in Definition \ref{ver st}, and let us denote by $i_x^{t^\lambda w}$ the closed embedding $\pt=\{t^\lambda w\}\hookrightarrow S_{\Fl,x}^{t^\lambda w}$. According to Proposition \ref{prop 5.2}, we only need to prove

\begin{equation}\label{A 2.12}
(i_x^{t^\lambda w})^!\circ j^*_{t^\lambda w,\Fl}\circ \pi_{\Fl,x}^!(j_{!,\glob,\Fl}^N)=0, 
\end{equation}
if $t^\lambda w\neq 1$. If $\lambda\notin \Lambda^\sharp$, it is true since semi-infinite sheaves on $S^{t^\lambda w}_{\Fl,x}$ is not $T$-monodromic (To be more precise, it has a different $T$-monodromy structure from $\pi_{\Fl,x}^!(j_{!,\glob,\Fl}^N)$).

Assume $\lambda\in \Lambda^\sharp$. By the Braden theorem, 
\begin{equation}\label{a 2.13}
    \eqref{A 2.12}= H(S^{-, t^\lambda w}_{\Fl}, \pi_{\Fl,x}^!(j_{!,\glob,\Fl}^N)|_{S^{-, t^\lambda w}_{\Fl}}).
\end{equation}
By base change and Proposition \eqref{prop a 2.6}, \eqref{a 2.13} is just the shifted $!$-fiber of the sheaf $v^{\lambda, w}_{\Fl,\glob,*}\circ j_!(c_{({\Bun}_N^{\omega^\rho})^{ w_0}\underset{\Bun_G'}{\times'} \Bun_{B^-}^{\lambda, w}})$ at $t^\lambda\cdot x\in \Conf^\lambda_x$. Using the Braden theorem again, we have
\begin{equation}
    v^{\lambda, w}_{\Fl,\glob,*}\circ j_!(c_{({\Bun}_N^{\omega^\rho})^{ w_0}\underset{\Bun_G'}{\times'} \Bun_{B^-}^{\lambda, w}})= (s_{\glob}^{t^\lambda w})^*\circ j_!(c_{({\Bun}_N^{\omega^\rho})^{ w_0}\underset{\Bun_G'}{\times'} \Bun_{B^-}^{\lambda, w}})=0.
\end{equation}
\end{proof}

As a combination of Lemma \ref{Lemma A.2.1} and Lemma \ref{lem A 2.8}, we obtain
\begin{cor}\label{cor A}
    The $!$-restriction of $j_!(\omega_{S_{\Fl,\Ran_x}^{w_0}})$ to $\Fl_{G,x}^{\omega^\rho}$ is isomorphic to $j_!(\omega_{S_{\Fl,x}^{w_0}})$.
\end{cor}

\section{Semi-infinite equivalence v.s Iwahori equivalence}
In the untwisted case, it is known that for a category $\cC$ with a strong action of $G(\cK)^{\omega^\rho}$, and any $\lambda,\mu\in \Lambda$, the following functors are equivalences:
\begin{equation}\label{B.0.1}
\begin{split}
\cC^{\Ad_{\lambda}I^{\omega^\rho}}\overset{\oblv}{\longrightarrow}\cC^{T(\cO)^{\omega^\rho}}\overset{\Av_*^{\Ad_{\mu}I^{\omega^\rho}/T(\cO)^{\omega^\rho}}}{\longrightarrow}\cC^{\Ad_{\mu} I^{\omega^\rho}}\\
        \cC^{I^{\omega^\rho}}\overset{\Av_!^{N^-(\cK)^{\omega^\rho}}}{\simeq} \cC^{N^-(\cK)^{\omega^\rho}T(\cO)^{\omega^\rho}}.
    \end{split}
\end{equation}

The first one is implicited in \cite[Lemma 8]{[AB]} and the second one is proved in \cite[Theorem 17.2.1, Corollary 17.2.3]{[Ras1]}.\footnote{The original statement of $\fp$-adic groups belongs to {W. Casselman, A. Borel, and} J. Bernstein {, cf. \cite[Lemma 4.7]{[Bor76]}}.} For self-completeness, we prove the metaplectic version of the above equivalence with a similar proof as in \cite{[Ras1]} and \cite[Proposition 5.2.2]{[Ga3]}.

\begin{prop}
    For a category $\cC$ with a strong action of $(\Shv_{\cG^G}(G(\cK)^{\omega^\rho}),\star)$, and any $\lambda,\mu\in \Lambda$, the functors in \eqref{B.0.1} are still equivalences.
\end{prop}

\begin{proof}
Let us first consider the first claim.

    The gerbe $\cG^G$ on $G(\cK)^{\omega^\rho}$ has a canonical trivialization on $T(\cO)^{\omega^\rho}$, which canonically extends to $\Ad_{\mu} I^{\omega^\rho}$ preserving the multiplication structure. We consider the  constant sheaf on $\Ad_{\mu} I^{\omega^\rho}$ under this trivialization, denoted by $c_{\Ad_{\mu} I^{\omega^\rho}}$. By definition, $\Av_*^{\Ad_{\mu}I^{\omega^\rho}/T(\cO)^{\omega^\rho}}(\cF)$ is given by $c_{\Ad_{\mu} I^{\omega^\rho}}\overset{T(\cO)^{\omega^\rho}}{\times} \cF$ for any $\Ad_\lambda I^{\omega^\rho}$-equivariant object $\cF\in \cC$.

  Choose a trivialization of $\cG^G$ at $t^{-\lambda}\in G(\cK)^{\omega^\rho}$, which determines a left transition functor $t^{-\lambda}\cdot -:\Shv_{\cG^G}(\Fl_G^{\omega^\rho})\to \Shv_{\cG^G}(\Fl_G^{\omega^\rho})$. Applying this functor to $\cF$, we obtain an $I^0$-equivariant $\cG^G$-twisted sheaf. However, note that conjugating the trivialization on $T(\cO)^{\omega^\rho}$ by $t^{-\lambda}$ will change the trivialization by a character sheaf $b_{-\lambda}$ on $T(\cO)^{\omega^\rho}$. So, $t^{-\lambda} \cF$ is $(I^{\omega^\rho}, b_{-\lambda})$-equivariant.

    Similarly, if we choose a trivialization of $\cG^G$ at $t^{-\mu}\in G(\cK)^{\omega^\rho}$, it determines a transition functor $t^{-\mu}\cdot-$. Let $c_{I^{\omega^\rho} t^{-\mu}}:= t^{-\mu}\cdot c_{\Ad_{\mu} I^{\omega^\rho}}$. We have 
    \begin{equation}\label{B.0.3}
        c_{\Ad_{\mu} I^{\omega^\rho}}\overset{T(\cO)^{\omega^\rho}}{\star} \cF\simeq t^\mu\cdot c_{I^{\omega^\rho} t^{-\mu}} \overset{T(\cO)^{\omega^\rho}}{\star} \cF\simeq t^\mu\cdot c_{I^{\omega^\rho} t^{-\mu}}\cdot t^{\lambda} \overset{T(\cO)^{\omega^\rho}}{\star} t^{-\lambda}\cdot \cF.
    \end{equation}

    Since $t^{-\lambda} \cF$ is $I^0$-equivariant, we can take right $I^0$-averaging of $c_{I^{\omega^\rho} t^{-\mu}} t^{\lambda}\in \Shv_{\cG^G}(G(\cK)^{\omega^\rho})$ before taking convolution. Up to a shift, it is isomorphic to the pullback of $(J_{-\mu+\lambda,*})_{-\lambda}$ along $G(\cK)^{\omega^\rho}\longrightarrow \widetilde{\Fl}$.
    
  In conclusion, we obtain that, up to a shift,
      $\Av_*^{\Ad_{\mu}I^{\omega^\rho}/T(\cO)^{\omega^\rho}}(\cF)$ is given by $ t^{\mu} (J_{-\mu+\lambda,*})_{-\lambda} \overset{I}{\star} t^{-\lambda}\cF$. The functor $t^{\mu} (J_{-\mu+\lambda,*})_{-\lambda} \overset{I}{\star} t^{-\lambda}\cdot-$ is an equivalence since transitions and convolution with twisted BMW sheaves are equivalences.

Now, we prove the second claim. 

First, we need to prove that $\Av_!^{N^-(\cK)^{\omega^\rho}}(\cF)$ is well-defined for $\cF$ in the image of $\cC^{I^{\omega^\rho}}\longrightarrow \cC$. We only need to prove $\Av_!^{\Ad_{\alpha}N^-(t\cO)^{\omega^\rho}}(\cF)$ exits for any dominant $\alpha$, and then $\Av_!^{N^-(\cK)^{\omega^\rho}}(\cF)= \colim \Av_!^{\Ad_{\alpha}N^-(t\cO)^{\omega^\rho}}(\cF)$. Since $\Ad_{\alpha}I^{\omega^\rho}=\Ad_{\alpha}N(\cO)^{\omega^\rho} \cdot T(\cO)^{\omega^\rho}\cdot \Ad_{\alpha}N^-(t\cO)^{\omega^\rho}$, and $\Ad_{\alpha}N(\cO)^{\omega^\rho}\subset I^{\omega^\rho}$, we obtain that for $\cF$ lies in the image of $\cC^{I^{\omega^\rho}}\longrightarrow \cC$, we have $\Av_!^{\Ad_{\alpha}N^-(t\cO)^{\omega^\rho}}(\cF)\simeq \Av_!^{\Ad_{\alpha}I^{\omega^\rho}/T(\cO)^{\omega^\rho}}(\cF)$. The latter exists and is the left adjoint functor of the equivalence functor $\cC^{\Ad_{\alpha} I^{\omega^\rho}}\overset{\oblv}{\longrightarrow}\cC^{T(\cO)^{\omega^\rho}}\overset{\Av_*^{I^{\omega^\rho}/T(\cO)^{\omega^\rho}}}{\longrightarrow}\cC^{I^{\omega^\rho}}$.

To be more precise, for $\cF$ lies in the image of $\cC^{I^{\omega^\rho}}\longrightarrow \cC$,
\[ \Av_!^{\Ad_{\alpha}N^-(t\cO)^{\omega^\rho}}(\cF)\simeq t^{\alpha} J_{-\alpha,!}{\star} \cF[\langle \alpha, 2\check{\rho}\rangle].\]
In particular, there is
\begin{equation}
    \begin{split}
        \Av_*^{N(\cO)^{\omega^\rho}}\circ \Av_!^{N^-(\cK)^{\omega^\rho}}(\cF)\simeq  \colim \Av_*^{N(\cO)^{\omega^\rho}} (t^{\alpha} J_{-\alpha,!}{\star} \cF[\langle \alpha, 2\check{\rho}\rangle])
        \simeq (J_{\alpha,*})_{-\alpha}\star J_{-\alpha,!}\star \cF
        \simeq \cF.
    \end{split}
\end{equation}

We note that the functor $\Av_!^{N^-(\cK)^{\omega^\rho}}:  \cC^{I^{\omega^\rho}}{\longrightarrow} \cC^{N^-(\cK)^{\omega^\rho}T(\cO)^{\omega^\rho}}$ is the left adjoint functor of $\Av_*^{N(\cO)^{\omega^\rho}}: \cC^{N^-(\cK)^{\omega^\rho}T(\cO)^{\omega^\rho}}\longrightarrow \cC^{I^{\omega^\rho}}$. So, it remains to show $\Av_*^{N(\cO)^{\omega^\rho}}$ is conservative, i.e., if  $\Av_*^{N(\cO)^{\omega^\rho}}(\cF)=0$ and $\cF$ is $N^-(\cK)^{\omega^\rho}T(\cO)^{\omega^\rho}$-equivariant, then $\Av_*^{N(\cO)^{\omega^\rho}}(\cF)=0$ implies $\cF=0$.

Indeed, since $N^-(t\cO)^{\omega^\rho}T(\cO)^{\omega^\rho}=\bigcap_{\alpha\in \Lambda^+} I^{\omega^\rho}\cap \Ad_{\alpha}I^{\omega^\rho}$, we have $\cF=\Av_*^{N^-(\cO)^{\omega^\rho}}(\cF)= \colim \Av_*^{I^{\omega^\rho}\cap \Ad_{\alpha}I^{\omega^\rho}/T(\cO)^{\omega^\rho}}(\cF)$. In particular, if $\cF\neq 0$, there exists a very dominant $\alpha$, such that \[\Av_*^{I^{\omega^\rho}\cap \Ad_{\alpha}I^{\omega^\rho}/T(\cO)^{\omega^\rho}}(\cF)\neq 0.\] Using the fact that $\cF$ is $N^-(\cK)^{\omega^\rho}T(\cO)^{\omega^\rho}$-equivariant, we have \[\Av_*^{I^{\omega^\rho}\cap \Ad_{\alpha}I^{\omega^\rho}/T(\cO)^{\omega^\rho}}(\cF)\simeq \Av_*^{\Ad_{\alpha}N(\cO)^{\omega^\rho}}(\cF).\] The latter is an $\Ad_{\alpha}I^{\omega^\rho}$-equivariant object. 

Now, note that
\[\Av_*^{N(\cO)^{\omega^\rho}}(\cF)\simeq \Av_*^{N(\cO)^{\omega^\rho}}\circ \Av_*^{\Ad_{\alpha}N(\cO)^{\omega^\rho}}(\cF),\]
the desired property $\Av_*^{N(\cO)^{\omega^\rho}}(\cF)\neq 0$ follows from $\Av_*^{\Ad_{\alpha}N(\cO)^{\omega^\rho}}(\cF)\neq 0$ and the fact that $\Av_*^{N(\cO)^{\omega^\rho}}: \cC^{\Ad_{\alpha}I^{\omega^\rho}}\longrightarrow \cC^{I^{\omega^\rho}}$ is an equivalence (the first claim of \eqref{B.0.1}).
\end{proof}

}
\subsection*{Acknowledgments}
This paper is a part of the author’s Ph.D. thesis which was defended in June 2020. The author thanks Dennis Gaitsgory for suggesting this subject and also deeply thanks him for his guidance. Without his help, the author could not finish this subject alone. The author thanks Sergey Lysenko for helpful discussions and comments on the paper.

The author also thanks Michael Finkelberg, Lin Chen, and  Yuchen Fu for their careful reading and detailed suggestions about writing and organizing this paper. Without their help, this article is far from being well-written. The author thanks Yifei Zhao for explaining his work \cite{[Zh]}.

{The author thanks the anonymous referee for pointing out mistakes and gaps in the former version and constructive comments, which make this paper more complete.}

The author also thanks Sam Raskin, Peng Zheng, Qiao Zhou, Zicheng Qian, Roman Travkin, Jonathan Wise, Justin Campbell, Lizao Ye, Simon Riche, Christoph Baerligea, and Mihai Pavel for their help. 





\end{document}